\documentclass[11pt,reqno]{article}

\usepackage[text={155mm,235mm},centering]{geometry}
\usepackage{enumerate}

\usepackage[mathscr]{eucal}
\usepackage[all]{xy}
\usepackage{mathrsfs}
\usepackage{authblk}
\usepackage{amsfonts}
\usepackage{amsmath}
\usepackage{amsthm}
\usepackage{tikz-cd}
\usepackage{amssymb,bm}
\usepackage{latexsym}
\usepackage{tabularx}
\usepackage{graphicx}
\usepackage{wrapfig}
\usepackage{pict2e}
\usepackage{amsmath}
\usepackage{pifont}

\usepackage{pdfpages}
\usepackage{mathtools}
\usepackage{tikz-cd}
\usepackage{enumitem}
\usepackage{xy}

\usepackage[colorlinks=true, allcolors=blue]{hyperref}

\usepackage{amsmath}
\usepackage{caption}
\usepackage{placeins}

\usepackage{amsmath}
\usepackage{amssymb}
\usepackage{tikz-cd}
\usepackage{placeins}
\usepackage{subcaption}

\usepackage{float}

\xyoption{poly}
\xyoption{arc}
\xyoption{knot}
\xyoption{curve}
\xyoption{arrow}
\xyoption{all}
\allowdisplaybreaks
\newtheorem{theorem}{Theorem}[section]

\newtheorem{proposition}[theorem]{Proposition}
\newtheorem{lemma}[theorem]{Lemma}
\newtheorem{construction}[theorem]{Construction}

\newtheorem{corollary}[theorem]{Corollary}

\theoremstyle{definition}

\newtheorem{definition}[theorem]{Definition}
\newtheorem{example}[theorem]{Example}
\newtheorem{notation}[theorem]{Notation}

\newtheorem{remark}[theorem]{Remark}
\newtheorem*{ntt}{Notations}
\newtheorem*{ack}{Acknowledgements}

\newcommand{\lab}{\textup{Col}}

\newcommand{\dq}{\overline{Q}}

\newcommand{\im}{\operatorname{im}}
\newcommand{\symm}{\mathrm{symm}}

\newcommand{\perm}{\mathrm{Perm}}
\newcommand{\id}{\mathrm{id}}
\newcommand{\diag}{\mathrm{diag}}

\newcommand{\na}{\mathbf{N}(Q)_{h,\hbar}}

\newcommand{\la}{\mathfrak{a}}

\usepackage{marginnote}
\let\oldmarginpar\marginpar
\renewcommand\marginpar[1]{\-\oldmarginpar[\raggedleft\footnotesize #1]
{\raggedright\footnotesize\color{red} #1}} 
\begin{document}

\title{Biquantization of the
necklace Lie bialgebra
\footnotetext{Email: xjchen@scu.edu.cn, huang.maozhou@outlook.com, liumeiliang@stu.scu.edu.cn,\\
zhangjun24@stu.scu.edu.cn}}

\author[1,2]{Xiaojun Chen}
\author[3]{Maozhou Huang}
\author[1,2]{Meiliang Liu}
\author[1,2]{Jun Zhang}

\renewcommand\Affilfont{\small}

\affil[1]{School of Mathematics, Sichuan University, Chengdu, Sichuan Province, 
610064 P.R. China}

\affil[2]{Department of Mathematics, New Uzbekistan University,
Tashkent, 100001 Uzbekistan}

\affil[3]{Mathematical Science Research Center, Chongqing University of Technology, Chongqing, 400054, P.R. China}

\date{}

\maketitle

\begin{abstract} 
For the double of a quiver, the works of Ginzburg, 
Bocklandt-Le Bruyn and Schedler show that its closed 
paths, called the necklaces, have a natural Lie bialgebra 
structure. Schedler also constructed,
in [Int. Math. Res. Notices, 2005 (12), 725-760], 
a Hopf algebra that 
quantizes this Lie bialgebra. In this paper, we pursue 
one more step in this direction by constructing its 
biquantization, in the sense of Turaev [Ann. Sci. \'Ecole 
Norm. Sup. (4) 24 (1991), no. 6, 635-704].

\noindent{\bf Keywords:} 
Necklace Lie bialgebra, Quantization,
Hopf algebra

\noindent{\bf MSC2020:} 16S38, 16T05
\end{abstract}

\setcounter{tocdepth}{2} 
\tableofcontents

\section{Introduction}\label{sect:Intro}

The purpose of this paper is
to construct a {\it biquantization}
of the necklace  Lie bialgebra
of quivers. Let
$Q$ be a quiver and $\overline Q$
be its double, which means
doubling the edges of $Q$
but with an opposite orientation.
It was proved by Ginzburg
\cite{Ginzburg}
and independently by
Bocklandt and Le Bruyn \cite{BocLeB}
that the cyclic paths of $\overline Q$,
called the necklaces,
have a Lie algebra structure. 
Later Schedler showed in
\cite{Schedler} that
the necklaces
also have a Lie cobracket, which is
a Lie cocycle, and
thus form a
Lie bialgebra. Moreover,
Schedler constructed
a Hopf algebra which
quantizes such a Lie bialgebra.
His construction is inspired
by Turaev's quantization of
the Lie bialgebra of loops (i.e., 
the homotopy classes of closed paths)
on Riemann surfaces
(see \cite{Turaev} with
some preceding work due to Goldman
\cite{Goldman}).

Strictly speaking, Schedler's
quantization is a 
\emph{coquantization}
of the Lie bialgebra in
the sense of Turaev \cite[\S7.3]{Turaev}
(see \cite[\S3.6]{Schedler} and \cite[\S2.3]{GinzSche} for more details).
More precisely,
given a Lie bialgebra 
$\mathfrak{g}$, its symmetric product 
$S(\mathfrak{g})$
has an induced Poisson bracket and co-Poisson
cobracket,
which are compatible and thus
form a bi-Poisson bialgebra.
We thus may deform the algebra
structure of $S(\mathfrak{g})$
in the direction of the Lie bracket
without changing the cobracket,
which is a quantization, or deform
the coalgebra structure
in the direction of the Lie
cobracket without changing
the bracket, which is a coquantization.
However, we may also consider the
two-parameter quantization
of a Lie bialgebra: one parameter
quantizes the Lie bracket and
the other quantizes the Lie cobracket,
which some compatibilities
that enhance the cocycle condition
of the Lie bialgebra.
This is exactly Turaev's biquantization
of a Lie bialgebra, and for the Lie bialgera
of loops on a Riemann surface,
such a biquantization is realized
as the skein algebra of links
on the thickened surface.

Since Turaev's and Schedler's constructions
are very similar to each other, it is
desirable to expect that Turaev's 
biquantization exists
for necklace Lie bialgebras
as well.

We also remark that, in a series of papers
\cite{EtiKaz1}-\cite{EtiKaz5},
Etingof and Kazhdan constructed
a quantization of (both finite and infinite
dimensional) Lie bialgebras, by applying
the Drinfeld associator on
the corresponding vector spaces.
Later in \cite{kassel2000biquantization},
Kassel and Turaev continued to construct
a biquantization of finite dimensional
Lie bialgebras. 
Since necklace Lie algebras
are generally infinite
dimensional, 
Kassel-Turaev's biquantization
does not directly apply in this case.
We hope to address this problem
somewhere else.

\subsection{Main result}

For a finite quiver $Q$, 
the set of necklaces is 
the commutator quotient space 
$(k\overline{Q})_\natural
:=k\overline{Q}/[k\overline{Q},k\overline{Q}]$,
where $k\overline{Q}$ is the path
algebra of $\overline{Q}$ over 
a field $k$ of characteristic zero.
The Lie bracket is obtained by removing 
the crossing edges of two necklaces and 
then re-connecting the rest, 
and the Lie 
cobracket is realized by cutting a necklace in its 
self-crossing edges and then reconnecting
the two components (see Definition
\ref{L} below for more details). 
Denote by $L$ the necklace Lie bialgebra
of $Q$.

According to 
Turaev \cite{Turaev},
there are two bialgebras 
$\varepsilon_\hbar(L_\hbar)$ and $V_h(L_h)$ that coquantizes and quantizes 
the symmetric product $S(L)$ of $L$. We
introduce the two canonical morphisms
\[
p_h : \mathbf{N}(Q)_{h,\hbar} \to V_h(L_h), 
\qquad   
p_{\hbar} : \mathbf{N}(Q)_{h,\hbar} \to \varepsilon_{\hbar}(L_\hbar),
\]
and prove that they 
fit into a commutative biquantization diagram, 
thereby realizing 
the simultaneous biquantization of $L$.

\begin{theorem}\label{main}
Suppose that $Q$ is a finite quiver.
Let $L:=(k\overline{ Q})_{\natural}$
be the commutator quotient
space of $k\overline Q$,
equipped with 
the necklace Lie bialgebra
structure.
Then there is a commutative diagram of  biquantization of $L$:
\begin{equation}\label{biquan}
\begin{tikzcd}
	\mathbf{N}(Q)_{h,\hbar} & {V_h(L_h)} \\
	{\varepsilon_\hbar(L_\hbar)} & S(L),
	\arrow["{p_h}", from=1-1, to=1-2]
	\arrow["{p_\hbar}"', from=1-1, to=2-1]
	\arrow["p", from=1-1, to=2-2]
	\arrow["{q_\hbar}", from=1-2, to=2-2]
	\arrow["{q_h}"', from=2-1, to=2-2]
\end{tikzcd}
\end{equation}
where the vertical maps are
the quantizations and the horizontal
maps are the coquantization maps respectively.
\end{theorem}

The above maps are motivated by and 
hence are analogous
to those of Turaev \cite{Turaev}
in the case of loops on Riemann surfaces;
however, it is technically more
involved (see \S\ref{quancon}
and Remark \ref{rem:why}
for more details).

In the above theorem,
if we set
$h=1$, then $q_{\hbar}\circ p_h$
coincides with the one of Schedler;
however, the construction of
the left-bottom corner
of \eqref{biquan}
is, as far as we know, new.
If we set $h=\hbar,$ then $p$ is the quantization of the Lie bialgebra $L$ in the sense of Eingof and Kazhdan \cite[\S3]{EtiKaz1}.

In the sequel to this paper,
we plan to do the following
two things: (1) comparing
our construction with
that of Kazhdan-Etingof and Kassel-Turaev;
(2) applying the above
biquantization to some
specific quivers,
such as the Jordan quiver (see Example
\ref{ex:Jordan} below for some preliminary
computations),
and studying the quantum geometry of
the associated quiver
varieties.

\subsection{Organization of the paper}
The paper is organized as follows. 
Section
\ref{biquanoflie} is a review of the necessary 
background on the necklace Lie 
bialgebra, and provides definition for the 
related quantization,  
coquantization bialgebras $V_h(L_h)$ and 
$\varepsilon_{\hbar}(L_\hbar)$. 
Next, the two-parameter algebra $\mathbf{N}
(Q)_{h,\hbar}$ is constructed in 
Sections \ref{algebrana}, 
and its coalgebra structure is given so that 
it is a bialgebra.  
In Sections \ref{quancon} and
\ref{coquancon}, we construct two morphisms 
$p_{\hbar}:\na\to 
\varepsilon_\hbar(L_\hbar)$ and $p_h:\na\to V_h(L_h)$ 
serving respectively as a 
quantization of $\varepsilon_\hbar(L_\hbar)$ and a 
coquantization of $V_h(L_h)$. In 
Section \ref{proofofmain} we synthesize the 
preceding contents and prove the above theorem.

\begin{ntt}
Throughout this paper, $k$ denotes a field of characteristic 0. 
Tensor products, tensor algebras, and symmetric algebras are taken over the stated coefficient ring; when none is specified, they are taken over $k.$
$Q$ denotes a finite 
quiver, $Q_0$ 
denotes its vertex set, and $Q_1$ denotes its
arrow set. For an edge $e\in Q_1$,
$e_s$ and $e_t$ denote its
source and target in $Q_0$ respectively.
For a quiver $Q$,
$\overline{Q}$ denotes its double,
which duplicates the arrows 
of $Q$ but with
an opposite direction.
The set of natural numbers 
$\mathbb N$ starts from 1.
For a set $S$, let $|S|$ denote the cardinality of 
$S$.
\end{ntt}  

\begin{ack}
We would like
to thank Farkhod Eshmatov
for his support during the preparation of the paper.
The last two authors thank New Uzbekistan University for inviting them and offering excellent studying conditions.
We are grateful to an anonymous referee for helpful and constructive comments on an earlier version of this manuscript.
This work is supported by NSFC No. 12271377 and 12261131498.
\end{ack}

\section{Lie bialgebras and their
quantizations}\label{biquanoflie}

In this section, recall some
basic concepts on quantizations.
In Section \ref{Liebialg}, we sketch the necklace 
Lie bialgebra structure of a quiver. 
In Section \ref{defofbiquant}, we introduce the 
notion of biquantization of a Lie bialgebra. 
Finally, in
Section \ref{bipoisson} we introduce
the easy part of biquantization of a Lie 
bialgebra, namely, its quantization 
and coquantization.

\subsection{Lie bialgebra and its
symmetric products}\label{Liebialg}

Recall the definition of Lie bialgebras.

\begin{definition}
A \textit{Lie coalgebra} over $k$ is a 
$k$-module $\mathfrak g$ equipped with a linear 
morphism
$\nu: \mathfrak g \to \mathfrak g\otimes 
\mathfrak g$ such that 
\begin{enumerate}[label = \rm{(}\arabic*),noitemsep,leftmargin=*,nosep]
\item (Skew-symmetry)
$\perm \circ \nu = - \nu$,\text{ and }
\item (Co-Jacobi identity) 
$(\tau^2+\tau+1)\circ(\id \otimes \nu)\circ \nu=0$,
\end{enumerate}
where $\mathrm{Perm}(a \otimes b) = b \otimes a$ and $\tau:a\otimes b\otimes c\mapsto c\otimes a\otimes b$  for all $a, b, c \in \mathfrak g.$ 
The morphism $\nu$ is 
called the
\textit{Lie cobracket} in $\mathfrak{g}$.
\end{definition}

 The Lie algebra structure, together with the Lie coalgebra structure, defines a Lie bialgebra structure if the cocycle condition is satisfied.

\begin{definition}
[{Drinfeld \cite[Section 3]{Drinfeld}}]
Suppose that  $\mathfrak g$ is a $k$-module.
Then $\mathfrak g $ is called a \textit{Lie bialgebra} over $k$ if it is provided with a Lie bracket $[-,-]$ and a Lie cobracket $\nu$ so that for any $a,b\in \mathfrak g$ they satisfy the cocycle condition,
\[\nu([a,b])=a\nu(b)+\nu(a)b,\]
where the action of $a\in \mathfrak g$ on $b\otimes c \in \mathfrak g\otimes \mathfrak g$ is given by $a(b\otimes c)=[a,b]\otimes c+b\otimes [a,c]$.
\end{definition}

\subsubsection{Example: the necklace Lie bialgebra}\label{sec:necklaceLiealgebra}

Let us recall the construction of
necklace Lie bialgebra.

Suppose that $\overline Q$ is the double quiver of a 
finite quiver $Q$. 
Let \[(k\dq)_\natural\coloneqq\frac{k\dq}{[k\dq,k\dq]}\] 
denote the commutator
quotient space of $k\dq$,
whose elements are called
the necklaces.
(In the following, for a cyclic path
$a_{1}a_{2}\cdots a_{k}$ we also 
use it to denote its equivalence class, if
the context is clear.)

\begin{theorem}[\text{Bocklandt-Le Bruyn \cite{BocLeB},
Ginzburg \cite{Ginzburg} and 
Schedler \cite{Schedler}}]\label{L}
For a finite quiver $Q$,
let $(k\dq)_\natural$ be as above.

\begin{enumerate}[label = \rm{(}\arabic*),noitemsep,leftmargin=*,nosep]
\item For two necklaces
$a_1a_2\cdots a_k, b_1b_2\cdots b_l\in(k\dq)_\natural$, 
let
\[\begin{aligned}
&\{a_{1}a_{2}\cdots a_{k},b_{1}b_{2}\cdots b_{l}\}\\
&:=
\sum_{1\leqslant i \leqslant k,\, 1\leqslant j 
\leqslant l}
\langle a_{i},b_{j}\rangle (a_i)_t
a_{i+1}a_{i+2}\cdots a_{k}a_{1}
\cdots a_{i-1} b_{j+1}\cdots b_{l}b_{1}\cdots b_{j-1},
\end{aligned}\]
where the pair $\langle-,-\rangle$ is given by 
\[
\langle a_i,b_j\rangle=
\begin{cases}
 1, & \text{if } a_i\in Q_1 \text{ and } b_j=a_i^*,\\
-1, & \text{if } b_j\in Q_1 \text{ and } a_i=b_j^*,\\
 0, & \text{otherwise},
\end{cases}
\qquad a_i,b_j\in\overline{Q}_1.
\]
and 
$a_i^*$ denotes the reversed edge corresponding to $a_i$ in the double quiver $\dq$.
Then $((k\dq)_\natural, \{-,-\})$
is a Lie algebra.
\item 
Let $\nu:(k\dq)_\natural \to (k\dq)_\natural\otimes (k\dq)_\natural$ be
\begin{align}\nu(a_1 a_2\cdots a_k) 
:= \sum_{1 \leq i <j \leq k} 
\langle a_i, a_j\rangle (a_j)_t 
a_{j+1} \cdots a_k a_1 \cdots a_{i-1} 
\wedge 
(a_i)_ta_{i+1} \cdots a_{j-1},
\label{formula:cobracket}
\end{align}
for any $a_1a_2\cdots a_k$,
where $P \wedge Q \coloneqq P \otimes Q-Q\otimes P$ and $\nu(\overline{Q}_0)=\{0\}$. 
Then $((k\dq)_\natural, \nu)$
is a Lie coalgebra.

\item With the Lie bracket and Lie cobracket, then
$((k\dq)_\natural, \{-,-\},\nu)$
is a Lie bialgebra, called
the necklace Lie bialgebra.
\end{enumerate}
\end{theorem}

Thereafter, we use
$L$ to denote the necklace Lie bialgebra $(k\dq)_\natural.$

\subsubsection{The Poisson and co-Poisson structures}

Lie bialgebras naturally arise from Poisson-Lie groups
(see \cite{Drinfeld}).
We next describe some algebraic structures
on the symmetric algebra of a Lie bialgebra.

In the rest of this section, definitions also apply with the ground field $k$ replaced by a commutative unital $k$-algebra. In particular, we use them over $k,$ $k[h],$ and $k[\hbar].$

\begin{definition}\begin{enumerate}[label = \rm{(}\arabic*),noitemsep,leftmargin=*,nosep]
    \item[(1)]A \textit{Poisson algebra} is a commutative associative algebra equipped with a Lie bracket satisfying the \textit{Leibniz rule}.  

    \item[(2)]A \textit{co-Poisson coalgebra} $A$ over $k$ is a cocommutative coassociative
   $k$-coalgebra $A$ equipped with a $k$-linear Lie cobracket
   $\nu\colon A\to A\otimes_k A$ satisfying following \textit{co-Leibniz rule},
\[
(\mathrm{id} \otimes \Delta) \circ \nu = 
(\nu \otimes \mathrm{id} + (\mathrm{Perm} \otimes \mathrm{id}) \circ 
(\mathrm{id} \otimes \nu)) \circ \Delta,
\]
where $\Delta$ denotes the comultiplication of $A$. 
\end{enumerate}
\end{definition}

\begin{definition}
\begin{enumerate}[label = \rm{(}\arabic*),noitemsep,leftmargin=*,nosep]
\item[(1)]  A \textit{bialgebra} over $k$ is an
associative algebra and a coassociative coalgebra over $k$ where the comultiplication $\Delta$ is an algebra morphism.
\item[(2)]A \textit{Poisson bialgebra} $A$ over $k$ is a $k$-module equipped with the structure of bialgebra and Poisson algebra with the same commutative multiplication and the comultiplication preserving the Lie bracket:\[
\Delta([a,b])=[\Delta(a),\Delta(b)]
\]for all $a,b\in A$.
\item[(3)]A \textit{co-Poisson bialgebra} $A$ is a co-Poisson coalgebra with a bialgebra structure where the same cocommutative comultiplication $\Delta$ and the Lie cobracket $\nu$ satisfy \begin{equation}\label{copobialge}
    \nu(ab)=\nu(a)\Delta(b)+\Delta(a)\nu(b)
    \end{equation}for all $a,b\in A$. 
\item[(4)]A  \textit{bi-Poisson bialgebra} $S$ over $k$ is a $k$-module equipped with the structure of the Poisson bialgebra and of the co-Poisson bialgebra with the same bialgebra structure satisfying\[
\nu([a,b])=[\Delta(a),\nu(b)]+[\nu(a),\Delta(b)]
\]for all $a,b\in S$.
\end{enumerate}

\end{definition}

\begin{example}[{Turaev \cite[Theorem 7.4 and Section 17.1]{Turaev}}]\label{3.1}
Let $\mathfrak g$ be a Lie bialgebra and $S(\mathfrak{g})$ be the symmetric algebra of $\mathfrak g$.
Put $V_h(\mathfrak g_h)\coloneqq T(k[h]\otimes \mathfrak g)/(ab-ba-h[a,b])$.
There is a bi-Poisson bialgebra structure on $S(\mathfrak{g})$ and a co-Poisson bialgebra structure on $V_h(\mathfrak g_h)$:   
\begin{enumerate}
    \item [$-$] The bracket on $S(\mathfrak{g})$ is the extension of the Lie bracket on $\mathfrak{g}$ following the Leibniz rule.
    \item [$-$] The map $\Delta$ for both $V_h(\mathfrak g_h)$ and $S(\mathfrak g)$ is given by extending 
    \[\Delta:a\mapsto a\otimes 1+1\otimes a ,a\in \mathfrak g\] 
    following the multiplicativity of $\Delta$.
    \item [$-$] With the cobracket $\nu_\mathfrak{g}$ on $\mathfrak g$, the cobracket on $V_h(\mathfrak g_h)$ is the map $\nu$ extending $\nu(x)\coloneqq \nu_\mathfrak{g}(x)$ for $x\in\mathfrak g$ following formula (\ref{copobialge}). 
    With $h = 0$, this map descends to $S(\mathfrak g)$.
\end{enumerate}

\end{example}

\subsection{Biquantization of a Lie bialgebra}\label{defofbiquant}

We now discuss the quantization problem
of the various Poisson structures in
the previous subsection.
A good reference is Turaev's paper
\cite{Turaev}.

Let $R$ be a commutative associative $k$-algebra with augmentation $\varphi:R\to k$. 
Given an $R$-module $A$ and a $k$-module $S$, an additive map $p:A\to S$ is $\varphi$-\textit{linear} if $p(ra)=\varphi(r)p(a)$ for any $r\in R$ and $a\in A$.

\begin{definition}[Quantization of a Poisson algebra]\label{2.7}
With the above $R$ and $\varphi$, fix a formal parameter $h\in \ker \varphi$. 
Then a \textit{quantization} over $(R,\varphi,h)$ of a Poisson $k$-algebra $S$ is a pair 
\[\text{(an $R$-algebra $A$, a $\varphi$-linear ring epimorphism $p:A\to S$)}\] satisfying \begin{align}
ab-ba\equiv hp^{-1}([p(a),p(b)])\mod  h\ker p, \textup{ for any } a,b\in A
\label{formula:quantization}\end{align}Here, by the notation $p^{-1}$, we mean to choose a lift under $p$, and the resulting element is independent of the chosen lift modulo $\ker p$.
For brevity, the pair $(A,p)$ is also called a quantization over $R$ or over $\varphi$. 
The quantization $(A,p)$ is called \textit{reduced} if
\(
\ker p=(\ker\varphi)A.
\)
In particular, if $\ker\varphi=(h)$, this condition is equivalent
to $\ker p=h\cdot A$.
\end{definition}

\begin{definition}[Coquantization of a co-Poisson coalgebra]\label{def:coquan}
Fix a formal parameter $\hbar\in \ker \varphi$. 
A \textit{coquantization} over $(R, \varphi,\hbar)$ of a co-Poisson $k$-coalgebra $S$ is a pair 
\[\text{(an $R$-coalgebra $A$, a  $\varphi$-linear coalgebra epimorphism $p:A\to S$)}\] satisfying
    \[\Delta(a) - \mathrm{Perm}_{A}(\Delta(a)) \equiv \hbar(p\otimes p)^{-1}(\nu\circ p)(a) \mod (\hbar\,\ker (p\otimes p)).\] 
     If $\ker p=\ker\varphi \cdot A$, then the coquantization $(A,p)$ is said to be \textit{reduced}. Here, the notation $(p\otimes p)^{-1}$ is interpreted in a similar way to the above $p^{-1}$.
\end{definition}

There is a canonical way to produce a bi-Poisson bialgebra from a bialgebra:

\begin{theorem}
[{Turaev \cite[Theorem 16.2.4]{Turaev}}]

Let $A$ be a bialgebra over $k[h,\hbar]$ that is free as a $k[h,\hbar]$-module. 
For any $a,b\in A$, suppose that \[ab-ba\in hA\text{ and }\Delta(a)-\perm(\Delta(a))\in \hbar A^{\otimes 2}.\]
Put $S\coloneqq A/(h,\hbar)A$. 
Then $S$ is a bi-Poisson bialgebra over $k$ with Lie bracket $[-,-]$ and Lie cobracket $\nu$ 
given as follows:
\[\begin{array}{rcl}
[-,-]: & S\otimes S\to S; & [p(a),p(b)]=p(h^{-1}(ab-ba)), \\
\nu: & S\to S\otimes S; & \nu(p(a))=(p\otimes p)(\hbar^{-1}(\Delta(a)-\perm(\Delta(a)))).
\end{array}\]
\end{theorem}

\begin{definition}[Biquantization of bi-Poisson 
bialgebra]
Let $S$ be a bi-Poisson bialgebra over $k$. 
A \textit{biquantization} of $S$ over $k[h,\hbar]$
consists of a bialgebra $A$ over $k[h,\hbar]$,
a co-Poisson bialgebra $A_h$ over $k[h]$,
a Poisson bialgebra $A_\hbar$ over $k[\hbar]$,
and a commutative diagram of surjective bialgebra morphisms,
\[\begin{tikzcd}
	A & {A_h} \\
	{A_\hbar} & S
	\arrow["{p_h}", from=1-1, to=1-2]
	\arrow["{p_\hbar}"', from=1-1, to=2-1]
	\arrow["p", from=1-1, to=2-2]
	\arrow["{q_\hbar}", from=1-2, to=2-2]
	\arrow["{q_h}"', from=2-1, to=2-2]
\end{tikzcd}\]
satisfying the following two conditions:
\begin{enumerate}[label = \rm{(}\arabic*),noitemsep,leftmargin=*,nosep]
    \item[(1)]     
    $p_\hbar,q_\hbar$ and $p_h,q_h$ are respectively quantizations and coquantizations over diagram 
\[\begin{tikzcd}
	{k[h,\hbar]} & {k[h]} \\
	{k[\hbar]} & k.
	\arrow["{\hbar\,\mapsto \,0}", from=1-1, to=1-2]
	\arrow["{h\,\mapsto \,0}"', from=1-1, to=2-1]
	\arrow[from=1-1, to=2-2]
	\arrow["{h\,\mapsto\, 0}", from=1-2, to=2-2]
	\arrow["{\hbar\,\mapsto\, 0}"', from=2-1, to=2-2]
\end{tikzcd}\]

\item[(2)] $q_\hbar$ is a morphism of co-Poisson bialgebras and $q_h$ is a morphism of Poisson bialgebras.
\end{enumerate}

A biquantization is called \textit{reduced} 
if $\ker p_\hbar=hA,$ $\ker p_h=\hbar A$, and $\ker p=hA+\hbar A $. 
\end{definition}

For a Lie bialgebra $\mathfrak{g}$, 
Example~\ref{3.1} indicates that the symmetric 
algebra $S(\mathfrak{g})$ is a bi-Poisson 
bialgebra. Hence, we have the following.

\begin{definition}[{Turaev \cite[Section 17.1]{Turaev}}]
\label{biofg}
Let $\mathfrak{g}$ be a Lie bialgebra. 
A \textit{biquantization} of $\mathfrak{g}$ is a reduced 
biquantization of $S(\mathfrak{g})$.
\end{definition}

\subsection{Quantization and coquantization for $S(L)$}\label{bipoisson}
 
This subsection is devoted to constructing the quantization \( V_h(L_h) \) and the coquantization $\varepsilon_{\hbar}(L_\hbar)$ of $S(L)$. We
first have the following.

\begin{theorem}[{Turaev \cite[Theorem 1.4]{Turaev}}]\label{qhbar}With the notations in Example~\text{\rm{\ref{3.1}}}.
Let $q_\hbar$ be the map given by $h\mapsto 0$.
Then $V_h(\mathfrak g_h)$ is a reduced quantization of $S(\mathfrak g)$ given by $q_\hbar$.
\end{theorem}

\noindent Apply this theorem with $\mathfrak{g}_h = L_h:=k[h]\otimes L $. The obtained quantization $V_h(L_h)\to S(L)$ is the $q_\hbar$ in Diagram (\ref{biquan}).
We next turn to the coquantization. 

\begin{definition}
[{Spiral Lie coalgebra, Turaev \cite[Section 11.1]{Turaev}}]

Let $\mathfrak g$ be a Lie coalgebra over $k$ with cobracket $\nu$. Set inductively $\nu^1:=\nu$ and \[
\nu^n:=(\id^{n-1}\otimes \nu)\circ \nu^{n-1}:\mathfrak g\to \mathfrak g^{\otimes (n+1)}.
\] 
Then $\mathfrak g$ is \textit{spiral} if $\mathfrak g=\bigcup_{n\geq1}\ker \nu^n$. 
A Lie bialgebra is \textit{spiral} if it is spiral as a Lie coalgebra.
\end{definition}

\begin{example}\label{epsilon}
The necklace Lie bialgebra $L$ is spiral.
Indeed, if we 
write $\nu(\alpha)=\sum \alpha_1\otimes 
\alpha_2$ for $\alpha
\in L$, then the 
sum of the lengths of 
$\alpha_1$ and $\alpha_2$ is always smaller than 
the length of 
$\alpha$ by (\ref{formula:cobracket}). 
Hence for any $\alpha\in L$, we have 
$\nu^N(\alpha)=0$ for 
$N$ big enough. 
\end{example}

Let $(R,\varphi,\hbar)$ be a triple as in 
Definition~\ref{def:coquan}.
Given a spiral Lie coalgebra $\mathfrak g$ over 
$(R,\varphi,\hbar)$, following Turaev \cite[Section 11.2]{Turaev}, 
the modified Campbell-Hausdorff series of $\log(e^xe^y)$ for any 
elements $x,y\in \mathfrak g^*$ writes 
\[
\mu_{\hbar}(x,y)=x+y+
\frac{\hbar}{2}[x,y]+\frac{\hbar^2}{12}([x,[x,y]]
+[y,[y,x]])+\cdots.\] 
This defines 
$\mathfrak g^*\times \mathfrak g^*\to \mathfrak g^*; 
(x,y)\mapsto \mu_{\hbar}(x,y).$
Note that one may identify $S(\mathfrak{g})$ 
(resp. $S(\mathfrak g)\otimes S(\mathfrak g)$) 
with the set of polynomial functions on 
$\mathfrak{g}^{*}$ 
(resp. $\mathfrak g^*\times\mathfrak g^*$).
The precomposition 
$a\mapsto a\circ \mu_{\hbar}$ 
defines a map $\nabla_{\hbar}: 
S(\mathfrak g)\to S(\mathfrak g)
\otimes S(\mathfrak g),$ which 
turns out to be a coproduct. 
For every $a\in\mathfrak g$, this coproduct satisfies:
\[\nabla_\hbar(a)
\equiv
a\otimes1+1\otimes a+\frac{\hbar}{2}\nu (a)\mod 
\bigoplus_{i,j\geq1,i+j\geq 3}\hbar^2(S^i(\mathfrak g)\otimes S^j(\mathfrak g)).\]

\begin{definition}\label{def:epsilon(g)}
Let $\mathfrak{g}$ be a spiral Lie coalgebra.
Denote by $\varepsilon_\hbar(\mathfrak g)$ 
the bialgebra 
$S(\mathfrak g)$ with the coproduct $\nabla_{\hbar}$.   
\end{definition}

\begin{theorem}
[{Turaev \cite[Theorem 11.4]{Turaev}}]
Let $\mathfrak g$ be a spiral Lie bialgebra over $k[\hbar]$,
free as a $k[\hbar]$-module. There exists a unique
$k[\hbar]$-bilinear Lie bracket $\{-,-\}$ on
\(
\varepsilon_\hbar(\mathfrak g)
=\bigl(S(\mathfrak g),\nabla_\hbar\bigr)
\)
such that $\varepsilon_\hbar(\mathfrak g)$ is a Poisson
bialgebra over $k[\hbar]$ and, for all $a,b\in\mathfrak g\subset 
\varepsilon_\hbar(\mathfrak g)$:
\[
\{a,b\}\equiv[a,b]
\pmod{\hbar\bigoplus_{n\geq 2}S^n(\mathfrak g)}.
\]
\end{theorem}

For $L$ the necklace Lie algebra, 
put $L_\hbar\coloneqq k[\hbar]\otimes_k L$. Its bracket is given by
\[[f(\hbar)\otimes\alpha,g(\hbar)\otimes\beta]
\coloneqq f(\hbar)g(\hbar)\otimes\{\alpha,\beta\},\text{ 
for any $\alpha,\beta\in L$ and $f(\hbar),g(\hbar)\in 
k[\hbar]$}\]
with $\{-,-\}$ in Theorem~\ref{L}. The cobracket on $L_\hbar$ is obtained from that on $L$
by extension of scalars.
With this theorem, it is straightforward to check the following (see also Turaev \cite[\S11.5]{Turaev}).

\begin{corollary}\label{qh}
For $L_\hbar$,
let $\varepsilon_\hbar(L_\hbar)$ be given by
Definition \ref{def:epsilon(g)}.
Then the map 
$q_h:\varepsilon_\hbar(L_\hbar)\to S(L)$ induced by the
augmentation map $k[\hbar] \to k$ is a morphism of Poisson bialgebras and a reduced coquantization. 
\end{corollary}

\section{The bialgebra $\na$}\label{algebrana}

In this section, we construct the bialgebra
$\na$ in Theorem~\ref{main}.
It is almost the same as the
one constructed by 
Schedler in \cite{Schedler},
and the only difference is that
here our algebra $\na$ is over $k[h,\hbar]$
while his is over $k[\hbar]$.
The bialgebra $\na$ is a refinement of Schedler's Hopf algebra satisfying in the sense of Drinfeld-Turaev's biquantization two-parameters framework.
Since it is the key notion of this paper, we
give full details of its construction.
We ask for the reader's forgiveness for duplication.

In Section~\ref{construct na}, we construct the 
algebra $\na$.
In Section~\ref{sec:coprodna}, we construct a 
coproduct structure on $\na$.
In Section~\ref{sec:bialg na}, we prove that this 
coproduct together with the natural product gives 
$\na$ a bialgebra structure.
The main result of this section is Theorem \ref{bi}.

\subsection{The algebra \(\na\)}\label{construct na}

We start with some notations 
(cf. Schedler \cite[Section~3.1]{Schedler}). 
Put $\mathsf{AH} := \overline{Q} \times \mathbb{N}$. 
Each element is an arrow with a height, i.e., an 
arrow in the double quiver $\overline{Q}$ equipped 
with a ``height'' parameter in $\mathbb{N}$.
Let $E_{\overline{Q}}$ denote the free $k$-vector 
space spanned by $\mathsf{AH}$ and $E_0=k^{|Q_0|}$ 
the ring generated by the vertex set $Q_0$. 
Then let $T_{E_0}(E_{\overline{Q}})$ be the tensor 
algebra such that the product of two paths is the 
composition if the first path ends at the start 
vertex (in $E_0$) of the second path, or otherwise 
zero. 
Put $\mathsf{LH} :=T_{E_0} E_{\overline{Q}}/[T_{E_0} 
E_{\overline{Q}},T_{E_0} E_{\overline{Q}}]$, the 
space of cyclic paths with heights.
Let $\mathsf{SLH}$ be the symmetric algebra of 
$\mathsf{LH}$.

Consider an element $X\in \mathsf{SLH}$ as below such that all involved heights $\lambda_{i,j}$ are all distinct \begin{equation}\label{form}
 \begin{split}
X:=(a_{1,1},\lambda_{1,1})\cdots (a_{1,l_{1}},\lambda_{1,l_{1}}) \& (a_{2,1},\lambda_{2,1})\cdots (a_{2,l_{2}},\lambda_{2,l_{2}})
\\ \& \cdots  \& 
(a_{k,1},\lambda_{k,1})\cdots (a_{k,l_{k}},\lambda_{k,l_{k}}) 
\&v_{1} \& v_{2} \&\cdots \& v_{m},
\end{split}
\end{equation}
where $a_{i,j}\in \overline{Q}$, $v_{i}\in Q_{0}$, $a_{i,1}\cdots a_{i,l_i}$ is a cyclic path in $k\dq$, and ``$\&$" denotes the symmetric product in $\mathsf{SLH}$.
An element $X$ of the form (\ref{form}) is called a $\textit{link}$. 
If $k=1$ and $m=0$, then $X$ is said to be a $\textit{knot}$. 

Let $\mathsf {SLH}'[h,\hbar]$ denote free $k[h,\hbar]$-module generated by $X$ having the form $(\ref{form})$.
Put $\tilde{A}[h,\hbar]$ to be the quotient of $\mathsf {SLH}'[h,\hbar]$ obtained by identifying any element $X$ of the form $(\ref{form})$  with the one obtained from $X$ by replacing the $ \lambda_{i,j}$ with $ \lambda'_{i,j}$ preserving order: that is, $ \lambda_{i,j}< \lambda_{i',j'}$ if and only if $ \lambda'_{i,j}< \lambda'_{i',j'}$ .
 
Consider the $k[h,\hbar]$-submodule $\tilde{B}[h,\hbar]$ of $\tilde{A}[h,\hbar]$ spanned by all elements of the following forms,
\begin{eqnarray} 
&& X-X'_{i,j,i',j'} -hX''_{i,j,i',j'}, \nonumber\\
&&\quad\textup{where}\  i \neq i',\lambda_{i,j}<\lambda_{i',j'}\text{, and there is no }(i'',j'')\text{ such that }\lambda_{i,j} < \lambda_{i'',j''} < \lambda_{i',j'};\label{quiver skein relations 1}\\
&&X-X'_{i,j,i,j'} - \hbar X''_{i,j,i,j'}, \nonumber\\ 
&&\quad\textup{where}\  \lambda_{i,j}<\lambda_{i,j'}\text{ and there is no }(i'',j'')\text{ such that }\lambda_{i,j} < \lambda_{i'',j''} < \lambda_{i,j'}. \label{quiver skein relations 2}
\end{eqnarray}
The summands above are defined as below:
\begin{itemize}
\item[$-$] $X'_{i,j,i',j'}$ is the same as 
$X$ except for that the heights $\lambda_{i,j}$ and $\lambda_{i',j'}$ are
interchanged.
$X''_{i,j,i',j'}$ is given by replacing the two 
components \[(a_{i,1},\lambda_{i,1})\cdots 
(a_{i,l_{i}},\lambda_{i,l_{i}})\text{ and }(a_{i',1},\lambda_{i',1})\cdots 
(a_{i',l_{i'}},\lambda_{i',l_{i'}})\]
with the single component
\begin{displaymath}
\langle a_{i,j},a_{i',j'}\rangle 
(a_{i,j+1},\lambda_{i,j+1})\cdots (a_{i,j-1},\lambda_{i,j-1})
(a_{i',j'+1},\lambda_{i',j'+1})\cdots 
(a_{i',j'-1},\lambda_{i',j'-1}).
\end{displaymath}

\item[$-$] $X'_{i,j,i,j'}$ is the same as 
$X'_{i,j,i',j'}$ except for that the heights 
$\lambda_{i,j}\textup{ and }\lambda_{i,j'}$ are interchanged.
$X''_{i,j,i,j'}$ is given by replacing the 
component $(a_{i,1},\lambda_{i,1})\cdots 
(a_{i,l_{i}},\lambda_{i,l_{i}})$ with the following two components
\begin{align*}
\langle a_{i,j},a_{i,j'}\rangle 
(a_{i,j'+1},\lambda_{i,j'+1})\cdots
(a_{i,j-1},\lambda_{i,j-1})
\&
(a_{i,j+1},\lambda_{i,j+1})\cdots 
(a_{i,j'-1},\lambda_{i,j'-1}).
\end{align*} 
\end{itemize}

The product $X \ast Y$ of $X,Y \in \tilde{A}[h,\hbar]$ is defined as ``placing $Y$ above $X$''. Precisely, given two elements $X$, $Y$ of form (\ref{form}), we may assume that all heights of edges in $Y$ are greater than those in $X$. Indeed, one may replace the heights $\lambda'_{i,j}$ in $Y$ with $N+\lambda'_{i,j}$ using a large enough $N$. Then $X \ast Y$ is given by $X \& Y$. 
One may straightforwardly check $\tilde{B}[h,\hbar]$ is an ideal in the algebra $\tilde{A}[h,\hbar].$

\begin{construction}[The algebra $\na$]\label{na construction}
Let $\tilde{A}[h,\hbar]$ and $\tilde{B}[h,\hbar]$ be defined as above.
Put
\[\mathbf{N}(Q)_{h,\hbar} := \tilde{A}[h,\hbar] / \tilde{B}[h,\hbar].\] 
It is endowed with an algebra structure inherited
from $\tilde{A}[h,\hbar]$.
\end{construction}

By a \emph{product link}, we mean a product of knots and vertex components, with the empty product allowed.
By relations \eqref{quiver skein relations 1}, \eqref{quiver skein relations 2}, we have
 \begin{proposition}\label{remark:general form2}\label{formmodh}
For an arbitrary element $X\in \na$, consider 
the following
two cases: \begin{enumerate}[label = 
\rm{(}\arabic*),noitemsep,leftmargin=*,nosep]
\item $[X]\in\na/\hbar\na$; 

\item $[X]\in\na/h\na$.
\end{enumerate}  
Then for both cases, the class $[X]$ has a representative element such that: 
\begin{itemize}[noitemsep,leftmargin=*,nosep]
\item 
In case (1), this representative is a linear combination over $k[h]$ of product links,
and in case (2), 
this representative is a linear combination over $k[\hbar]$ of product links. 

\item The heights satisfy $\lambda_{i,j} < \lambda_{i',j'}$ whenever $(i, j) < (i', j')$ in the lexicographical order.
\end{itemize}
 \end{proposition}

\begin{proof}
    Without loss of generality, we assume that $X$ is a link.
    We prove the result for case (1) by induction on the number of components $|X|$ of $X$.
    For a knot  \( X\in\na\) as in (\ref{form}) with $k=1$, 
    we can exchange adjacent heights $\lambda_{i,j}$ and $\lambda_{i,j'}$, because the relation (\ref{quiver skein relations 2}) writes: \begin{align}
    X-X'_{i,j,i,j'}=0 \mod \hbar.\label{formula:overhbar}\end{align} Hence there is a desired representative for $[X]\in \na/\hbar \na$. 

 Assume the existence of the desired representative of $X$ if $|X|<n$. 
Put 
\[m(X)=\left|\left\{((i,j),(i',j'))\mid (i,j)<(i',j'),\ \lambda_{i,j}>\lambda_{i',j'}\right\}\right|.\]
 If $m(X)=0$, then $X$ is the desired representative of $[X].$ 
If $m(X)>0,$ choose two adjacent heights whose corresponding indices occur in reverse lexicographic order. 
If they are in the same knot, then \eqref{formula:overhbar} applies.
Otherwise the obtained $X'+h X''$ by exchanging them satisfies $m(X')=m(X)-1$ and $|X''|<n$. 
Hence (1) follows from the induction.

We move to the proof of (2).
Similar to (\ref{formula:overhbar}), we have
\begin{align}
X-X_{i,j,i',j'}'=0\text{ in }\na/h\na.\label{formula:overh}\end{align}
This means $[X]$ has a representative $Y$ that is a $k[\hbar]$-linear combination of elements of the form (\ref{form}) so that each of these elements satisfies the following height separation condition:
\begin{align}
\max\{\,\lambda_{i,j}\mid 1\leq j\leq l_i\,\} < \min\{\,\lambda_{i+1,j}\mid 1\leq j\leq l_{i+1}\,\} \text{ for each $1 \leq i \leq k-1$}.
\label{formula:afterskein1}
\end{align}
We may restrict our attention to the case where $Y$ is equal to a product link. 
We prove (2) by induction on the total number of arrows, and, for fixed arrow length, by induction on $m(Y).$ 
The case with no arrows consists of products of vertex components.
By (\ref{quiver skein relations 2}), we have $Y = Y' + \hbar Y''.$
The exchanged heights are adjacent and inverted, so $m(Y')=m(Y)-1,$ and the inner induction applies to $Y'.$ Every nonzero smoothing term $Y''$ has two fewer arrows, so the outer induction applies to $Y''.$ If the height separation condition \eqref{formula:afterskein1} fails for $Y'',$ we first restore it using \eqref{formula:overh}. This completes the induction.
\end{proof}

\subsection{The coproduct on $\na$}\label{sec:coprodna}

In this subsection, we 
give the coproduct on $\na$
(see also \cite[Section~3.2]{Schedler}).

\subsubsection{Construction of coproduct on $\na$}

Throughout this subsection, fix $X$ to be an element of the form $ (\ref{form}).$ 
Admit $a_{(i,j)}:=a_{i,j}$ and $\lambda_{(i,j)}:=\lambda_{i,j}.$ 
Let \[
P_X:=\{a_{(i,j)}\mid 1\leq i \leq k,1\leq j\leq l_i\}
\] be the set of all edges in $X$.
Admit the following expression on the lower index \[a_{
(i,j)+j'}:=a_{(i,j+j'\mod l_i)}.\] 
Let $V_X:=\{v_1,\cdots,v_m\}$ be the set of vertices in $X$. Denote by $|X|$ the number of components for $X$.

\begin{definition}[Coloring on $\na$]\label{color}
For an $X\in\na$, a $n$\textit{-coloring} of $X$ is a triple $(I,\phi,c)$ as below. When there is no risk of confusion, we abbreviate the tuple $(I,\phi,c)$ to $c.$
\begin{itemize}[noitemsep,leftmargin=*,nosep]
    \item  A subset $I$ of $P_X$ and its involutive map $\phi:I\to I$ satisfy that if $a_{i,j}\in I$ then $a_{i^\prime ,j^\prime} \coloneqq \phi(a_{i,j})$ satisfies $a^*_{i,j}=a_{i^\prime,j^\prime}$ and $a_{i,j}=a^*_{i^\prime,j^\prime};$
    \item  $c:P_X\cup V_X\to\{1,2,\cdots ,n\}$ that satisfies 
\[
c(a_{i,j}) = 
\begin{cases}
c(a_{i,j+1}), & \textup{if }a_{i,j}\notin I, \\
c(a_{i^\prime,j^\prime+1}),  & \textup{if }a_{i,j}\in I,
\end{cases}
\] and $c(a_{i,j}) > c(\phi(a_{i,j}))$ if and only if $\lambda_{i,j} > \lambda_{i',j'}$.
\end{itemize}Let $\lab_n(X)$ denote the set of all $n$-colorings of $X\in \na$.
\end{definition}

For a triple of $n$-coloring $(I,\phi,c)$ and $t\in \{1,2,\cdots,n\}$, we set an element $X^t_{I,\phi,c}$ of $\mathbf{N}(Q)_{h,\hbar}$ as below. 
First, define a bijection $f:P_X\to P_X$ by
\[
 f(a_{i,j})=\begin{cases}
a_{i,j+1}, & \textup{if }a_{i,j}\notin I, \\
a_{i',j'+1} , & \textup{if }a_{i,j}\in I.
\end{cases}
\]  
Then $P_X$ decomposes into orbits of $f$: $P_X=P_1\sqcup P_2\sqcup \cdots \sqcup P_q$ where $c(P_i)=\{t\}$ for some $1\leq t\leq n $. 
We arrange the index of $P_{1},\ldots,P_{q}$ so that we can write for $t=1,\ldots,n$
$$\{P_i\mid c(P_i)=\{t\}\}=\{P_{i_{t}},P_{i_{t}+1},\cdots,P_{i_{t+1}-1} \}  \ \text{with} \ 1=i_1\leq i_2\leq\cdots\leq i_{n+1}=q+1.$$  
We may write $P_i=\{a_{x_1},\cdots,a_{x_p}\}$ satisfying $f(a_{x_k})=a_{x_{k+1}}$ and $f(a_{x_p})=a_{x_1}$. 
Put \begin{align}
Y_i\coloneqq y_1\cdots y_p\in\mathsf{LH}\text{ with }
y_i = 
\begin{cases}
(a_{x_i},\lambda_{x_i}), &\textup{ if}\  a_{x_i}\notin I, \\
(a_{x_i})_s,  & \textup{ otherwise}.
\end{cases}\label{formula:cutting}\end{align} 
It means that $Y_{i_{t}},Y_{i_{t}+1},\cdots,Y_{i_{t+1}-1}$ are all the $Y_i$ with color $t$. 
We also order $V_X$ so that $c(v_i)=t$ for $j_t\leq i<j_{t+1}$  again with
$1= j_1 \leq j_2 \leq \cdots\leq j_{n+1} = m+1$.
Put
\[
X^t_{I,\phi,c}:=Y_{i_t}\&Y_{i_t+1}\&\cdots\&Y_{i_{t+1}-1}\& v_{j_t}\&\cdots\&v_{j_{t+1}-1}.
\] 
Put
\[ |c| := |I|\textup{ and }|
|c||:=|X|- 
\sum_{t=1}^n |X_{I,\phi,c}^t|,
\]
where \(|X|\) and \(|X^t_{I,\phi,c}|\) 
are  
the number of components 
in $X$ and \(X^t_{I,\phi,c}\), 
respectively (single vertices are treated as individual 
components).
Note that $|c|$ is always even due to the 
fact that $I$ contains both $a$ and $a^*$. For brevity, 
we denote $X^t_c$ by  $X^t_{I,\phi,c}$.
Consider
\begin{align}c_- \coloneqq 
\left\{ a_{i,j} \in I \cap Q \mid \lambda_{i,j} > \lambda_{i',j'} 
\text{ for } \phi(a_{i,j}) = a_{i',j'} \right\}.
\label{formula:c-}\end{align}
Let $|c_-|$ denote the cardinal of $c_-$. Then $|c_-|$ is the number of edges in $I\cap Q$ that have height larger than the corresponding edges (via $\phi$) in $Q^*$.

\begin{definition}[Coproduct on $\na$]\label{coproduct}
Given the notations above, put
\[\langle X, c \rangle\coloneqq (-1)^{|c_-|} h^{(|c|/2 + \|c\|)/2} \hbar^{(|c|/2 - \|c\|)/2},\quad \Delta(X, c)\coloneqq X^1_{c} \otimes \cdots \otimes X^n_{c}.\]
The $(n-1)$\textsuperscript{st} coproduct $\Delta^{n-1}$ of $X$ is given by:
\[\begin{aligned}
\Delta^{n-1}(X) &
\coloneqq \sum_{(I,\phi,c) \in \lab_n(X)}
\langle X, c \rangle \Delta(X, c).
\end{aligned}
\]
When $n=2$, $\Delta^1:\na\to\na\otimes\na$ is called
the \textit{coproduct} on $\na$ 
and is denoted by $\Delta$. 
When $n = 1$, we have $\Delta^0 = \mathrm{id}$, 
and when $n = 0$, it is defined as:
\[
\Delta^{-1}(X) =
\begin{cases}
1 & \textup{ if }X = \varnothing, \\
0 & \text{otherwise}.
\end{cases}
\]
\end{definition}
We verify that the exponents in $\langle X,c\rangle$ are nonnegative integers.
Resolve the $r=|I|/2$ prescribed pairs successively. Each resolution either merges two components or splits one component into two, counting vertex components individually. 
If $m$ resolutions are mergers and $s$ are splittings, then $r=m+s$ and $\sum_t|X_c^t|=|X|-m+s.$
Hence $\|c\|=m-s$, and therefore $(|c|/2+\|c\|)/2=m$ and $(|c|/2-\|c\|)/2=s.$
In particular, both exponents are nonnegative integers.
\begin{remark}
Given $X$ and $(I,\phi,c)$, we claim:
\[\deg_h\langle X,c\rangle-\deg_{\hbar}\langle X,c\rangle=||c||,\quad \deg_h\langle X,c\rangle+\deg_{\hbar}\langle X,c\rangle=|c|/2.\]
The first equation geometrically means that each factor of $h$ records a merger, decreasing the number of components by one; each factor of $\hbar$ records a splitting, increasing it by one.
The second equation means that the sum of the degrees of $h$ and $\hbar$ is equal to the number of the resolved pairs (cf. cutting points).
The above equations, which are implicit in Schedler's construction, could be regarded as an analog of the Conway relation in \cite[Figure~1]{Turaev}.
Our refinement is implicit in Schedler's construction but introducing $\hbar$ into the construction is necessary to match the biquantization framework.
\end{remark}
\begin{example}[Coproduct] Let $Q$ be the 
following quiver:
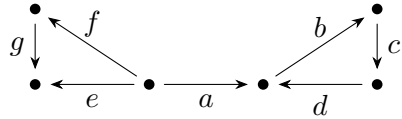
\vspace{-.3cm}
\begin{center}\begin{tikzpicture}[>=Stealth, line cap=round, line join=round]

\coordinate (U)  at (0,  .7);
\coordinate (D)  at (0,  -.7);
\coordinate (LU)  at (-3,  .5);
\coordinate (LD)  at (-3,  -.5);
\coordinate (ML)  at (-1.5,  -.5);
\coordinate (MR)  at (0,  -.5);
\coordinate (RU)  at (1.5,  .5);
\coordinate (RD)  at (1.5,  -.5);
\fill (LU)  circle (2pt);
\fill (LD)  circle (2pt);
\fill (ML)  circle (2pt);
\fill (MR)  circle (2pt);
\fill (RU)  circle (2pt);
\fill (RD)  circle (2pt);
\fill (U)  circle (0pt);
\fill (D)  circle (0pt);

\draw[shorten >=0.2cm,shorten <=.2cm,->]  (LU) -- node[left] {$g$} (LD);
\draw[shorten >=0.2cm,shorten <=.2cm,->]  (ML) -- node[below] {$e$} (LD);
\draw[shorten >=0.2cm,shorten <=.2cm,->]  (ML) -- node[above] {$f$} (LU);
\draw[shorten >=0.2cm,shorten <=.2cm,->]  (ML) -- node[below] {$a$} (MR);
\draw[shorten >=0.2cm,shorten <=.2cm,->]  (RD) -- node[below] {$d$} (MR);
\draw[shorten >=0.2cm,shorten <=.2cm,->]  (MR) -- node[above] {$b$} (RU);
\draw[shorten >=0.2cm,shorten <=.2cm,->]  (RU) -- node[right] {$c$} (RD);
\end{tikzpicture}
\end{center}\vspace{-.3cm}
Let
$X=(a,1)(d^*,2)(c^*3)(b^*,4)(a^*,5)(f,6)(g,7)(e^*,9)$ in $\na$, pictorially as follows:
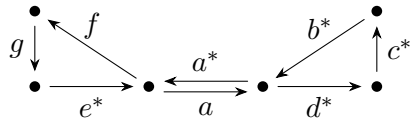
\vspace{-.3cm}
\begin{center}\begin{tikzpicture}[>=Stealth, line cap=round, line join=round]
\coordinate (U)  at (0,  .7);
\coordinate (D)  at (0,  -.7);
\coordinate (LU)  at (-3,  .5);
\coordinate (LD)  at (-3,  -.5);
\coordinate (ML)  at (-1.5,  -.5);
\coordinate (MR)  at (0,  -.5);
\coordinate (RU)  at (1.5,  .5);
\coordinate (RD)  at (1.5,  -.5);
\fill (LU)  circle (2pt);
\fill (LD)  circle (2pt);
\fill (ML)  circle (2pt);
\fill (MR)  circle (2pt);
\fill (RU)  circle (2pt);
\fill (RD)  circle (2pt);
\fill (U)  circle (0pt);
\fill (D)  circle (0pt);

\draw[shorten >=0.2cm,shorten <=.2cm,->]  (LU) -- node[left] {$g$} (LD);
\draw[shorten >=0.2cm,shorten <=.2cm,->]  (LD) -- node[below] {$e^*$} (ML);
\draw[shorten >=0.2cm,shorten <=.2cm,->]  (ML) -- node[above] {$f$} (LU);
\draw[shorten >=0.2cm,shorten <=.2cm,->]  ([yshift= -2pt]ML) -- node[below] {$a$} ([yshift= -2pt]MR);
\draw[shorten >=0.2cm,shorten <=.2cm,->]  (MR) -- node[below] {$d^*$} (RD);
\draw[shorten >=0.2cm,shorten <=.2cm,->]  (RD) -- node[right] {$c^*$} (RU);
\draw[shorten >=0.2cm,shorten <=.2cm,->]  (RU) -- node[above] {$b^*$} (MR);
\draw[shorten >=0.2cm,shorten <=.2cm,->]  ([yshift= 2pt]MR) -- node[above] {$a^*$} ([yshift= 2pt]ML);
\end{tikzpicture}
\end{center}\vspace{-.3cm}
Noticing
that the only non-trivial coloring $(I,\phi,c)$ is as 
follows:
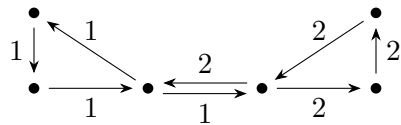
\vspace{-.3cm}
\begin{center}\begin{tikzpicture}[>=Stealth, line cap=round, line join=round]
\coordinate (U)  at (0,  .7);
\coordinate (D)  at (0,  -.7);
\coordinate (LU)  at (-3,  .5);
\coordinate (LD)  at (-3,  -.5);
\coordinate (ML)  at (-1.5,  -.5);
\coordinate (MR)  at (0,  -.5);
\coordinate (RU)  at (1.5,  .5);
\coordinate (RD)  at (1.5,  -.5);
\fill (LU)  circle (2pt);
\fill (LD)  circle (2pt);
\fill (ML)  circle (2pt);
\fill (MR)  circle (2pt);
\fill (RU)  circle (2pt);
\fill (RD)  circle (2pt);
\fill (U)  circle (0pt);
\fill (D)  circle (0pt);

\draw[shorten >=0.2cm,shorten <=.2cm,->]  (LU) -- node[left] {$1$} (LD);
\draw[shorten >=0.2cm,shorten <=.2cm,->]  (LD) -- node[below] {$1$} (ML);
\draw[shorten >=0.2cm,shorten <=.2cm,->]  (ML) -- node[above] {$1$} (LU);
\draw[shorten >=0.2cm,shorten <=.2cm,->]  ([yshift= -2pt]ML) -- node[below] {$1$} ([yshift= -2pt]MR);
\draw[shorten >=0.2cm,shorten <=.2cm,->]  (MR) -- node[below] {$2$} (RD);
\draw[shorten >=0.2cm,shorten <=.2cm,->]  (RD) -- node[right] {$2$} (RU);
\draw[shorten >=0.2cm,shorten <=.2cm,->]  (RU) -- node[above] {$2$} (MR);
\draw[shorten >=0.2cm,shorten <=.2cm,->]  ([yshift= 2pt]MR) -- node[above] {$2$} ([yshift= 2pt]ML);
\end{tikzpicture}\end{center}\vspace{-.3cm}
then we have $c_-=\varnothing,|c|=2$ and $||c||=-1$. So we have the coproduct of $X$ to be\vspace{-.3cm}
\begin{center}
\begin{tikzpicture}[>=Stealth, line cap=round, line join=round]

\coordinate (U)  at (0,  .7);
\coordinate (D)  at (0,  -.7);
\coordinate (LL)  at (-4.3,  0);
\coordinate (LU)  at (-1.8,  .5);
\coordinate (LD)  at (-1.8,  -.5);
\coordinate (ML)  at (-.3,  -.5);
\coordinate (M)   at (0,0);
\coordinate (MR)  at (.3,  -.5);
\coordinate (RU)  at (1.8,  .5);
\coordinate (RD)  at (1.8,  -.5);
\fill (LU)  circle (2pt);
\fill (LD)  circle (2pt);
\fill (ML)  circle (2pt);
\fill (MR)  circle (2pt);
\fill (RU)  circle (2pt);
\fill (RD)  circle (2pt);
\fill (U)  circle (0pt);
\fill (D)  circle (0pt);

\node at (LL) {$\Delta(X) = 1\otimes X+X\otimes 1 + \hbar$};
\node at (M) {$\otimes$};

\draw[shorten >=0.2cm,shorten <=.2cm,->]  (LU) -- (LD);
\draw[shorten >=0.2cm,shorten <=.2cm,->]  (LD) -- (ML);
\draw[shorten >=0.2cm,shorten <=.2cm,->]  (ML) -- (LU);
\draw[shorten >=0.2cm,shorten <=.2cm,->]  (MR) -- (RD);
\draw[shorten >=0.2cm,shorten <=.2cm,->]  (RD) -- (RU);
\draw[shorten >=0.2cm,shorten <=.2cm,->]  (RU) -- (MR);
\end{tikzpicture}
\end{center}\vspace{-.3cm}
This turns out to justify the coquantization condition in Definition~\ref{def:coquan}.
\end{example}

\subsubsection{Well-definedness of the 
coproduct}
  
Next, we check that 
$\Delta$ is well-defined; i.e., 
for $\tilde{B} \coloneqq \tilde{B}[h,\hbar]$ and $\tilde{A} \coloneqq \tilde{A}[h,\hbar]$ with $\widetilde A^{\otimes0}=k[h,\hbar],$ we have (all tensor products are taken over $k[h,\hbar]$)
\[
\Delta^{n-1}(\tilde{B})\subset 
\tilde{B}^n:=\sum_{i=1}^n\tilde{A}^{\otimes(i-1)} 
\otimes\tilde{B}\otimes \tilde{A}^{\otimes(n-i)}.
\]
When $n=2$, it becomes \[
\Delta(\tilde{B})\subset \tilde{B}^2:=\tilde{A}\otimes\tilde{B}+\tilde{B} \otimes \tilde{A}.
\]
\begin{lemma}\label{well}
The map $\Delta: \na \to \na\otimes \na$ in Definition~\text{\rm{\ref{coproduct}}} is well-defined. 
\end{lemma}

\begin{proof}
Fix indices $i,i',j,j'$ as in (\ref{quiver skein relations 1}) and (\ref{quiver skein relations 2}) such that $i$ and $i'$ may be equal. 
Let $|X|$ (resp. $|X''|$) denote the number of components in $X$ (resp. $X''$).
Put \( X' := X'_{i,j,i',j'} \), \( X'' := X''_{i,j,i',j'} \), and  
\[G \coloneqq X - X' - h^{\varepsilon} X'',\text{ where }
h^\varepsilon = \begin{cases}
h, &\textup{ if } \varepsilon = 1, \\
\hbar, & \textup{ if }\varepsilon = -1,
\end{cases}\text{ with }\varepsilon = |X|-|X''|.\]  
The goal is to show $\Delta(G) \in \tilde{B}^2$.
Put $f\coloneqq a_{i,j}\in Q $ and $f^*\coloneqq a_{i',j'}\in Q^*$. Put $P_{X''} := P_X\setminus \{ f, f^* \}$, the set of all edges of $X''.$
We prove this lemma for the case $\lambda(f^*)>\lambda(f)$, but the proof of the case $\lambda(f^*)<\lambda(f)$ turns out to be similar and is omitted.
We first show the lemma for the case where neither $a_{i', j'} = a_{i, j+1}$ nor $a_{i, j} = a_{i', j'+1}$ holds. 

If $f$ and $f^*$ are not reverses, exchanging their adjacent heights preserves all comparisons between paired edges. 
If either $f$ or $f^*$ is resolved, the resulting tensor factors agree up to order-preserving relabeling. 
If both survive, then they might be labeled with different color and their relative height is irrelevant; or otherwise, then their height exchange has zero smoothing term. 
Thus their contributions agree in the quotient.

We find relations between the colorings of $X$, $X'$, and $X''$ using the following graph, where $f+1$ (resp. $f-1$, etc.) denotes $a_{i,j+1}$ (resp. $a_{i,j-1}$, etc.)

\begin{minipage}{.4\textwidth}
\begin{center}
\begin{tikzpicture}[>=Stealth, line cap=round, line join=round]

\coordinate (Ltop)  at (-1.5,  1.0);
\coordinate (Lmid)  at (-1.5,  0.0);
\coordinate (Lbot)  at (-1.5, -1.0);

\coordinate (Rtop)  at (1.5,  1);
\coordinate (Rmid)  at (1.5,  0);
\coordinate (Rbot)  at (1.5, -1);

\coordinate (Mtop)  at (-0,  1.0);
\coordinate (Mbot)  at (-0, -1.0);
\coordinate (Mbbot) at (-0, -1.5);

\fill (Ltop)  circle (2pt);
\fill (Lbot)  circle (2pt);
\fill (Lmid) circle (2pt);
\fill (Rmid) circle (2pt);
\fill (Rtop)  circle (2pt);
\fill (Rbot)  circle (2pt);

\draw[shorten >=0.2cm,shorten <=.2cm,->]  (Lmid)-- node[pos=0.7, below] {\scriptsize $f^*+1$\quad\quad\quad\quad}  (Ltop);
\draw[shorten >=0.2cm,shorten <=.2cm,->]  (Lbot) -- node[pos=0.7, below] {\scriptsize $f-1$\quad\quad\quad\quad} (Lmid);

\draw[shorten >=0.2cm,shorten <=.2cm,->]  (Rtop)-- node[pos=0.7, above] {\scriptsize \quad\quad\quad$f^*-1$} (Rmid);
\draw[shorten >=0.2cm,shorten <=.2cm,->]  (Rmid) -- node[pos=0.7, above] {\scriptsize \quad\quad\quad$f+1$} (Rbot);

\draw[shorten >=0.2cm,shorten <=.2cm,->]  ([yshift=2pt]Rmid)-- node[pos=0.4, above] {\scriptsize $f^*$\quad\quad}  ([yshift=2pt]Lmid);
\draw[shorten >=0.2cm,shorten <=.2cm,->]  ([yshift=-2pt]Lmid) -- node[pos=0.45, below] {\scriptsize \quad\,$f$} ([yshift=-2pt]Rmid);

\node at (Mtop) {$\cdots$};
\node at (Mbot) {$\cdots$};
\node at (Mbbot) {The case $i\neq i'$};
\draw[shorten >=0.35cm,shorten <=.2cm,->]  (Ltop) -- (Mtop);
\draw[shorten >=0.2cm,shorten <=.35cm,->]  (Mtop) -- (Rtop);
\draw[shorten >=0.2cm,shorten <=.35cm,->]  (Mbot) -- (Lbot);
\draw[shorten >=0.35cm,shorten <=.2cm,->]  (Rbot) -- (Mbot);


\end{tikzpicture}
\end{center}
\end{minipage}
\begin{minipage}{.4\textwidth}
\begin{tikzpicture}[>=Stealth, line cap=round, line join=round]

\coordinate (Ltop)  at (-1.5,  1.0);
\coordinate (Lmid)  at (-1.5,  0.0);
\coordinate (Lbot)  at (-1.5, -1.0);

\coordinate (Cleft) at (0,  0.0);
\coordinate (Cright)at (1.5,  0.0);
\coordinate (Mbbot) at (.75, -1.5);

\coordinate (Rtop)  at (3,  1);
\coordinate (Rmid)  at (3,  0);
\coordinate (Rbot)  at (3, -1);

\fill (Ltop)  circle (2pt);
\fill (Lbot)  circle (2pt);
\fill (Cleft) circle (2pt);
\fill (Cright)circle (2pt);
\fill (Rtop)  circle (2pt);
\fill (Rbot)  circle (2pt);

\node at (Lmid) {$\vdots$};
\draw[shorten >=0.15cm,shorten <=.2cm,->]  (Ltop) -- (Lmid);
\draw[shorten >=0.2cm,shorten <=.35cm,->]  (Lmid) -- (Lbot);

\node at (Rmid) {$\vdots$};
\draw[shorten >=0.2cm,shorten <=.35cm,->]  (Rbot) -- (Rmid);
\draw[shorten >=0.15cm,shorten <=.2cm,->]  (Rmid) -- (Rtop);

\draw[shorten >=0.2cm,shorten <=.2cm,->] (Cleft) -- node[pos=0.4, above] {\scriptsize \quad \quad $f^{*}+1$} (Ltop);
\draw[shorten >=0.2cm,shorten <=.2cm,->] (Lbot)  -- node[pos=0.6, below] {\scriptsize \quad\,$f-1$}  (Cleft);

\draw[shorten >=0.2cm,shorten <=.2cm,->] (Rtop)  -- node[pos=0.6, above] {\scriptsize $f^{*}-1$\quad\quad} (Cright);
\draw[shorten >=0.2cm,shorten <=.2cm,->]  (Cright)-- node[pos=0.4, below] {\scriptsize $f+1$\quad\quad}  (Rbot);

\draw[shorten >=0.2cm,shorten <=.2cm,->]  ([yshift=-2pt]Cleft)  -- node[pos=0.5, below] {\scriptsize $f$}
      ([yshift=-2pt]Cright);

\draw[shorten >=0.2cm,shorten <=.2cm,->]  ([yshift= 2pt]Cright) -- node[pos=0.5, above] {\scriptsize $f^{*}$}
      ([yshift= 2pt]Cleft);
\node at (Mbbot) {The case $i = i'$};
\end{tikzpicture}
\end{minipage}

First consider a coloring $c=(I,\phi,c)$ for which at least one of $f,f^*$ belongs to $I,$ but they are not paired with each other. 
Since their heights are adjacent, exchanging them preserves all height comparisons between paired edges. 
Thus the same pairing and colors define a coloring of $X',$ with the same coefficient. 
Moreover, at least one of $f,f^*$ disappears upon cutting, so the resulting tensor factors agree up to order-preserving relabeling of heights. 
Consequently, $\Delta(X,c)=\Delta(X',c).$
This correspondence is bijective, and these contributions cancel in $\Delta(X)-\Delta(X').$ 
Henceforth, we consider only colorings for which either both selected occurrences are uncut or $\phi(f)=f^*$.

By Definition 3.3, if \(\phi(f)=f^*\), then $c(f)=c(f^*+1),$ $c(f^*)=c(f+1).$
If $f,f^*\notin I,$ then $c(f)=c(f+1),$ $c(f^*)=c(f^*+1).$
These give the following six cases:
\begin{table}[H] 
    \centering
    \begin{tabular}{c|c|c|c|c|c|c}
        & $C_1$ & $C_2$ & $C_3$ & $C_4$ & $C_5$ & $C_6$ \\
        \hline
        $c(f^*)$ & 1 & 2 & 2 & 1 & 1 & 2 \\
        $c(f+1)$   & 1 & 2 & 2 & 1 & 2 & 1 \\
        $c(f^*+1)$ & 1 & 2 & 1 & 2 & 1 & 2 \\
        $c(f)$   & 1 & 2 & 1 & 2 & 2 & 1 \\
    \end{tabular}
    \caption{Possible colorings}\label{Tableforresolvingcases}
\end{table}\vspace{-.3cm}
\vspace{-.3cm}
\noindent By Table~\ref{Tableforresolvingcases}, given $f$ and $f^*$, a coloring of $X$, $X'$, or $X''$ needs to satisfy one of the cases in this table. 
Let $c = (I,\phi,c)$ denote such a coloring thereafter.
\begin{itemize}[noitemsep,leftmargin=*,nosep]
    \item If $c$ satisfies the cases $C_1$ or $C_2$, it (or its restriction) must be a coloring of $X,$ $X',$ and $X''$ simultaneously.
    We have that $f$ and $f^*$ are both not in $I$.
    \item
    If $c$ satisfies the cases $C_3$ or $C_4$, it must be a coloring of $X''$ and a coloring of exactly one of $X$ or $X'$.
    We have that $f$ and $f^*$ are both in $I$.
    \item If $c$ satisfies the cases $C_5$ or $C_6$, it is not a coloring of $X''$ and is a coloring of both $X$ and $X'$.
    We have that $f$ and $f^*$ are both not in $I$.
\end{itemize}
Consider the following notations for the summands in the coproduct of $X,$ $X',$ and $X''$ that are corresponding to $c$: 
\begin{align*}
    \Delta(X,c)   &= \langle X, c \rangle (X)_{c}^1 \otimes (X)_{c}^2,   \\
    \Delta(X',c) &= \langle X', c \rangle (X')_{c}^1 \otimes (X')_{c}^2, \\
    \Delta(X'',c) &= \langle X'', c \rangle (X'')_{c}^1 \otimes (X'')_{c}^2.
\end{align*}
The desired inclusion will be proved on a case-by-case basis.
The proof for the cases $C_2$, $C_4,$ $C_6$ are respectively similar to $C_1$, $C_3,$ $C_5$ and are
left to the interested reader. 

Given $C_1$, we compute
\begin{align}
&\Delta(X,c) - \Delta(X',c) - h^\varepsilon\Delta(X'',c)\nonumber\\
&= \langle X,c\rangle (X)_c^1\otimes (X)_c^2-\langle X',c\rangle (X')_{c}^1\otimes (X')_{c}^2-h^\varepsilon\langle X'',c\rangle (X'')_{c}^1\otimes (X'')_{c}^2\nonumber\\
&=\big(\langle X,c\rangle(X)_c^1-\langle X',c\rangle(X')^1_{c}-h^\varepsilon \langle X'',c\rangle(X'')_{c}^1\big)\otimes (X)^2_c,
\label{formula:deltacomputation}
\end{align}
where the last equation is because of $(X)^2_c=(X')^2_{c}=(X'')^2_{c}$, by  $c(f)=c(f^*)=1$.

For the same $c$, we have $\langle X,c\rangle = \langle X',c\rangle$.
We let $|X|$, $|(X)_c^1|$, $|(X)_c^2|$, etc. denote the number of knots in respectively $X$, $(X)_c^1,$ $(X)_c^2$, etc.
Note that the set $c_-$ for $X$ is the same as the one when $c$ restricts  to $X''$.
By Definition~\ref{coproduct}, we compute 
\begin{equation}
\begin{array}{ccl}
\dfrac{\langle X,c\rangle}{\langle X'',c\rangle}
    & = &\dfrac{(-1)^{|c_-|}h^{|I|/4+(|X|-|(X)^1_c|-|(X)_c^2|)/2}\hbar^{|I|/4-(|X|-|(X)^1_c|-|(X)_c^2|)/2}}{(-1)^{|c_-|}h^{|I|/4+(|X''|-|(X'')^1_c|-|(X'')_c^2|)/2}\hbar^{|I|/4-(|X''|-|(X'')^1_c|-|(X'')_c^2|)/2}}
\\
& = & h^{(|X|-|X''|)/2-(|(X)_c^1|-|(X'')_{c}^1|)/2}\hbar ^{-(|X|-|X''|)/2+(|(X)^1_c|-|(X'')^1_{c}|)/2}
\\
& = & \dfrac{h^\varepsilon}{h^\mu},
\end{array}\label{formula:C1compute}\end{equation}
where $\mu:=|(X)^1_c|-|(X'')^1_{c}|=\pm 1.$
By (\ref{formula:deltacomputation}), we have the desired
\begin{align*}
&  \Delta(X,c) - \Delta(X',c) 
- h^\varepsilon\Delta(X'',c) \\
& =(\langle X,c\rangle)((X)_c^1-(X')^1_{c}
-h^\mu (X'')_{c}^1)\otimes (X)^2_c \in \tilde{B}^2.
\end{align*}

For case $C_3$, 
we have $\lambda(f^*)>\lambda(f)$ in $X$, then $\Delta(X',c) = 0$.
It suffices to show  $\Delta(X,c)=h^\varepsilon\Delta(X'',c)$.
The restriction of $c$ on $X''$ is $(I'',\phi|_{I''},c|_{P_{X''}})$ with $I''=I-\{f,f^*\}$.
Hence $(X'')_{c}^i=(X)_{c}^i$ for $i=1,2$ and $|X|-|X''|=\varepsilon.$ 
By $\lambda(f^*)>\lambda(f)$, we know $c_{-} = (c|_{P_{X''}})_{-}.$
By Definition~\ref{coproduct}, similar to (\ref{formula:C1compute}), we have the desired \[
\frac{\langle X,c\rangle}{\langle X'',c \rangle}=h^{(1+\varepsilon)/2}\hbar^{(1-\varepsilon)/2}=h^\varepsilon.\]

For case $C_5$, as $f,f^*\notin I$ and $c(f^*)\neq c(f+1)$, we have $\Delta(X'',c)= 0$. 
Note that the only difference between $X$ and $X'$ is the heights of $f$ and $f^*$. 
Then $f,f^* \notin I$ implies $\Delta(X,c) = \Delta(X',c)$, as desired.

It remains to show $\Delta(G)\in \tilde{B}^2$ for the case where (1) $a_{i', j'} = a_{i, j+1}$ or (2) $a_{i, j} = a_{i', j'+1}$. 
Now, we still have the set of edges $P_{X''}$ defined as above. 
If exactly one of (1) and (2) holds, then $V_{X''}=V_X\sqcup\{v\},$ where the new vertex component $v$ and its color are given below.
\[\begin{cases}
\textup{$v = (a_{i,j})_t$ and }c(v) = c(a_{i', j'}),
&\textup{if\, } a_{i', j'} = a_{i, j+1},\\
\textup{$v = (a_{i,j})_s$ and }c(v) =c(a_{i, j}),
&\textup{if\,} a_{i, j} = a_{i', j'+1}.
\end{cases}\]
We treat the case where (1) holds and (2) does not. The case where only (2) holds is analogous.
The proof for the other case is similar and is omitted.
Now the relation between the coloring of $X$, $X'$, and $X''$ could be seen by\vspace{-.3cm}
\[\begin{tikzpicture}[>=Stealth, line cap=round, line join=round]

\coordinate (Lmid)  at (-1.5,  0.0);
\coordinate (Mmid)  at (0,  0.0);

\coordinate (Rtop)  at (1.5,  1.0);
\coordinate (Rmid)  at (1.5,  0);
\coordinate (Rbot)  at (1.5,  -1.0);

\coordinate (Mbbot)  at (0,  -1.5);
\fill (Lmid)  circle (2pt);
\fill (Mmid)  circle (2pt);
\fill (Rtop)  circle (2pt);
\fill (Rbot)  circle (2pt);


\draw[shorten >=0.2cm,shorten <=.2cm,->]  ([yshift=2pt]Lmid)  -- node[pos=0.5, above] {$f^{*}$} ([yshift=2pt]Mmid);

\draw[shorten >=0.2cm,shorten <=.2cm,->]  ([yshift= -2pt]Mmid) -- node[pos=0.5, below] {\scriptsize $f$}
      ([yshift= -2pt]Lmid);

\node at (Rmid) {$\vdots$};
\draw[shorten >=0.2cm,shorten <=.2cm,->]  (Mmid) -- node[pos=0.3, above]{$f^*+1$\quad\quad} (Rtop);
\draw[shorten >=0.2cm,shorten <=.2cm,->]  (Rbot) -- node[pos=0.7, below]{$f-1$\quad\quad} (Mmid);
\draw[shorten >=0.15cm,shorten <=.2cm,->]  (Rtop) -- (Rmid);
\draw[shorten >=0.2cm,shorten <=.35cm,->]  (Rmid) -- (Rbot);


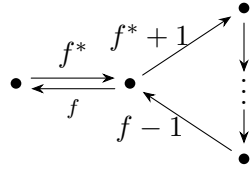
\node at (Mbbot) {The case $a_{i',j'} = a_{i,j+1}$};
\end{tikzpicture}\vspace{-.3cm}
\]
If both (1) and (2) hold, the component consists of the two edges $a_{i,j}$ and $a_{i',j'}.$ 
Resolving them produces two vertex components representing $(a_{i,j})_s$ and $(a_{i,j})_t,$ with colors $c(a_{i,j})$ and $c(a_{i',j'}),$ respectively. 

If $\phi(f)=f^*,$ then $c(f)=c(f^*+1).$
If \(f,f^*\notin I\), then $c(f)=c(f^*)=c(f^*+1).$
These give the following four cases:
\vspace{-.3cm}
\begin{table}[H]
    \centering
    \begin{tabular}{c|c|c|c|c}
        & $C_1'$& $C_2'$ &$C_3'$&$C_4'$\\
         \hline
       $c(f^*+1)$&     1&2&1&2\\
       $c(f)$&         1&2&1&2\\
     
       $c(f^*)=c(v)$&  1&2&2&1
        
    \end{tabular}
    \caption{Colorings}
    \label{tab:placeholder4}
\end{table}
\vspace{-.3cm}\noindent Note that there is no analogues of $C_5$ and $C_6$. For cases \(C_1'\) and \(C_2'\), we know $I=I''$, and \( X \), \( X' \), \( X'' \) are nontrivial. 
Subsequent discussions are consistent with the case $C_1$.
For cases \( C_3' \) and \( C_4' \), the vertex $v = (a_{i,j})_s$ or $(a_{i,j})_t$ is treated as a component. 
Subsequent discussions are consistent with the case $C_3$ and $C_4$.
This finishes the proof.
\end{proof}

\begin{lemma}\label{coasso}
Let the coproduct $\Delta$ be 
given in Definition \ref{coproduct}. Then the coproduct 
$\Delta$ in $\na$ is coassociative; i.e., 
\[
(\Delta\otimes 1)\circ\Delta=(1\otimes \Delta)\circ\Delta.
\]
\end{lemma}

\begin{proof}

We prove the equation $(\Delta\otimes 1)\circ\Delta(X)= \Delta^{2}(X).$
Given this equation, the lemma follows from $(1\otimes \Delta)\circ\Delta(X)= \Delta^{2}(X)$, whose proof is similar and hence omitted. 
Compute
\[\begin{aligned}
   (\Delta\otimes1)\circ\Delta(X)&=(\Delta\otimes1)\bigg(\sum_{c\in \lab_2(X)}\langle X,c\rangle X_c^1\otimes X_c^2\bigg)\\&=\sum_{c\in \lab_2(X)}\langle X,c\rangle \sum_{c^1\in \lab_2(X_c^1)}\langle X_c^1,{c^1}\rangle (X_c^1)_{c^1}^1\otimes (X_c^1)_{c^1}^2\otimes X_c^2\\
    &=\sum_{\substack{c\in \lab_2(X)\\c^1\in \lab_2(X_c^1)}}\langle X,c\rangle\langle X_c^1,c^1\rangle (X_c^1)_{c^1}^1\otimes (X_c^1)_{c^1}^2\otimes X_c^2.
\end{aligned}\] 
Our next goal is to construct an bijection $\lab_3(X)\to \bigsqcup_{c\in \lab_2(X)}\lab_2(X_{c}^1).$
For each $c\in \lab_2(X)$ and $c^1\in \lab_2(X_c^1),$ write $c = (I,\phi,c)$ and $c^1 = (I^1,\phi^1,c^1)$.
We treat the elements in  $P_{(X_c^1)_{c^1}^1}\sqcup P_{(X_c^1)_{c^1}^2} \sqcup P_{X_c^2}\sqcup  V_{(X_c^1)_{c^1}^1}\sqcup V_{(X_c^1)_{c^1}^2} \sqcup V_{X_c^2}\sqcup I \sqcup I^1$ as the corresponding elements in $P_X\sqcup V_X$.
Here, the identification with $P_X\sqcup V_X$ applies to original edge occurrences and original vertex components, and newly created vertex components are labeled by their full cutting orbits.
Put 
\begin{align}(\tilde{I},\tilde{\phi},\tilde{c}) \coloneqq (I^1 \sqcup I, \phi^1\sqcup \phi, \tilde{c})\text{ with }
\tilde{c}(x)\coloneqq \begin{cases}3, &\textup{if\,}c(x) = 2,\\
c^1(x), &\textup{if\,} 
c(x) = 1,\,\,x\notin I,\\
c^1(x'), &\textup{if\,} c(x)=1,\,\,x\in I,\end{cases}\,\,x\in P_X\sqcup V_X.
\label{formula:coass1}\end{align}
To define $x$, notice that on the cutting set $I$, 
the refined color propagates forward along the cutting permutation $f$ of $(I,\phi)$ until reaching an edge outside $I$ (See Definition~\ref{color} for $f$); and if the whole orbit lies in $I$, its refined color is the color assigned by $c^1$ to the resulting vertex component. 
For $x\in I$ with $c(x)=1,$ let $O_x$ be the orbit of $x$ under $f$.
If $O_x\setminus I\ne\varnothing$, put $r(x):=\min\{\,q\ge1\mid f^q(x)\notin I\,\}.$
If $O_x\subseteq I$, let $v_{O_x}$ denote the vertex component produced by this orbit by \eqref{formula:cutting}. Define
\[x'=\begin{cases}f^{r(x)}(x),&O_x\setminus I\ne\varnothing,\\v_{O_x},&O_x\subseteq I.\end{cases}\]
The minimum exists in the first case because $O_x$ is finite. 
Vertex components produced by different orbits are kept distinct, even when they represent the same vertex of the quiver.
For $X_{\tilde{c}}^1$, $X_{\tilde{c}}^2$, and $X_{\tilde{c}}^3$, we have 
\begin{align}
X_{\tilde{c}}^1=(X^1_{c})^1_{c^1},\,\,X_{\tilde{c}}^2=(X^1_{c})^2_{c^1},\text{ and }X_{\tilde{c}}^3=X^2_{c}.\label{formula:coassoccomp}\end{align}
Given a coloring $(\tilde{I},\tilde{\phi},\tilde{c})\in \lab_3(X),$ put 
\begin{equation}
\begin{gathered}
c(x)\coloneqq \begin{cases}1, &\textup{if } \tilde{c}(x) = 1\text{ or }2\\
2, &\textup{if } \tilde{c}(x) = 3,\end{cases},\quad c^1 \coloneqq \tilde{c}|_{X_c^1},\\ 
\text{$I^1 \coloneqq \big\{a_{i,j},a_{i,j}^*\in \tilde{I}\mid c(a_{i,j}) = c(a_{i,j}^*) = 1\big\},$ $I\coloneqq \tilde{J} - I^1,$ and $\phi^1\coloneqq \tilde{\phi}|_{I^1}$, $\phi\coloneqq \tilde{\phi}|_I$,}\end{gathered}\label{formula:coass2}\end{equation}
and a vertex component $v_{O_x}$ produced by an orbit $O_x\subseteq I$ of color $1$, put $c^1(v_{O_x})=\widetilde c(x)$ for $x\in O_x.$
This is independent of the choice of $x.$
One may check straightforwardly that (\ref{formula:coass1}) and (\ref{formula:coass2}) define the following bijection
\[\lab_3(X)\to \bigsqcup_{c\in \lab_2(X)}\lab_2(X_{c}^1);\quad (\tilde{I},\tilde{\phi},\tilde{c})\mapsto (I^1,\phi^1,c^1)\in \lab_{2}(X_{c}^1).\]

Compare
\begin{align*}
(\Delta\otimes1)\circ\Delta(X)=\sum_{\substack{c\in 
\lab_2(X)\\c^1 \in \lab_2(X_c^1)}}\langle X,c\rangle\langle 
X_c^1,c^1\rangle (X_c^1)_{c^1}^1\otimes (X_c^1)_{c^1}^2\otimes 
X_c^2
\end{align*}
and
\begin{align*}
\Delta^2(X)&=\sum_{\tilde{c}\in\lab_3(X)}\langle 
X,\tilde{c}\rangle  X_{\tilde{c}}^1\otimes  
X_{\tilde{c}}^2\otimes X_{\tilde{c}}^3.
\end{align*}
By the bijection and (\ref{formula:coassoccomp}), it suffices to show $\langle X,\tilde{c}\rangle=\langle X,{c}\rangle\langle X_c^1,{c}^1\rangle$ for $c$ and $c^1$ corresponding to any given $\tilde{c}$.
By Definition~\ref{coproduct}, this follows from the computations below:
\[\begin{aligned}
|\tilde{c}|&=|\tilde{J}|=|I^1|+|I|\text{, }\\
||\tilde c||&=|X|-|X_{\tilde{c}}^1|-|X_{\tilde{c}}^2|-|X_{\tilde{c}}^3|\\&=|X|-|(X^1_{c})^1_{c^1}|-|(X^1_{c})^2_{c^1}|-|X_{c}^2|\\&=
|X|-|X_{c}^1|+|X_{c}^1|-|(X^1_{c})^1_{c^1}|-|(X^1_{c})^2_{c^1}|-|X_{c}^2|\\&=
|X|-|X_{c}^1|-|X_{c}^2|+||c^1||\\&=
||c||+||c^1||\text{, and}\\
|\tilde c_-| & =|c_-|+|c^1_-|,
\end{aligned}
\]
where the last equation follows from
\[
\begin{aligned}\tilde{c}_-&=\{a_{i,j}\in \tilde{J}\cap Q\mid \lambda_{i,j}>\lambda_{i',j'}\}\\&=\{a_{i,j}\in I^1\cap Q\mid \lambda_{i,j}>\lambda_{i',j'}\}\sqcup\{a_{i,j}\in I\cap Q\mid \lambda_{i,j}>\lambda_{i',j'}\}\\&=c^1_-\sqcup c_-.
\end{aligned}\]
This finishes the proof.
\end{proof}

By the above lemma, we are able to safely set the 
multiple composition of the coproduct: $\Delta^{(n-1)}:=
(\Delta\otimes\id\otimes\cdots\otimes\id)\circ\Delta^{(n-2)}$ 
by induction.

\subsection{Bialgebra conditions for $\na$}\label{sec:bialg na}

We next show the coproduct
is a map of algebras.

\begin{lemma}\label{bial} Let $X,Y\in \na$ be two arbitrary 
elements. Then
\[
\Delta(XY)=\Delta(X)\Delta(Y).
\]
\end{lemma}

\begin{proof}
By Definition~\ref{coproduct}, we compute 
\begin{align}\label{formula:coprodmorphism}
\Delta(XY)&=\sum_{c\in \lab_2(XY)}\langle XY,c\rangle(XY)_c^1\otimes(XY)_c^2
\end{align}
and
\begin{align*}
\Delta(X)\Delta(Y)&=\bigg(\sum_{d\in \lab_2(X)}\langle X,d\rangle X_d^1\otimes X_d^2\bigg)\bigg(\sum_{e\in \lab_2(Y)}\langle Y,e\rangle Y_e^1\otimes Y_e^2\bigg)\\
&=\sum_{\substack{d\in \lab_2(X)\\e\in \lab_2(Y)}}\langle X,e\rangle\langle Y,b\rangle X_d^1Y_e^1\otimes X_d^2Y_e^2.\end{align*} 
To compare $\Delta(XY)$ and $\Delta(X)\Delta(Y)$, we look into (\ref{formula:coprodmorphism}).
Fix $c\in \lab_2(XY).$
If $\phi(P_X\cap I)\cap P_Y = \varnothing$, or equivalently $\phi(P_Y\cap I)\cap P_X = \varnothing$, then for each $c = (I,\phi,c)\in \lab_2(XY)$, we have two colorings $(I\cap P_X,\phi|_{P_X},c|_{P_X\sqcup V_X})\in \lab_2(X)$ and $(I\cap P_Y,\phi|_{P_Y},c|_{P_Y\sqcup V_Y})\in \lab_2(Y).$
Notice $P_{XY} = P_{X}\sqcup P_Y$ and $V_{XY} = V_{X}\sqcup V_Y$.
Two colorings $(I_d,\phi_d,d)$ and $(I_e,\phi_e,e)$ defines a coloring $(I_{d}\sqcup I_{e},\phi_{d}\sqcup \phi_{e},d\sqcup e)\in \lab_2(XY)$. 
These establish a bijection
\[\lab_2(XY)' \to \lab_2(X) \times \lab_2(Y)\]
with $\lab_2(XY)' \coloneqq \{(I,\phi,c)\in \lab_2(XY)\mid \phi(P_X\cap I)\cap P_Y = \varnothing\}.$
Put $\lab_2(XY)''\coloneqq \lab_2(XY) - \lab_2(XY)'$. 
Notice $\langle X,d\rangle\langle Y,e\rangle =\langle XY,d\sqcup e\rangle$.
By (\ref{formula:coprodmorphism}), we have:
\[\begin{aligned}
\Delta(XY)
&=\sum_{\substack{d\in \lab_2(X)\\e\in \lab_2(Y)}}\langle X,d\rangle\langle Y,e\rangle X_d^1Y_e^1\otimes X_d^2Y_e^2+\sum_{c\in\lab_2(XY)''} \langle X,c\rangle(XY)_c^1\otimes (XY)^2_c
\\&=\Delta(X)\Delta(Y)+\sum_{c\in\lab_2(XY)''} \langle X,c\rangle(XY)_c^1\otimes (XY)^2_c.
\end{aligned}\] 

It suffices to prove $\lab_2(XY)'' = \varnothing.$
Assume conversely that there exists $a_{i,j}\in I\cap P_X$ such that $\phi(a_{i,j})\in P_Y$ for some $c\in \lab_2(XY).$
Put $S_{c}\coloneqq \{a_{i,j}\in I\cap P_X\mid \phi(a_{i,j})\in P_Y\}.$ 
We show that
\begin{equation}\label{SC}
    \sum_{a_{i,j}\in S_c}c(a_{i,j}) - c(a_{i,j+1})\text{ is both }=0\text{ and }<0,
\end{equation}
which implies $\lab_2(XY)''= \varnothing$ and the desired equation hence holds.
For the $<0$ part, notice $\lambda_{i,j}<\lambda_{i',j'}$.
We have $c(a_{i,j}) - c(a_{i,j+1}) = c(a_{i,j}) - c(a_{i',j'}^*) < 0$ for $a_{i,j}\in S_{c}$ and $a_{i',j'}^* = \phi(a_{i,j})$.
Hence, the $<0$ part follows.
As for the $=0$ part, because $X$ is given by cyclic paths, we have
\[\sum_{a_{i,j}\in P_X}c(a_{i,j}) - c(a_{i,j+1}) = 0.\]
On the other hand, because $c(a_{i,j}) = c(a_{i,j+1})$ if $a_{i,j}\notin I$, we have
\begin{align*}
\sum_{a_{i,j}\in P_X}c(a_{i,j}) - c(a_{i,j+1}) 
& =\sum_{\substack{a_{i,j}\in I\cap P_X\\
\phi(a_{i,j})\in P_X}}\bigl(c(a_{i,j})-
c(a_{i,j+1})\bigr)+\sum_{a_{i,j}\in 
S_c}\bigl(c(a_{i,j})
-c(a_{i,j+1})\bigr)\\
& = \sum_{a_{i,j}\in S_c}
c(a_{i,j})-c(a_{i,j+1}),
\end{align*}
where the second equation holds because for each 
$a_{i,j}\in I\cap P_X$ with $\phi(a_{i,j})\in P_X$, 
the equation $c(a_{i,j+1}) = c(a_{i',j'}^*)$ holds for 
$a_{i',j'}^* = \phi(a_{i,j})$ and the summmand 
$c(a_{i,j})-c(a_{i',j'}^*) + c(a_{i',j'}^*)-c(a_{i,j})$ 
is in the sum. 
Hence the $=0$ part holds.
\end{proof}

Summarizing the above Lemmas \ref{well}-\ref{bial},
we have:

\begin{theorem}\label{bi}
Let $\na$ be equipped with the coproduct $\Delta$ and 
the product $*$ as above. Then $\na$ is a bialgebra.
\end{theorem}

\section{The quantization  
$\mathbf{N}(Q)_{h,\hbar} 
\to
\varepsilon_\hbar(L_\hbar)$}\label{quancon}

In this section, we first consider a Poisson 
algebra $E(L_\hbar)$ that is isomorphic to 
$\varepsilon_\hbar(L_\hbar)$ (cf. 
\cite[Section~12]{Turaev}).
Then we construct a map $\mathcal{J}:\na \to 
E(L_\hbar)$ and show that the 
composition 
\begin{equation}
\label{eq:quantizationviaE}
\na \overset{\mathcal{J}}{\to} E(L_\hbar) 
\overset{\cong}{\to} 
\varepsilon_\hbar(L_\hbar)
\end{equation}
is a reduced quantization of 
$\varepsilon_\hbar(L_\hbar).$

\subsection{A Poisson bialgebra and its 
quantization}\label{sec:constructELhbar}

In this subsection, we first
construct $E(L_\hbar)$
and $\mathcal{J}$ in \eqref{eq:quantizationviaE}.
That \eqref{eq:quantizationviaE} is
a quantization is stated in
Theorem~\ref{quantization varepsilon},
whose proof is given in
the rest of this section.

\subsubsection{A Poisson algebra $E(L_\hbar)$}

For a Lie bialgebra $\la $ over $k[\hbar]$, set 
$\mathcal{ F} = \mathcal{ F}(\la):=\bigoplus_{n\geq 0 } \la ^{\otimes n}$
the tensor algebra generated by $\mathfrak a$ over
$k[\hbar]$.
It has a natural bialgebra 
structure given as follows: 
the coproduct is given by 
\[\textup{diag}:\mathcal{ F}\to \mathcal{ F}\otimes\mathcal{ F},
\quad a_1\cdots a_m
\mapsto \sum_{i=0}^m(a_1\cdots a_i)
\otimes( a_{i+1}\cdots a_m).\]

The product on $\mathcal{F}$, denoted by $\bullet$,
is given by the shuffle product: 
for $a = a_{1}\cdots a_{m}\in \la^{\otimes m}$ and $b = b_{1}\cdots b_{n}\in \la^{\otimes n}$, \begin{align}\label{shuffle}
&(a_1\cdots a_m)\bullet(b_1\cdots b_n) \coloneqq \sum_{\sigma\in \mathrm{Shuf}(m,n)} c_{\sigma,1}\cdots c_{\sigma,m+n}
\end{align}
with
$$\text{Shuf}(m,n):=\left\{\sigma\in S_{m+n}
\left|
\begin{array}{l}\sigma(1)<\sigma(2)<\cdots<\sigma(m)\textup{ and }\\
\sigma(m+1)<\cdots<\sigma(m+n)
\end{array}\right.\right\}$$
and $$c_{\sigma,i}:=
\begin{cases}
a_{\sigma^{-1}(i)},&\textup{if $\sigma^{-1}(i)\leq m$},\\
b_{\sigma^{-1}(i)-m},&\textup{otherwise.}
\end{cases}$$

\begin{lemma}
[{Turaev \cite[Lemma 12.2]{Turaev}}]\label{Lie}
For $\mathcal{F}$ defined as above, 
consider the bracket $\{-,-\}:\mathcal F\otimes\mathcal{F}\to \mathcal{ F}$ defined by
\[\{a_1 \cdots a_m,b_1 \cdots b_n\} = \sum_{\substack{1\leq i\leq m\\1\leq j\leq n}}(a_{1,i-1} \bullet b_{1,j-1}) \otimes [a_{i},b_{j}] \otimes (a_{i+1,m} \bullet b_{j+1,n}),\]
where $a_{i,i+j} := a_{i}\cdots a_{i+j}\in\mathcal{F}$ and $[a_{i},b_{j}]$ denotes the Lie bracket of $a_{i},b_{j}$ in $\la$.  
Then this bracket $\{-,-\}$ is a Poisson bracket so that $\mathcal{F}$ is a Poisson bialgebra.
\end{lemma}

For an element $a\in \mathcal F$,
suppose 
$a=\prescript{0}{}{a} + \prescript{1}{}{a} + \cdots$ with $\prescript{n}{}{a}\in \mathfrak{a}^{\otimes n}$ for $n\geq 0$.
Given an integer $n$, let $\sigma_i$ and $\nu_{i}$ denote the operators 
such that 
\begin{align*}
\sigma_i (a_1\cdots a_n) 
& = a_{1,i-1}\otimes a_{i+1}\otimes 
a_{i}\otimes a_{i+2,n},\\
\nu_i (a_1\cdots a_{n-1}) 
& = a_{1,i-1}\otimes \nu(a_{i})\otimes a_{i+1,n-1}.
\end{align*}
Set \begin{align}E(\mathfrak a) \coloneqq \{a\in\mathcal{F}\mid 
\prescript{n}{}{a} - \sigma_i(\prescript{n}{}{a}) 
= \hbar \nu_i(\prescript{n-1}{}{a})\text{ for all }n>i\geq 1\}
\subset \mathcal{F},\label{formula:defE}\end{align} which is a $k$-submodule of $\mathcal{F}$. 
For general spiral Lie algebras, we have

\begin{theorem}[{Turaev \cite[Lemma 12.4]{Turaev}}]\label{symmisomorphism}
For a spiral Lie bialgebra $\la$ over $k[\hbar]$, the submodule $E(\la)$ is a Poisson 
subbialgebra of $\mathcal{F}$ with the Lie bracket in Lemma~\text{\rm{\ref{Lie}}}. Moreover, the map  
\begin{equation}
\mathrm{symm}|_{E(\la)}:
E(\la )\to S(\la ), \quad 
a_1\cdots a_m\mapsto\frac{1}
{m!}\prod_{i=1}^m a_i,
\end{equation}
is an algebra morphism,  which transforms the comultiplication in $E(\la )$ into the comultiplication $\nabla_\hbar$ in $S(\la)$ as in Definition~\text{\rm{\ref{def:epsilon(g)}}}.
\end{theorem}

\begin{corollary}\label{E}
Let $\varepsilon_\hbar(L_\hbar)$ be given as in 
Definition~\text{\rm{\ref{def:epsilon(g)}}}. Then  
the map $\symm|_{E(L_\hbar)}$ is an isomorphism of Poisson 
bialgebras from $E(L_\hbar)$ to $\varepsilon_\hbar(L_\hbar)$.
\end{corollary}

\subsubsection{The quantization map $\mathcal{J}$ of $E(L_\hbar)$}\label{sec:mapJ}

We are to construct a map $\mathcal{J}:\na\to E(L_\hbar)$.
Let $X$ be an element in $\na$ of the form \eqref{form}.
If $X$ has exactly one component, let $\langle X\rangle $ 
denote the underlying path of $X$, i.e., 
\[\text{if $X = (a_{1},\lambda_{1})\cdots (a_{l},\lambda_{l})$, 
then put $\langle X\rangle \coloneqq a_{1}\ldots a_{l}$; if $X = v$, 
then put $\langle X\rangle = v$}.\]
For the natural projection $\pi:k[h,\hbar] \to k[\hbar]$, 
consider a $\pi$-linear map $q:\na\to L_\hbar$ defined below
\[q(X)=\begin{cases}
\langle X\rangle,
 &\text{if $X$ has only one component,} \\
0,&\text{otherwise}.
\end{cases}\] 
Here the $\pi$-linearity means $q(rX) = \pi(r)q(X)$ for any $r\in k[h,\hbar]$ and $X\in\na$.

The map $q$ is well-defined:
regarding the relations (\ref{quiver skein relations 1}) 
and (\ref{quiver skein relations 2}), the equations 
$q(X) = q(X_{i,j,i',j'}')$, $q(hX_{i,j,i',j'}'')=0$ 
and $q(\hbar X_{i,j,i,j'}'')=0$ always hold.
For $\la=k[\hbar]\otimes L=L_\hbar$, let $\mathcal{F}$ 
denote the tensor algebra $\mathcal{F}(\mathfrak{a})$ in 
Section~\ref{sec:constructELhbar}, set the 
map\begin{equation}\label{j}
\mathcal{J}:=\bigoplus_{n\geq 0}q^{\otimes n}\circ
\Delta^{n-1}:\mathbf N(Q)_{h,\hbar}\to \mathcal{ F}.
\end{equation}  
The rest of this section is devoted to showing 

\begin{theorem}\label{quantization varepsilon}
The map $p_\hbar=\mathrm{symm}|_{E(L_\hbar)}\circ \mathcal{J}$ 
is a reduced quantization of $\varepsilon_\hbar(L_\hbar).$
\end{theorem}

\noindent See the end of Section~\ref{section:JisaquantizationofELhbar} 
for the proof for the quantization part.
The reducedness will be proved in Section~\ref{reducon}.

\subsection{Proof of quantization, part 1}\label{sec:quantizationpart1}

We first show that the map $\mathcal{J}$ is a bialgebra morphism.
We then check all remaining conditions for quantization
(Definition~\ref{2.7}), except for (\ref{formula:quantization}), 
which we prove in Section~\ref{sec:quantizationpart2}.

\subsubsection{$\mathcal{J}$ is a bialgebra morphism}

\begin{notation}\label{0}
For the set of $n$-coloring on $X$, let $\lab_n(X)$ be 
introduced by Definition \ref{color}. 
Set 
\[\lab_n^0(X):=\bigl\{(I,\phi,c)\in\lab_n(X)
\,\big|\, |c|/2=-||c||,\text{ and }|X_c^i| = 1\text{ for }1\leq i\leq n\bigr\},\] 
where $|X_c^r|$ counts components, including vertex components.
Admit the convention that $\lab_0^0(1)$ contains the unique empty coloring, while $\lab_0^0(X)=\varnothing$ for nonempty links. 
This accounts for the degree-zero term of $\mathcal J.$
\end{notation} 
\noindent The condition for $(I,\phi,c)$ means 
that each $X_c^i$ for $i = 1,\cdots, n$ in $\Delta^{n-1}$ 
has exactly one component. This allows us to rewrite 
$\mathcal{ J}(X)$ for $X $ of the form \eqref{form} as
\begin{equation}\label{buhuizaiyong}
\mathcal{ J}(X)=\sum_{n\geq 0}\bigg(\sum_{(I,\phi,c)\in 
\lab_n^0(X)}\langle X,c\rangle\langle 
X_c^1\rangle\otimes\cdots\otimes\langle X_c^n\rangle\bigg).
\end{equation}

\begin{lemma}\label{buhuizaiyong1} 
The map $\mathcal J$ is an algebra morphism.
\end{lemma}

\begin{proof}
For $X,Y\in \na$,  we want to prove $\mathcal{J}
(XY)=\mathcal{J}(X)\bullet\mathcal{J}(Y)$. 
By definition \eqref{j} of $\mathcal{J}$, we only prove 
that \[ q^{\otimes n}\circ 
\Delta^{n-1}(XY)=\sum_{r=0}^n (q^{\otimes 
r}\circ\Delta^{r-1}(X))\bullet 
(q^{\otimes n-r}\circ\Delta^{n-r-1}(Y)).\]
By 
\eqref{buhuizaiyong}, we have
\[
\begin{aligned}
q^{\otimes n}\circ \Delta^{n-1}(XY)&= q^{\otimes 
n}\bigg(\sum_{c\in \lab_n^0(XY)}\langle XY,c\rangle 
(XY)^{1}_c\otimes\cdots\otimes
(XY)^{n}_c\bigg).\end{aligned}
\]
Similarly 
to Lemma~\ref{bial}, one may show that there is no 
coloring $c=(I,\phi,c)\in \lab_n^0(XY)$ such that 
$\phi(x)\in P_Y$ for $x\in I\cap P_X$.
We also have $\langle XY,c\rangle
=\langle X,c|_{P_X\sqcup V_X }\rangle\langle 
Y,c|_{P_Y\sqcup V_Y }\rangle$ for $c\in \lab_n(XY)$.
By the definition of $\lab_n^0(XY),$ each $(XY)_c^i$ either equals $X_{c}^i$, or equals $Y_{c}^i$, where $X_c^i$ is in $\Delta(X)$ given by the color $(I\cap P_X,\phi|_{P_X},c|_{P_X\cup V_X})$ and $Y_c^i$ is in $\Delta(Y)$ given similarly.
Hence $(XY)^{1}_c\otimes\cdots\otimes (XY)^{n}_c\ $ is actually a shuffle of $X_c^i$ and $Y_c^i$.

We compute
\begin{eqnarray*}
&& q^{\otimes n}\circ \Delta^{n-1}(XY)\\
&=& q^{\otimes n}\Big(\sum_{c\in \lab_n^0(XY)}\langle 
XY,c\rangle (XY)^{1}_c\otimes\cdots\otimes 
(XY)^{n}_c\Big)\\
&=& q^{\otimes n}\Big(\sum_{r=0}^n\sum_{\sigma\in 
\textup{Shuf}(r,n-r) }\sigma\big(\sum_{g\in 
\lab_r(X)}\langle X,g\rangle\Delta(X,g)\otimes 
\sum_{f\in \lab_{n-r}(Y)}\langle 
Y,f\rangle\Delta(Y,f)\big)\Big)\\
&=& q^{\otimes n}\Big(\sum_{r=0}^n \Delta^{r-1}
(X)\bullet \Delta^{n-r-1}(Y)\Big)\\
&=& \sum_{r=0}^n q^{\otimes r}\circ \Delta^{r-1}
(X)\bullet q^{\otimes n-r}\circ \Delta^{n-r-1}
(Y),
\end{eqnarray*}
which completes the proof.
\end{proof}

To show that $\mathcal{J}$ is a coalgebra morphism 
on the Poisson bialgebra, we start
with the following map
for a fixed $X\in\na$ of 
the form (\ref{form}) and two fixed natural 
numbers $m,$ $n$:
\begin{align}
\lab_{m+n}(X) \longrightarrow 
\bigsqcup_{\substack{(I,\phi,c)\in 
\lab_2(X)}}\lab_m(X_c^1)\times\lab_n(X_c^2),
\label{fomrula:Jiscoalgebraic}
\end{align}
and its inverse.
These maps turn out to give a bijection after 
slight modifications.
We will write elements on the right hand side
of \eqref{fomrula:Jiscoalgebraic}
in the form $((I_1,\phi_1,c_1),
(I_2,\phi_2,c_2))_{(I,\phi,c)}$.

We first construct a tuple $((I_1,\phi_1,c_1),
(I_2,\phi_2,c_2))_{(I,\phi,c)}$ in the codomain of 
(\ref{fomrula:Jiscoalgebraic}) from a $(m+n)$-coloring 
$(I_d,\phi_d,d)$ of $X$.
Let $(I,\phi,c)$ be
\begin{equation}\label{2c}
\begin{aligned}
c(x) & \coloneqq \begin{cases}1, 
& \textup{if}\,
d(x)\leq m,\\
2, & \textup{if}\,
d(x)> m,\end{cases}\text{ for any $x\in 
P_X\sqcup V_X$},\\
I & \coloneqq\{a_{i,j},a_{i,j}^*\in I_d\mid 
c(a_{i,j})\neq c(a_{i,j}^*)\},
\text{ and }\phi\coloneqq \phi_d|_I,
\end{aligned}
\end{equation}
and let
the $3$-tuples 
\begin{align*}
(I_1,\phi_1,c_1)&\coloneqq (I_d\cap 
P_{X_c^1},\phi_d|_{I_d\cap 
P_{X_c^1}},d|_{X_c^1})\quad\text{ and }\\
(I_2,\phi_2,c_2)&\coloneqq(I_d\cap 
P_{X_c^2},\phi_d|_{I_d\cap 
P_{X_c^2}},d|_{X_c^2}-m)\end{align*} be 
the $m$-coloring on $X_c^1$ and an $n$-coloring on $X_c^2$ respectively. 
Here, if a full cutting orbit $O_x\subseteq I$ of $(I,\phi)$ produces a vertex component $v_{O_x},$ the restrictions are understood as $c_1(v_{O_x})=d(x)$ if $c(x)=1,$ $c_2(v_{O_x})=d(x)-m$ if $c(x)=2.$
Since $\phi=\phi_d|_I,$ the coloring rule implies that $d$ is constant on $O_x,$ so these values are independent of the choice of $x.$

Conversely, 
given a tuple $((I_1,\phi_1,c_1),
(I_2,\phi_2,c_2))_{(I,\phi,c)}$,
put
\begin{equation}\label{constr:inverse}
I_d:=I\sqcup I_1\sqcup I_2\quad\mbox{and}\quad 
\phi_d:=\phi\sqcup \phi_1\sqcup \phi_2.
\end{equation}
It remains to construct a coloring map $d$ on 
$P_X\sqcup V_X\subset P_{X_c^1}
\sqcup V_{X_c^1}\sqcup P_{X_c^2}\sqcup 
V_{X_c^2}\sqcup I$.
We do this case by case:
\begin{itemize}[noitemsep,leftmargin=*,nosep]

\item If $x\in P_{X_c^1}\sqcup V_{X_c^1}$, 
then $d$ is set to $d(x)=c_1(x)$; 

\item If $x\in P_{X_c^2}\sqcup V_{X_c^2}$, 
then $d$ is set to $d(x)=c_2(x)+m$; 

\item If $x\in I$, define $x'$ as in (15), using the cutting permutation $f$ associated with $(I,\phi)$. 
Namely, $x'$ is the first edge outside $I$ reached by iterating $f$ from $x$; if the entire orbit lies in $I$, then $x'$ is the vertex component labeled by that orbit. 
Then put \[d(x)=\begin{cases}c_1(x'),&c(x)=1,\\m+c_2(x'),&c(x)=2.\end{cases}\]
\end{itemize}

The two maps constructed above lead to a bijection in the lemma below.

\begin{lemma}\label{bij}
Let $X\in \na$ be of the form \text{\rm{(\ref{form})}}, and let $m,n\geq 0$ be integers.
Then the restriction of the map~\eqref{fomrula:Jiscoalgebraic} to $\lab^0_{m+n}(X)$ is a bijection
\begin{equation}\lab^0_{m+n}(X) \overset{\sim}{\longrightarrow}\bigsqcup_{\substack{(I,\phi,c)\in \lab_2(X)\\|c|/2=-||c||}}\lab^0_m(X_c^1)\times \lab^0_n(X_c^2),\label{4.5}\end{equation}
whose inverse is given above.
Here the colorings with $I=\varnothing$ are included in the union.
\end{lemma}

For the proof, we need the following.
\begin{lemma}\label{lemma:coeffcompute}
For the map
\eqref{fomrula:Jiscoalgebraic},
let $\langle X,c\rangle$ be
the coefficient of $\Delta(X)$ given
by Definition~\text{\rm{\ref{coproduct}}}. 
Then there is an equation
\begin{equation}\label{eq:coeff}
\langle X,d\rangle=\langle 
X_c^1,c_1\rangle\langle 
X_c^2,c_2\rangle\langle X,c\rangle.
\end{equation}
\end{lemma}

\begin{proof}
By Definition \ref{coproduct}, we have
\[\begin{aligned}
||c|| = |X| - (|X_{c}^1| + |X_{c}^2|),\,\,||c_1|| = |X_c^1| - 
\sum |(X_{c}^1)_{c_1}^i|,\,\,
||c_2|| = |X_c^2| - \sum |(X_c^2)_{c_2}^i|.
\end{aligned}\] 
Adding them together, by
$({X}_{c}^1)_{c_1}^i=({X})_{d}^i$ and 
$({X}_{c}^2)_{c_2}^i=({X})_{d}^{i+m}$, 
we get
\begin{equation}\label{||||}
||c||+||c_1||+||c_2||=  |{X}|- 
\sum |({X}_{c}^1)_{c_1}^i|- \sum |
({X}_{c}^2)_{c_2}^i|=||d||.
\end{equation}
By Definition \ref{coproduct}  and \eqref{2c}, 
we have
\begin{equation}\label{||}
\begin{aligned}
|d| &= 
|I_d|=|I_1\sqcup I_2\sqcup I|=|c|+ |c_1| 
+ |c_2|, \\
|d_{-}| &= |c_{-}| + |(c_1)_{-}| + |{(c_2)}_{-}|.
\end{aligned}
\end{equation}
The equation \eqref{eq:coeff} follows from the definition of $\langle X,c\rangle$ in Definition~\ref{coproduct} for any coloring $c$.
\end{proof}

\begin{proof}[Proof of Lemma
\text{\rm{\ref{bij}}}]
We first claim that the restriction of \eqref{fomrula:Jiscoalgebraic} on $\lab^0_{m+n}(X)$ (resp. its inverse) has images in the codomain (resp. the domain) of \eqref{4.5}.
For $(I_1,\phi_1,c_1),(I_2,\phi_2,c_2)$ and 
$(I,\phi,c)$ satisfying
$$|c_1|/2=-||c_1||, \quad
|c_2|/2=-||c_2||\quad\mbox{and}\quad 
|c|/2=-||c||,$$ 
we have by (\ref{||||}) and (\ref{||}) the equation
$|d|/2=-||d||$, which means the coloring 
$(I_d,\phi_d,d)$ constructed 
above lies in 
$\lab_{m+n}^0(X)$. Conversely, given $(I_d,\phi_d,d)$ 
satisfying $|d|/2=-||d||$, 
the power of $h$ in $\langle X,d\rangle$ is $0$. 
Thus, 
from
\eqref{eq:coeff} 
we get that the powers of $h$ in 
$\langle X,c\rangle,\langle X_c^1,c_1\rangle$ and 
$\langle X_c^2,c_2\rangle$  are all $0$ since the powers are nonnegative. 
Hence $$|c_1|/2=-||c_1||,\quad
|c_2|/2=-||c_2||\quad\mbox{and}\quad 
|c|/2=-||c||.$$
Since $({X}_{c}^1)_{c_1}^i=({X})_{d}^i$ and 
$({X}_{c}^2)_{c_2}^j=({X})_{d}^{j+m}$, we 
get that $({X}_{c}^1)_{c_1}^i$ and $({X}_{c}^2)_{c_2}^j$ 
have exactly one component for all $i,j$ if and only if 
$(X)_d^k$ has exactly one component for all $k$.
Summarizing the above arguments, our first claim follows.

We show the bijectivity of \eqref{4.5} by proving the two compositions of the map \eqref{4.5} and its inverse to be the identity.
Given
an $(m+n)$-coloring $(I_d,\phi_d,d)$ for $X$ in 
$\lab_{m+n}^0(X)$, 
by \eqref{2c}
we get three colorings $(I,\phi,c),$ $(I_1,\phi_1,c_1)$, and 
$(I_2,\phi_2,c_2)$.
Then by \eqref{constr:inverse}, we get a $(m+n)$-coloring
$(I_{d'},\phi_{d'},d')$.
We show 
$(I_d, \phi_d, d)=(I_{d'}, \phi_{d'}, d')$.
Due to
$P_X=P_{X_c^1}\sqcup  P_{X_c^2}\sqcup I$, we have $I_{d'}=(I_d\cap P_{X_c^1})\sqcup (I_d\cap 
P_{X_c^2})\sqcup I$ and hence $I_{d'}=I_d$. 
Similarly
$\phi_{d'}=\phi_d$. 
The value of $d'$ on $I$ only depends on $I,$ $\phi,$ the heights of edges in $I$, and $c_1$, $c_2$. Hence $d'$ is the only function on $I_d$ identical to $c_1$ on $X_c^1$ and identical to $c_2+m$ on $X_c^2$. We have $d'=d$.

On the other hand, given three tuples $(I,\phi,c)$, $(I_1,\phi_1,c_1)$, and $(I_2,\phi_2,c_2)$,
after getting $(I_d,\phi_d,d)$ by \eqref{constr:inverse}, we get
three colorings
$(I',\phi',c'),$ $(I_1',\phi_1',c_1'),$ and 
$(I_2',\phi_2',c_2')$ by \eqref{2c}.
By construction, $d(x)\leq m$ precisely when $c(x)=1,$ both on surviving occurrences $(P_X\setminus I)\sqcup V_X$, and on the cutting set $I$ by the orbit rule. 
Since the recovered coloring $c'$ is defined by these same two color ranges, we have $c'=c$.
Hence $c'=c$. 
We show $I=I'$. 
For any pair of edges $a,a^*\in I$, 
exactly one of $d(a)$ or $d(a^*)$ satisfies $\leq m$, we have 
$I\subset I'$ by (\ref{2c}). 
For any pair of edges $a,a^*\in I'$, we may write $d(a)\leq m$ and $d(a^*)>m$. 
To show $a, a^*\in I$, assume conversely that one of $a,a^*\notin I$, say $a\notin I$.
As $a\in  I'\subset I_d=I\sqcup I_1\sqcup I_2$, we have $a\in I_1$. 
But $\phi_1(a) = \phi_{d}(a) = a^* \in I_2$, which is a contradiction.
By $ I=I'$, the equation $I_d=I\sqcup I_1\sqcup 
I_2=I'\sqcup I_1'\sqcup I_2'$ implies $ 
I_1 \sqcup I_2=I_1'\sqcup I_2'$ and hence $I_1 = I_1'$, $I_2 = I_2'.$
The equations for $\phi$, $\phi_1$, $\phi_2$, $c_1$, and $c_2$ follows straightforwardly from \eqref{2c} and \eqref{constr:inverse}.
\end{proof}

Now we are ready to show the following.

\begin{proposition}\label{J shi bialgebra mophiam}
Let $\mathcal{J}$ be given by \eqref{j}.
Then $\mathcal J$ 
is a bialgebra morphism.
\end{proposition}

\begin{proof}
By Lemma \ref{buhuizaiyong1}, $\mathcal{J}$ is algebra 
morphism. It suffices to show 
\begin{equation}\label{eq:J}
(\mathcal{J}\otimes\mathcal{J})
(\Delta(X))=\diag(\mathcal J(X)),
\end{equation}
for any 
$X\in \na$. 
Recall that the coproduct on $\mathcal{F}$ is given by
\[
\diag(a_1\otimes\cdots\otimes 
a_m)=\sum_{i=0}^m(a_1\otimes\cdots\otimes 
a_i)\otimes(a_{i+1}\otimes\cdots\otimes a_m).
\]
Thus 
\[(\mathcal{J}\otimes\mathcal{J})
(\Delta(X))=\sum_{c\in \lab_2(X)}\pi
(\langle X,c\rangle)\mathcal{J}(X_c^1)\otimes\mathcal{J}
(X_c^2),
\] 
where $\pi$ is the projection $k[h,\hbar] \to k[\hbar].$
We have $\pi(\langle X,c\rangle)\neq0$ if and only if  $|c|/2=-||c||$.
Thus the left hand side of \eqref{eq:J} equals
\[\sum_{|c|/2=-||c||}\langle X,c\rangle(\mathcal{J}(X_c^1)\otimes\mathcal{J}(X_c^2)).\] 
From the definition of $\mathcal{J}$ in \eqref{buhuizaiyong}, we have
\[\begin{aligned}
\mathcal{J}(X_c^1)
&=\sum_mq^{\otimes m}\circ \Delta^{m-1}
(X_c^1)=\sum_mq^{\otimes m} \sum_{c_1\in 
\lab_m(X_c^1)}\langle X_c^1,c_1\rangle 
\Delta(X_c^1,c_1)\\
&=\sum_m\sum_{c_1\in 
\lab_m^0(X_c^1)}\langle X_c^1,c_1\rangle\langle 
\Delta(X_c^1,c_1)\rangle,\\
\mathcal{J}
(X_c^2)
&=\sum_n\sum_{c_2\in 
\lab_n^0(X_c^2)}\langle X_c^2,c_2\rangle\langle 
\Delta(X_c^2,c_2)\rangle.
\end{aligned}
\]
Here the notations $\langle 
\Delta(X_c^1,c_1)\rangle$ and $\langle 
\Delta(X_c^2,c_2)\rangle$ are defined by putting $\langle X_1\otimes X_2\otimes\cdots\otimes X_n\rangle:=\langle X_1\rangle\otimes\langle X_2\rangle\otimes\cdots\otimes\langle X_n\rangle$. 
By Lemmas~\ref{bij} and \ref{lemma:coeffcompute},  
we only need to take summations in $\lab^0_{m+n}$ and 
hence
the left hand side of \eqref{eq:J} equals 
\begin{align}
\sum_{m,n\geq 0}\sum_{b\in 
\lab_{m+n}^0(X)}\langle X,b\rangle \left(\langle 
X_b^1\rangle\otimes\cdots\otimes \langle 
X_b^{m}\rangle\right)\otimes \left(\langle 
X_b^{m+1}\rangle\otimes\cdots\otimes \langle 
X_b^{m+n}\rangle\right)
\label{eq:lcoprod}
\end{align} 
On the other hand, 
for the right hand side of \eqref{eq:J}, we have
\begin{align}
&\diag(\mathcal{J}
(X))\nonumber\\
&=
\diag\Big(\sum_l\sum_{c\in 
\lab_{l}^0(X)}\langle X,c\rangle \langle 
X_c^1\rangle\otimes \cdots \otimes\langle  
X_c^l\rangle\Big)\nonumber\\
&=
\sum_{k\leq l} \sum_{c\in 
\lab_{l}^0(X)}\langle X,c\rangle \left(\langle X_c^1\rangle 
\otimes \cdots 
\otimes\langle  
X_c^{l-k}\rangle\right)
\otimes\left(\langle 
X_c^{l-k+1}\rangle\otimes \cdots \otimes \langle 
X_c^{l}\rangle\right)\nonumber \\
&=
\sum_{m,n} \sum_{c\in 
\lab_{m+n}^0(X)}\langle X,c\rangle \left(\langle X_c^1\rangle 
\otimes \cdots \otimes \langle 
X_c^{m}\rangle\right)\otimes\left(\langle X_c^{m+1}\rangle\otimes 
\cdots \otimes \langle X_c^{n+m}\rangle\right) .
\label{eq:rcoprod}
\end{align}  

Comparing \eqref{eq:lcoprod}
and \eqref{eq:rcoprod}, the desired equation follows.
\end{proof}

\subsubsection{$\mathcal{J}$ is a quantization
of $E(L_\hbar)$}\label{section:JisaquantizationofELhbar}
We check the rest conditions for quantization. We prepare some notations.

\begin{notation}\label{sign}
    Let $P_X$ and $P_Y$ denote the edge sets of elements $X$ and $Y\in \na$, respectively. Then  the set of \textit{crossing edges} $T_{X,Y}$ is given as below and put $T_{X } :=  T_{X , X}$. 
\[\begin{aligned}
    T_{X, Y} := \{ 
    a_{i,j} \in P_X,\ b_{k,l} \in P_Y \mid a_{i,j} = b_{k,l}^* \},
\end{aligned}
\]
\end{notation}

\begin{lemma}\label{map}
Let 
$p_{\hbar}:=\mathrm{symm}\circ \mathcal J$.
Then
$p_\hbar(\mathbf N(Q)_{h,\hbar})= \varepsilon_\hbar(L_\hbar)$.
\end{lemma}

\begin{proof}
For any knot $X\in \na$, we have $\Delta^0(X)=X$ and
\[
\Delta(X)=1\otimes X+X\otimes1+
\textstyle\sum_i u_i\otimes v_i\text{ with }u_i\text{ or }v_i\in h\na+\hbar\na.\]   
Note that $|T_{u_i}|+|T_{v_i}|<|T_X|$. 
For \begin{align}\Delta^{k-1}(X)
=
\sum _{i=0}^{k-1}1^{\otimes i}\otimes X\otimes  
1^{\otimes (k-i-1)}+\sum_i u_i^1\otimes\cdots\otimes 
u_i^k,\label{formula:k-1coproduct}\end{align}
we have $\sum_{l=1}^k|T_{u_i^l}|<|T_X|$ for all $i$. 

For every vertex $v\in Q_0$, we have $\Delta(v)=1\otimes v+v\otimes1$, and hence $p_\hbar(v)=v$, as $q(1)=0$.
Let $\alpha$ be a cyclic path in $L$, and let $X_\alpha$ be a knot lifting $\alpha.$ 
Put $n=|T_{X_\alpha}|.$
Notice $q(1)=0$.
With the above computation of $\Delta^{k-1}(X)$ for $k\geq 1$, the definition of the maps $\mathrm{symm}$ and $\mathcal{J}$ implies $p_\hbar(X_\alpha)=\alpha+\sum _iu_i$ with $u_i\in \hbar \varepsilon_\hbar(L_\hbar)$ being a polynomial 
over $k[\hbar]$ 
and $|T_{u_i}|\leq n-1$.
Note that $|T_{X_\alpha}|$ only depends on $\alpha$. 
We prove the lemma by induction on $|T_{X_\alpha}|$. 
If $|T_{X_\alpha}|=0$, then $u _i= 0$ and $p_\hbar(X_\alpha)=\alpha$. 
Suppose $\alpha\in \im p_\hbar$ if 
$|T_{X_\alpha}|\leq n-1$. 
For any $\alpha\in L\subset \varepsilon_\hbar(L_\hbar)$ with 
$|T_{X_\alpha}|=n$, we have $p_\hbar(X_\alpha)=\alpha+\sum_i u_i$. 
Every cyclic-path factor occurring in each $u_i$ has fewer than $n$ crossing edges. 
By the induction hypothesis, including the vertex case, these factors lie in $\mathrm{im}\,p_\hbar$. 
Multiplicativity therefore gives $\sum_i u_i\in\mathrm{im}\,p_\hbar,$ and hence $\alpha\in\mathrm{im}\,p_\hbar.$
We have the desired 
$p_\hbar(\mathbf N(Q)_{h,\hbar})
= \varepsilon_\hbar(L_\hbar)$ by the multiplicativity of $p_\hbar$.
\end{proof}


\begin{proposition}\label{surj}
Let $\mathcal{J}$ be as in  
\eqref{j}.
Then $\mathcal J(\mathbf 
N(Q)_{h,\hbar})= E(L_\hbar)$.
\end{proposition}

\begin{proof}
Put $\mathcal J_n=q^{\otimes n}\circ\Delta^{n-1}.$ Definition~\ref{coproduct} and the cobracket formula \eqref{formula:cobracket} give $(q\otimes q)\circ(\Delta-\operatorname{Perm}\circ\Delta)=\hbar\,\nu\circ q.$
Indeed, a straighforward check shows the equation for knots and zero for $\geq 2$ components.
By coassociativity, applying this identity in adjacent tensor positions yields $(1-\sigma_i)\mathcal J_n=\hbar\,\nu_i\mathcal J_{n-1}$ for $1\le i<n.$
This checks the defining relations of $E(L_\hbar)$ in \eqref{formula:defE}, and hence $\mathcal J(\mathrm N(Q)_{h,\hbar})\subseteq E(L_\hbar).$
The equality follows from Lemma~\ref{map} and Corollary~\ref{E}.
\end{proof}

The following proposition establishes
(\ref{formula:quantization}) in the definition of quantization.
Its proof is postponed to Section~\ref{sec:quantizationpart2}.

\begin{proposition}\label{quantization}
For 
elements \(X\) and \(Y\) of the form \eqref{form}, we have
\begin{equation}
\nonumber
X*Y-Y*X=h\mathcal{J}^{-1}
\{\mathcal{J}(X),\mathcal{J}(Y)\}\mod 
h\ker\mathcal{J},
\end{equation}
where \(\{\cdot,\cdot\}\) is the Poisson bracket on 
the Poisson bialgebra $E(L_\hbar)$.
\end{proposition}

\begin{proof}[The quantization part in Theorem~\ref{quantization varepsilon}]
By Corollary~\ref{E}, the map
$
\symm|_{E( L_\hbar)}:
E( L_\hbar)\to\varepsilon_\hbar(L_\hbar)$ is
an isomorphism of Poisson bialgebras.
By Proposition~\ref{quantization}, the map $\mathcal{J} \colon 
\mathbf{N}(Q)_{h,\hbar} \to E( L_\hbar)$ is a 
quantization.
Thus their composition 
$p_\hbar = \symm|_{E( L_\hbar)} 
\circ \mathcal{J}$
is also a quantization.
\end{proof}
\subsubsection{The quantization $p_\hbar$ is reduced}\label{reducon}

To finish the proof of Theorem~\ref{quantization varepsilon}, we show that the quantization $p_\hbar$ is reduced, i.e., $\ker p_\hbar=h \mathbf{N}(Q)_{h,\hbar}$.
This equation is a corollary
of the lemma below.
Note that the map $\mathcal{J}$ and hence the map $p_\hbar$ is $\pi$-linear with $\pi:k[h,\hbar]\to k[\hbar].$
The map $p_\hbar$ naturally induces $\widetilde{p}_\hbar:\mathbf{N}(Q)_{h,\hbar}/h\mathbf{N}
(Q)_{h,\hbar} \to \varepsilon_\hbar(L_\hbar)$.

\begin{lemma}\label{xiq-1=0}
There exists an algebra homomorphism $\xi :\varepsilon_\hbar(L_\hbar) \to 
\mathbf{N}(Q)_{h,\hbar}/h\mathbf{N}(Q)_{h,\hbar}$
such that for any element 
$X\in \na$,
\[
(\xi \circ \widetilde{p}_\hbar - \id)^{n+1}([X]) = 0,
\quad\text{ for some $n\in\mathbb N$,}
\]
where $[X]\in \na/h\na$ denotes the equivalence
class of $X$. 
\end{lemma}

\begin{proof} 
Recall that $\varepsilon_\hbar(L_\hbar) = S(L_\hbar)$ as algebras.
Notice that the algebra $\mathbf{N}(Q)_{h,\hbar}/h\mathbf{N}(Q)_{h,\hbar}$ is commutative by (\ref{quiver skein relations 1}).
We are to construct a $k[\hbar]$-linear map $\xi: L_\hbar \to \na/h\na$ and check that the induced algebra morphism on $\varepsilon_\hbar(L_\hbar)$, again denoted $\xi$, is the desired map.

For any cyclic path $\alpha \in  L$, let $X_\alpha\in\na$ be any lifting of $\alpha$ and $\xi$ be the $k[\hbar]$-linear map given by $\alpha \mapsto X_\alpha$.
We show that the extension of $\xi$ to $\varepsilon_\hbar(L_\hbar)$ is the map.
Let $F_d$ be the $k[\hbar]$-span of classes of links having at most $d$ arrows, with vertex components contributing zero to this count, and put $F_d=0$ for $d<0$. By the proof of Proposition~\ref{formmodh}, $F_d$ is spanned by products of knots and vertex components whose total arrow length is at most $d$.
For a knot $Y$ with $d$ arrows, the expansion in the proof of Lemma~\ref{map} gives $(\xi\circ\widetilde p_\hbar)([Y])-[X_{\langle Y\rangle}]\in F_{d-2}.$
Indeed, every correction term resolves at least one pair of arrows, and the chosen lifts preserve arrow length. Exchanging adjacent heights using \eqref{quiver skein relations 2} also gives $[X_{\langle Y\rangle}]-[Y]\in F_{d-2}.$
Moreover, $\xi\circ\widetilde p_\hbar$ fixes every vertex component and the unit.

Put $D=\xi\circ\widetilde p_\hbar-\mathrm{id}$. Since $\xi\circ\widetilde p_\hbar$ is multiplicative,
\[D(ab)=D(a)b+aD(b)+D(a)D(b).\]
Applying this identity to products of knots and vertex components yields $D(F_d)\subseteq F_{d-2}.$
Consequently, $D^{\lfloor d/2\rfloor+1}(F_d)=0.$
Every element belongs to some $F_d$, which proves the lemma.
\end{proof}

\begin{proof}[The reducedness part in Theorem~\ref{quantization varepsilon}]
The inclusion $h\mathbf{N}(Q)_{h,\hbar} \subset \ker p_\hbar$ is straightforward.
It suffices to show that $\widetilde{p}_\hbar$ is injective. 
Suppose $\widetilde{p}_\hbar([X]) = 0$ for some $[X] \in \mathbf{N}(Q)_{h,\hbar}/h\mathbf{N}(Q)_{h,\hbar}$.
We have 
\[(\xi \circ \widetilde{p}_\hbar - \id)([X]) = \xi \circ \widetilde{p}_\hbar([X])-[X] = 0 - [X] = -[X].\]
Applying $(\xi \circ \widetilde{p}_\hbar - 1)^n$ to both sides of this equality, 
by Proposition~\ref{xiq-1=0}
we have  
\[
0 = (\xi \circ \widetilde{p}_\hbar - \id)^{n+1}([X]) 
= (\xi \circ \widetilde{p}_\hbar 
- \id)^n(-[X]) = (-1)^{n+1} [X],
\]
which implies $[X] = 0$.  
Thus $\widetilde{p}_\hbar$ is injective.
\end{proof}

\subsection{Proof of quantization, part 2}\label{sec:quantizationpart2}

The goal of this subsection
is to show Proposition~\ref{quantization}.
By Proposition~\ref{remark:general form2}, 
it suffices to show the following equation for any 
two knots \(X\) and \(Y\):
\begin{equation}
\label{eq:Jofcommutator}
X*Y-Y*X=h\mathcal{J}^{-1}
\{\mathcal{J}(X),\mathcal{J}(Y)\}\mod 
h\ker\mathcal{J}.
\end{equation}
Indeed, multiplicativity of $\mathcal J$ and the Leibniz rules for the commutator and Poisson bracket extend \eqref{eq:Jofcommutator} to products, while both sides vanish when either argument is a vertex. Moreover, for arbitrary $W,Y\in \mathrm{N}(Q)_{h,\hbar},$ we have $\mathcal J(hW)=0$ and $h(W*Y-Y*W)\in h\ker\mathcal J,$ so replacing either argument by a congruent element modulo $h$ leaves \eqref{eq:Jofcommutator} unchanged; Proposition~\ref{formmodh}(2) and bilinearity therefore justify the reduction.

We divide the proof of \eqref{eq:Jofcommutator} into 
five steps.
We show that
the left hand side consists of
two parts: 
the first part, called the primitive part, 
will be shown to be equal to the right 
side;
the rest, called the non-primitive part, 
will be proved to be zero.
The proof 
is inspired by  
\cite[Section~13]{Turaev}.
\vspace{1mm}

\noindent\textbf{Step 1: Rewrite 
$X*Y-Y*X$ as the sum $h\sum{\varepsilon_i}N^{i,j}$ of knot 
terms.}
\vspace{1mm}

Retain Notation~\ref{sign}. 
Let $e_{1},\ldots,e_\alpha$ be edges of $X$ such that $\{e_1,\ldots,e_\alpha\} = T_{X, Y}\cap P_X$, enumerated in decreasing order of their heights in $X.$
For each $e_i$, let $e^*_{i,1},\ldots,e^*_{i,k_i}$ be edges in $T_{X,Y}\cap P_Y$ such that all $e_{i,j}^*$ are the reserve of $e_i$ in $\overline{Q}$ 
and all $e_{i,j}^*$ are ordered so that: the height of $e_{i,j-1}^*$ $<$ the height of $e_{i,j}^*$.
Put \begin{align}
\label{varepsilon_i meishayongde}
\varepsilon_{i}\coloneqq 
\begin{cases}
1,& \textup{ if } e_{i}\in Q
\cap{T_{X, Y}}\cap P_X,\\
-1,& \textup{ if } e_{i}
\in Q^*\cap T_{X, Y} \cap P_X.
\end{cases}
\end{align}
If $\varepsilon_i = -1,$ then $e_{i,j}^* \in Q$ for any $j$ but has height larger than that of $e_{i}$ (cf. \eqref{formula:c-}).
Let \((X*Y)^{i,j}\) be the element obtained from \(X*Y\) following the sequential height exchange below using the skein relations~\eqref{quiver skein relations 1} while preserving the original height orders within \(X\) and \(Y\):
\begin{enumerate}[label = \rm{(}\arabic*),noitemsep,leftmargin=*,nosep]
    \item First, sequentially exchange the height of \(e_l\) with those of \(e^*_{l,1}, e^*_{l,2}, \dots, e^*_{l,k_l}\) for each \(l = 1, \dots, i-1\). Note if there is an edge, say $e$, that has height between $e_{l}$ and $e^*_{l,1}$, then one exchanges the heights of $e_l$ and $e$ without producing any new summand by \eqref{quiver skein relations 1};
    \item Second, sequentially exchange the height of \(e_i\) with those of \(e^*_{i,1}, e^*_{i,2}, \dots, e^*_{i,j}\). 
\end{enumerate}
Put \((X*Y)^{1,0} = X*Y\) and \((X*Y)^{i,0} = (X*Y)^{i-1,k_{i-1}}\) for $i\geq 2$.
By \eqref{quiver skein relations 1}, we have
\begin{equation}\label{N^ij}
    (X*Y)^{i,j-1}-(X*Y)^{i,j}=h\varepsilon_{i}N^{i,j}\text{ for }j=1,\ldots,k_i
\end{equation}
where \(N^{i,j}\) denotes the knot obtained from \(X*Y\) by first resolving the crossing edges \(e_i\) and \(e^*_{i,j}\) following relation~\eqref{quiver skein relations 1} and then sequentially exchanging the heights of \(e_l\) with those of  \(e^*_{l,1},e^*_{l,2},\dots,e^*_{l,k_l}\) for each \(l = 1, \dots, i-1\), while preserving the original height order within \(X\) and within \(Y\). 
Hence each $N^{i,j}$ has one component; if no arrow remains, this component is the resulting vertex.
Now observe the telescoping sum:
\begin{align}
X*Y - Y*X 
&= \bigl(X*Y-(X*Y)^{1,1}\bigr) + \bigl((X*Y)^{1,1}-(X*Y)^{1,2}\bigr) + \cdots\nonumber \\
&\quad + \bigl((X*Y)^{1,k_1-1}-(X*Y)^{1,k_1}\bigr) + \bigl((X*Y)^{1,k_1}-Y*X\bigr)\nonumber\\
&=
\displaystyle
h\sum_{j=1}^{k_1}\varepsilon_1 N^{1,j} + \bigl((X*Y)^{1,k_1}-Y*X\bigr)\nonumber\\
&=
\displaystyle
h\sum_{\substack{1\leq i\leq \alpha\\ 
1\leq j\leq k_i}} \varepsilon_i N^{i,j} 
+ \bigl((X*Y)^{\alpha,k_\alpha} 
- Y*X\bigr)\nonumber \\
&=
\displaystyle
h\sum_{\substack{1\leq i\leq \alpha\\ 
1\leq j\leq k_i}} \varepsilon_i N^{i,j}.
\label{formula:X*Y-Y*X}
\end{align}
If $T_{X,Y} = \varnothing$, the sums are empty and both sides are zero.
This finishes Step~1.
By \eqref{eq:Jofcommutator}, it remains to show
   \begin{equation}\label{8}
          \mathcal{J}\left(\textstyle\sum_{i,j}{\varepsilon_i}N^{i,j}\right)=\{\mathcal{J}(X),\mathcal{J}(Y)\}.
   \end{equation}

\vspace{1mm}
\noindent\textbf{Step 2: Expand $\{\mathcal{J}(X),\mathcal{J}(Y)\}$.}
\vspace{1mm}

By \eqref{j}, we compute 
\[\{\mathcal{J}(X),\mathcal{J}(Y)\}
= 
\sum_{m,n\geq 0}\{q^{m}\circ \Delta^{m-1}(X),q^{n}\circ \Delta^{n-1}(Y)\}.\]
By Lemma~\ref{Lie} and \eqref{buhuizaiyong}, we have
\begin{align}
\{\mathcal{J}(X),\mathcal{J}(Y)\}\nonumber
&=\sum_{m,n\geq 0}\sum_{\substack{d\in \mathrm{Col}_m^0(X)\\
d'\in \mathrm{Col}_n^0(Y)}}\sum_{\substack{1\leq k\leq m\\1\leq l\leq n}}\sum_{\substack{\sigma\in \mathrm{Shuf}(k-1,l-1)\\\varsigma\in\mathrm{Shuf}(m-k,n-l)}}
\langle X,d\rangle\langle Y,d'\rangle\\& \hspace{1.3cm}\langle W_{d,d'}^{1,k+l-2}(\sigma)\rangle\otimes[\langle X^k\rangle,\langle Y^l\rangle]_{L_{\hbar}}\otimes \langle W_{d,d'}^{k+l+1,m+n}(\varsigma)\rangle,
\label{formula:Poissonbracketcompute}
\end{align}
with notations explained as follows:
\begin{itemize}[noitemsep,leftmargin=*,nosep]
    \item For colorings $d,d'$ and group elements $\sigma,\varsigma$ as in the equation, if $X^1,\ldots,X^m$ and $Y^1,\ldots,Y^n$ are components given by $d,d'$ and $X,Y$, then we write 
\begin{align*}
   \langle W_{d,d'}^{1,k+l-2}(\sigma)\rangle& \coloneqq\sigma\big(\langle X^1\rangle\otimes \cdots\otimes  \langle X^{k-1}\rangle\otimes \langle Y^1\rangle\otimes \cdots\otimes \langle Y^{l-1}\rangle\big);\\
   \langle W_{d,d'}^{k+l+1,m+n}(\varsigma)\rangle & \coloneqq\varsigma\big(\langle X^{k+1}\rangle\otimes \cdots\otimes  \langle X^{m}\rangle\otimes \langle Y^{l+1}\rangle\otimes \cdots\otimes \langle Y^{n}\rangle\big).
\end{align*}
\item Let $[\langle X^k\rangle,\langle 
Y^l\rangle]_{L_{\hbar}}$ denote 
the Lie bracket 
of $\langle X^k\rangle$ and $\langle Y^l\rangle$ 
in $L_{\hbar}$.
\end{itemize}

We study $[\langle X^k\rangle,\langle 
Y^l\rangle]_{L_{\hbar}}$ given $(I_d,\phi_d,d)\in 
\lab_m^0(X)$ and $(I_d',\phi_d',d')\in 
\lab_n^0(Y).$
Let $s_i$ and $s^*_{i,j}$ be
the integers such that $e_i\in \langle 
X^{s_i}\rangle$ and $e_{i,j}^* \in \langle 
Y^{s^*_{i,j}}\rangle$ for any $i,j$. 
If $e_{i}\in I_d$ or $e_{i,j}^*\in I_{d'}$ for some $i,j$, then $e_{i}$ does not appear in $X^{s_i}$ or $e_{i,j}^*$ does not appear in $Y^{s^*_{i,j}}.$
In this case, we regard that $s_i$ or $s_{i,j}^*$ do not exist.
Then 
by Theorem~\ref{L}
we have 
\[[\langle X^k\rangle,\langle 
Y^l\rangle]_{L_{\hbar}} 
\coloneqq \sum_{\{i\,\mid\,s_i = k\}}
\sum_{
\{j\,\mid\,s_{i,j}^*=l\}}\varepsilon_i[\langle 
X^k\rangle,\langle 
Y^l\rangle]_{(e_{i},e_{i,j}^*)},\]
where the lower index $(e_{i},e_{i,j}^*)$ means 
computing the Lie bracket in $L_\hbar$ by pairing $e_{i}$ and $e_{i,j}^*$ regardless of the sign (the sign is given by $\varepsilon_i$).
The bracket $[\langle X^k\rangle,\langle Y^l\rangle]_{(e_{i},e_{i,j}^*)}$ is nonzero only if both $s_i$ and $s_{i,j}^*$ exist.
Notice $[\langle X^k\rangle,\langle Y^l\rangle]_{L_{\hbar}} = 0$ if $X^k$ contains no $e_i$ or $Y^l$ contains no $e_{i,j}^*$.
Using the Kronecker delta $\delta$, equation 
(\ref{formula:Poissonbracketcompute}) is 
rewritten as
\begin{align}
 & \{\mathcal{J}(X),\mathcal{J}(Y)\}\nonumber\\
&= \sum_{m,n\geq 0}\sum_{\substack{d\in \mathrm{Col}_m^0(X)\\d'\in \mathrm{Col}_n^0(Y)}}\sum_{i,j}\sum_{\substack{1\leq k\leq m\\1\leq l\leq n}}\sum_{\substack{\sigma\in \mathrm{Shuf}(k-1,l-1)\\\varsigma\in\mathrm{Shuf}(m-k,n-l)}}
\langle X,d\rangle\langle Y,d'\rangle\nonumber \\
& \quad\quad\quad\langle W_{d,d'}^{1,k+l-2}(\sigma)\rangle\otimes\left(\delta_{k}^{s_i}\delta_{l}^{s_{i,j}^*}\cdot \varepsilon_i [\langle X^k\rangle,\langle Y^l\rangle]_{e_i,e_{i,j}^*}\right)\otimes \langle W_{d,d'}^{k+l+1,m+n}(\varsigma)\rangle.
\label{formula:Poissonbracketcompute2}
\end{align}
A summand with an undefined $s_i$ or $s^*_{i,j}$ is interpreted as zero.

\vspace{1mm}
\noindent\textbf{Step 3:  Compute $\mathcal{J}\left(\sum{\varepsilon_i}N^{i,j}\right)=\text{primitive part}+\text{nonprimitive part}$.}
\vspace{1mm}

We compute $\mathcal{J}\left({\varepsilon_i}N^{i,j}\right)$ using the definition of $\mathcal{J}$:
\begin{align}\label{11}   
\mathcal{J}\left(\sum_{i,j}{\varepsilon_i}N^{i,j}\right)
&=\sum_{i,j}{\varepsilon_i}\mathcal{J}(N^{i,j})= \sum_{i,j}\varepsilon_i\bigg (\sum_{p\geq 0 }q^{\otimes p}\circ\Delta^{p-1} (N^{i,j})\bigg)\nonumber\\ 
&= \sum_{i,j}\varepsilon_i\sum_{p\geq 1}\sum_{(I_c,\phi_c,c)\in \lab_p^0(N^{i,j})}\langle N^{i,j},c\rangle \langle N^{i,j}\rangle_c^1\otimes\cdots \otimes  \langle N^{i,j}\rangle_c^p,\\
   &\hspace{3.3cm}\text{(by the definition of $\Delta^{-1}$ in Definition~\ref{coproduct})}\nonumber
\end{align}
where we use $\langle N^{i,j}\rangle_c^k$ to abbreviate $\langle (N^{i,j})_c^k\rangle$.
Note $T_{X*Y}=T_X\cup T_Y\cup T_{X, Y}$. 
Any coloring $(I_c,\phi_c,c)$ in (\ref{11}) are of two types:
\begin{enumerate}[label = \rm{(}\arabic*),noitemsep,leftmargin=*,nosep]
\label{fenlei}
\item[(1)] (``primitive colorings'') If $\varphi_c(I_c\cap P_X)=I_c\cap P_X$ and $\varphi_c(I_c\cap P_Y)=I_c\cap P_Y$, then the coloring $(I_c,\phi_c,c) \in \lab_p^0(N^{i,j})$ is said to be \textit{primitive}. 
The set consisting of all such primitive colorings is denoted by  $\lab_p^{\text{\rm{pr}}}(N^{i,j})$. 
\item[(2)] (``non-primitive colorings'') Otherwise, these colorings are said to be \textit{non-primitive}. 
The set consisting of all such non-primitive colorings is denoted by $\lab_p^{\text{\rm{np}}}(N^{i,j})$. 
\end{enumerate}

\vspace{1mm}
\noindent\textbf{Step 4: Prove  
  the primitive part  of $\mathcal{J}\left(\sum{\varepsilon_i}N^{i,j}\right)$ equals $\{\mathcal{J}(X),\mathcal{J}(Y)\}$.} 
\vspace{1mm}

Comparing  
(\ref{formula:Poissonbracketcompute2}) and (\ref{11}), we fix $p,i,j$ and show
\begin{align}
&\varepsilon_i\sum_{(I_c,\phi_c,c)\in \lab_p^{\text{\rm{pr}}}
(N^{i,j})}\langle N^{i,j},c\rangle \langle 
N^{i,j}\rangle_c^1\otimes\cdots \otimes  
\langle N^{i,j}\rangle_c^p \nonumber \\
&=
\sum_{\substack{m,n\geq 0\\m+n=p+1}}\sum_{\substack{d\in \mathrm{Col}_m^0(X)\\d'\in \mathrm{Col}_n^0(Y)}}\sum_{\substack{1\leq k\leq m\\1\leq l\leq n}}\sum_{\substack{\sigma\in \mathrm{Shuf}(k-1,l-1)\\\varsigma\in\mathrm{Shuf}(m-k,n-l)}}
\langle X,d\rangle\langle Y,d'\rangle\nonumber \\
&\quad\quad\quad \langle W_{d,d'}^{1,k+l-2}(\sigma)\rangle\otimes\left(\delta_{k}^{s_i}\delta_{l}^{s_{i,j}^*}\cdot \varepsilon_i [\langle X^k\rangle,\langle Y^l\rangle]_{e_i,e_{i,j}^*}\right)\otimes \langle W_{d,d'}^{k+l+1,m+n}(\varsigma)\rangle.
\label{formula:primitivepartcompute}
\end{align}
The proof is similar to that of Lemma~\ref{bij}. 
For fixed $p,i,j$, we construct two maps between the set of tuples $(m,n,d,d',k,l,\sigma,\varsigma)$ satisfying the conditions on the right-hand side of \eqref{formula:primitivepartcompute}, with $s_i$ and $s_{i,j}^*$ both defined and $k=s_i,$ $l=s_{i,j}^*$, and the set $\mathrm{Col}_{p}^{\mathrm{pr}}(N^{i,j}))$ of primitive colorings. 
First, we associate a primitive coloring to each such tuple; conversely, we recover a tuple from each primitive coloring. 
We then verify that these two constructions are inverse to each other. 
Comparing the coefficients and the ordered tensor factors associated with each tuple and its corresponding primitive coloring proves \eqref{formula:primitivepartcompute}.

Given any colorings $d = (I_d,\phi_d,d)$, $d' = (I_d',\phi_d',d')$, and group elements $\sigma,$ $\varsigma$, we are to find a primitive coloring $(I_c,\phi_c,c)\in \mathrm{Col}_p^{\mathrm{pr}}(N^{i,j}).$
We may assume that $s_i$ and $s_{i,j}^*$ both exist (that is $e_{i}\notin I_d$ and $e_{i,j}^*\notin I_d'$), and $k=s_i$, $l = s_{i,j}^*$.
Notice $P_{N^{i,j}} = P_X \cup P_Y - \{e_i,e_{i,j}^*\}.$
Put $p = m+n-1,$ $I_{c}\coloneqq I_d \sqcup I_{d'}$, and $\phi_c \coloneqq \phi_{d} \sqcup \phi_{d'}$.
To define $c$, put for $1\le r\le m$ and $1\le s\le n$
\[x_r=\begin{cases}\sigma(r),&r<k,\\k+l-1,&r=k,\\k+l-1+\varsigma(r-k),&r>k,\end{cases}\text{, and }y_s=\begin{cases}\sigma(k-1+s),&s<l,\\k+l-1,&s=l,\\k+l-1+\varsigma(m-k+s-l),&s>l,\end{cases}\]
For every $a\in P_{N^{i,j}}$, put
\[c(a)=
\begin{cases}
x_{d(a)},&a\in P_X\setminus \{e_i\},\\
y_{d'(a)},&a\in P_Y\setminus \{e_{i,j}^*\}.
\end{cases}\]
These formulae apply to all remaining original edges, including those in $I_c.$ If $N^{i,j}$ itself is a vertex $v$, set $c(v)=1$; in this case $m=n=k=l=p=1.$
The sequences $x_1<\cdots<x_m$ and $y_1<\cdots<y_n$ satisfy $\{x_1,\ldots,x_m\}\cup\{y_1,\ldots,y_n\}=\{1,\ldots,p\}$, $\{x_1,\ldots,x_m\}\cap\{y_1,\ldots,y_n\}=\{k+l-1\}.$
Since $d$ and $d'$ are constant on their respective cutting orbits and $x_k=y_l=k+l-1,$ the coloring $c$ is constant on each cutting orbit of $N^{i,j}.$ 
As the sequences $(x_r)$ and $(y_s)$ are strictly increase and have intersection $k+l-1$, each color labels exactly one resulting component, including vertex components.
As the height orders within $X$ and $Y$ are preserved, the coloring inequalities for all pairs in $\phi_c$ follow from those for $d,d'$. 
As $|I_c|/2=|I_d|/2+|I_{d'}|/2=p-1=-||c||,$ the result coloring $(I_c,\phi_c,c)$ is primitive.

Conversely, we are given a primitive coloring $(I_c,\phi_c,c)\in \mathrm{Col}^{\mathrm{pr}}_{p}(N^{i,j}).$
We are to define $m,$ $n,$ $(I_d,\phi_d,d)$, $(I_{d'},\phi_{d'},d')$, $k,$ $l,$ $\sigma,$ $\varsigma$ as in \eqref{formula:primitivepartcompute} for $m+n=p+1$.
Let $m$ and $n$ be the numbers of components obtained by resolving the pairings $(I_d,\phi_d)$ in $X$ and $(I_{d'},\phi_{d'})$ in $Y$, respectively, leaving $e_i,e_{i,j}^*$ uncut. Here vertex components are counted individually and distinguished by their cutting orbits. Resolving the distinguished pair subsequently merges exactly the two components containing $e_i$ and $e_{i,j}^*$, while leaving all other components unchanged. The resulting components are precisely the $p$ components determined by $c$ on $N^{i,j}$ and hence $m+n-1=p.$
Put $I_d \coloneqq I_c \cap P_X$, $I_{d'} \coloneqq I_c \cap P_Y$, $\phi_d \coloneqq \phi_c|_{I_d}$, and $\phi_d' \coloneqq \phi_c|_{I_d'}.$
Assign to each component obtained from $X$ or $Y$ the $c$-color of its corresponding component after resolving the distinguished pair. 
Order the components from $X$ by their inherited colors $x_1<\cdots<x_m$, and similarly order those from $Y$ by $y_1<\cdots<y_n$. Their corresponding components in the resolved $N^{i,j}$ are therefore $(N^{i,j})_c^{x_1},\ldots,(N^{i,j})_c^{x_m}$ and $(N^{i,j})_c^{y_1},\ldots,(N^{i,j})_c^{y_n},$ respectively. 
Both lists include vertex components. 
The distinguished merged component appears in both lists, even when it contains no surviving edge from one or both of $X,Y$; every other component appears in exactly one list.
Then we put 
\[d:P_{X}\to \{1,\ldots,m\};
\quad d(a) = \begin{cases}r, 
&\textup{if } a\in (N^{i,j})_c^{x_r},\\
r', &\textup{if } a\in I_d\text{ and }c(a') = x_{r'},\end{cases}\]
where we define $c(a')$ analogously to $c^1(x')$ in \eqref{formula:coass1}.
Following the same way, we define $d'.$
Let $\bar p$ be the unique common color in the two lists $\{x_1,\ldots,x_m\}\cap\{y_1,\ldots,y_n\}=\{\bar p\}.$
Let $k,l$ be the unique indices satisfying $x_k=y_l=\bar p$ such that $d(e_{i}) = k$ and $d'(e_{i,j}^*) = l.$
Finally, let $\sigma$ be the shuffle in $\mathrm{Shuf}(k-1,l-1)$ such that $\sigma(r) = x_r$ for $r\leq k-1$ and $\sigma(r) = y_{r-k+1} $ for others.
The element $\varsigma$ is defined in the same way.
Similarly, the summand corresponding to $d,d',\sigma,\varsigma$ equals the one corresponding to $(I_c,\phi_c,c).$ 

Now we check that the above two constructions give the desired bijection.
Given a primitive coloring $(I_c,\phi_c,c)\in \mathrm{Col}^{\mathrm{pr}}_{p}(N^{i,j}),$ we get $m,$ $n,$ $(I_d,\phi_d,d)$, $(I_{d'},\phi_{d'},d')$, $k,$ $l,$ $\sigma,$ $\varsigma$, and then $(I_{c'},\phi_{c'},c')\in \mathrm{Col}^0_{p}(N^{i,j})$. The coloring $(I_{c'},\phi_{c'},c')$ is primitive since $I_{c'}=I_d\sqcup I_{d'}$ where  $I_d\subset P_X$ and $I_{d'}\subset P_Y$. And we get $I_c=I_{c'}$ since  $I_{c'}=I_d\sqcup I_{d'}= (I_c \cap P_X) \sqcup (I_c \cap P_Y)=I_c$. Thus, we also get $\phi_{c'}=\phi_c$. 
For $c$ and $c',$ notice that, for an uncut edge $a$ from $X$, obtaining $d(a)=r$ from $c$ amounts to finding the index $r$ such that $c(a)=x_r,$ where
\[x_r=\begin{cases}\sigma(r),&r<k,\\ k+l-1,&r=k,\\ k+l-1+\varsigma(r-k),&r>k.\end{cases}\]
The forward construction assigns the color $x_r$ to this component again. 
The argument for edges from $Y$ is analogous. 
Hence $c=c'$ on $P_{N^{i,j}}\setminus I_c$.
Note that the coloring on $I_c$ is uniquely determined by its values on $P_{N^{i,j}}\setminus I_c,$ together with the colors of the vertex components produced by resolving the edges in $I_c.$ 
The preceding renumbering and shuffle construction also restores the original color of each such vertex component, with the distinguished merged component receiving color $k+l-1.$ 
Therefore $c=c'$ on $I_c$ as well, and hence on all of $P_{N^{i,j}}$.

Similarly, given $m,$ $n,$ $(I_d,\phi_d,d)$, $(I_{d'},\phi_{d'},d')$, $k,$ $l,$ $\sigma,$ and $\varsigma$, we get $(I_c,\phi_c,c)\in \lab_p^{\mathrm{pr}}(N^{i,j})$, and then $m_1,$ $n_1,$ $(I_{d_1},\phi_{d_1},d_1)$, $(I_{d_1'},\phi_{d_1'},d_1')$, $k_1,$ $l_1,$ $\sigma_1,$ and $\varsigma_1$. 
By construction, 
\[I_{d_1}=I_c\cap P_X=I_d,\qquad I_{d_1'}=I_c\cap P_Y=I_{d'}\text{ and}\]
\[\phi_{d_1}=\phi_c|_{I_d}=\phi_d,\qquad \phi_{d_1'}=\phi_c|_{I_{d'}}=\phi_{d'}.\]
Thus the recovered cutting operations produce the original components, including vertex components, so $m_1=m$ and $n_1=n.$ 
The forward construction preserves the relative order of the $X$-components and, separately, of the $Y$-components. 
Hence the inverse renumbering assigns color $r$ to $X_d^r$ and color $s$ to $Y_{d'}^s$ again, including when these components are vertices. 
Since the colorings are constant on the corresponding cutting orbits, we obtain $d_1=d$ and $d_1'=d'.$ 
In particular, the components containing the fixed edges $e_i$ and $e_{i,j}^*$ recover their original colors, giving $k_1=k$ and $l_1=l$. 
Finally, the inverse construction recovers the positions occupied by the individual $X$- and $Y$-components before and after the distinguished merged component. 
These are precisely the positions prescribed by $\sigma$ and $\varsigma,$ respectively, so $\sigma_1=\sigma$ and $\varsigma_1=\varsigma.$ Therefore the reconstructed tuple equals the original tuple.

Finally, since $I_c=I_d\sqcup I_{d'}$ and the relative height orders within $X$ and $Y$ are preserved, we have $c_-=d_-\sqcup(d')_-,$ and hence $\langle N^{i,j},c\rangle = \langle X,d\rangle\langle Y,d'\rangle.$
Resolving the distinguished pair commutes with the other cuts, since the pairs are disjoint. 
Thus the tensor factor of color $k+l-1$ is $[\langle X^k\rangle,\langle Y^l\rangle]_{(e_i,e^*_{i,j})},$
while the remaining factors occupy the positions prescribed by $\sigma,\varsigma$. 
Both summands carry the sign $\varepsilon_i,$ so the corresponding summands agree, proving \eqref{formula:primitivepartcompute}.
 
\vspace{1mm}
\noindent\textbf{Step 5: Show the non‑primitive part  of $\mathcal{J}(\sum{\varepsilon_i}N^{i,j})$ equals 0.}
\vspace{1mm}

By Step~3, it suffices to show $R = 0$ with $R$ being the non-primitive part.
\begin{align}
R \coloneqq \sum_{i,j}\varepsilon_i\sum_{p\geq 1}\sum_{(I_c,\phi_c,c)\in \lab_p^{\mathrm{np}}(N^{i,j})}\langle N^{i,j},c\rangle \langle N^{i,j}\rangle_c^1\otimes\cdots \otimes  \langle N^{i,j}\rangle_c^p = 0.
\end{align}
Here $\langle N^{i,j},c\rangle = (-1)^{|c|_-}\hbar^{|c|/2}$.
By Step 4, we have
$R=\mathcal J(\textstyle \sum_{i,j}\varepsilon_iN^{i,j})-\{\mathcal J(X),\mathcal J(Y)\}\in E(L_\hbar),$
where the inclusion follows from Proposition~\ref{surj} and the closure of $E(L_\hbar)$ under the bracket (Theorem~\ref{symmisomorphism}). 
Thus, by the injectivity of $\operatorname{symm}|_{E(L_\hbar)}$ in Corollary~\ref{E}, it suffices to prove $\operatorname{symm}(R)=0$.
We first rewrite $R$ and for this:
\begin{itemize}[noitemsep,leftmargin=*,nosep]
    \item 
Notice that each $N^{i,j}$ has exactly one component.
By Notation~\ref{0}, we have $|c| = |I_c| = 2(p-1)$.
Let $I_{c}'$ be the adjoint union of $I_{c}\sqcup\{e_i,e^*_{i,j}\}$ and $\phi_{c}'$ the extension of $\phi_c$ by $\phi_c'(e_i) = e_{i,j}^*.$ 
We have $|I_c'| = 2p$.
    \item
Consider pairs $(I,\phi)$, where $I$ is a subset of $T_{X*Y}$ and $\phi$ is an involution of $I$ such that for any $a\in I\cap Q$, we have $\phi(a)$ belonging to $Q^*$ and reversing $a$. 
Let $\mathcal{S}_{X*Y}$ be the set of all pairs $(I,\phi)$ such that there are some $i,j,p$ such that $(I,\phi) = (I_c',\phi_c')$ for some $(I_c,\phi_c,c)\in \lab_p^{\mathrm{np}}(N^{i,j})$.
Given $(I,\phi)\in \mathcal{S}_{X*Y},$ let $r_1<r_2<\cdots<r_n$ be indices such that $\{e_{r_1},\ldots,e_{r_n}\} = I\cap \{e\in I\cap P_X\mid\phi(e)\in P_Y\}$ (cf. $e_i$ in Step 1).
\end{itemize}
Rewrite the non-primitive part $R$ as
\begin{align*}
\sum_{(I,\phi)\in \mathcal{S}_{X*Y}}\sum_{\substack{i=1,\ldots,n}}\varepsilon_{r_i}\sum_{\substack{(I_c,\phi_c,c)\in \lab_{|I|/2}^{0}(N^{r_{i},j})\\(I_c',\phi_c') = (I,\phi)}}(-1)^{|c_-|}\hbar^{|I|/2-1} 
\langle N^{r_{i},j}\rangle_c^1\otimes\cdots \otimes  \langle N^{r_{i},j}\rangle_c^{|I|/2}
\end{align*}
with $j$ being the unique index satisfying $e_{r_i,j}^* = \phi(e_{r_i}).$
Here we also note that $(I_c,\phi_c,c)\in \lab_{|I|/2}(N^{r_{i},j})$ and $(I_c',\phi_c') = (I,\phi)$ implies the coloring $(I_c,\phi_c,c)$ is non-primitive.
Applying the map $\mathrm{symm}$, by linearity we have
\begin{align*}
\sum_{(I,\phi)\in \mathcal{S}_{X*Y}}\sum_{\substack{i=1,\ldots,n}}\varepsilon_{r_i}\sum_{\substack{(I_c,\phi_c,c)\in \lab_{|I|/2}^{0}(N^{r_{i},j})\\(I_c',\phi_c') = (I,\phi)}}(-1)^{|c_-|}\hbar^{|I|/2-1} 
\mathrm{symm}\left(\langle N^{r_{i},j}\rangle_c^1\otimes\cdots \otimes  \langle N^{r_{i},j}\rangle_c^{|I|/2}\right) .
\end{align*}
Notice that the factor $\mathrm{symm}
(\langle N^{r_{i},j}\rangle_c^1\otimes\cdots 
\otimes  \langle N^{r_{i},j}\rangle_c^{|I|/2})$ 
depends only on $(I,\phi)$ and not on the 
heights.
To show the desired equation, we 
fix 
$(I,\phi)\in \mathcal{S}_{XY}$ and show 
\begin{align}
\sum_{\substack{i=1,\ldots,n}}
\sum_{\substack{(I_c,\phi_c,c)\in 
\lab_{|I|/2}^{0}(N^{r_{i},j})\\(I_c^i,\phi_c^i) 
= (I,\phi)}}\varepsilon_{r_i}
(-1)^{|(I,N^{r_i,j})_-|} = 
0.\label{formula:nonprimitivefixedI}
\end{align}
Here we write $|(I,N^{r_i,j})_-|$ for $|c_-|$ 
because $|c_-|$ depends only on $I$ and 
$N^{r_i,j}$ (see the paragraph before 
Definition~\ref{coproduct}).
Note the factor 
$\mathrm{symm}(\langle 
N^{r_{i},j}\rangle_c^1\otimes\cdots 
\otimes  \langle N^{r_{i},j}
\rangle_c^{|I|/2})$ can also be obtained 
from $X*Y$. 
We can rewrite the sign part in 
(\ref{formula:nonprimitivefixedI}).

\begin{lemma}
For fixed $(I,\phi)\in \mathcal{S}_{X*Y}$, 
let $|(I,X*Y)_-|$ be the number of edges in 
$I\cap Q$ that have heights in $X*Y$ larger 
than the corresponding edges via $\phi$ in 
$Q^*$.
We have 
\begin{equation}
\varepsilon_{r_i}(-1)^{|(I,N^{r_i,j})_-|}=(-1)^{i-1}(-1)^{|(I,X*Y)_-|}.
\nonumber
\end{equation}
\end{lemma}
\begin{proof}

Recall that $|(I,N^{r_i,j})_-|$ is the number of edges in $I\cap Q$ that have heights in $N^{r_i,j}$ larger than the corresponding edges via $\phi$ in $Q^*$.
Note the heights of $e_{r_k}$ and $e_{r_k,j}^* = \phi(e_{r_k})$ in $X*Y$ and $N^{r_{i},j}$ are different for $k<i$. 
The sign $\varepsilon_{r_i}(-1)^{|(I,N^{r_i,j})_-|}$ is the one corresponding to (in Defintion~\ref{coproduct}) any coloring of $X*Y$ having the fixed $(I,\phi)$ but with the heights of $e_{r_k}$ and $\phi(e_{r_k})$ interchanged. 
As each of $e_{r_1},\ldots,e_{r_{i-1}}$ produces a $-1$, this sign can be alternatively computed by $(-1)^{i-1}(-1)^{|(I,X*Y)_-|}$.
\end{proof}
By this lemma, the left hand side 
of \eqref{formula:nonprimitivefixedI} writes
\begin{align}
& \sum_{i=1}^{n}\sum_{\substack{(I_c,\phi_c,c)\in \lab_{|I|/2}^{0}(N^{r_{i},j})\\(I_c',\phi_c') = (I,\phi)}}\varepsilon_{r_i}(-1)^{|(I,N^{r_i,j})_-|}\nonumber \\
&=  
(-1)^{|(I,X*Y)_-|}\sum_{i=1}^{n}(-1)^{i-1}\left|\left\{(I_c,\phi_c,c)\in \left.\lab_{|I|/2}^{0}(N^{r_{i},j})\,\right|\, (I_c',\phi_c') = (I,\phi)\right\}\right|.
\label{formula:nonprimitivevanishment}
\end{align}
It suffices to show this equals zero.

We will use a result from Turaev \cite{Turaev}.
Let $\Gamma$ be a finite directed graph. 
For $v(\Gamma)$ the set of vertices of $\Gamma,$ an \emph{order} in a directed graph $\Gamma$ is a bijective map $\omega:v(\Gamma) \to \{1,\ldots,m\}$ for $m$ the cardinal of $v(\Gamma)$ such that $\omega(a) > \omega(b)$ if there is an \emph{arrow} $ab = a\to b$ in $\Gamma$. 
Let $\mathrm{Ord}(\Gamma)$ denote the set of all orders on $\Gamma$.

\begin{lemma}[{Turaev \cite[Lemma 13.10]{Turaev}}]\label{graph}
Let $\Gamma$ be as above.
For an integer $n \geq 2$, let $a_1,\ldots, a_n$ be certain
distinct vertices of $\Gamma$ cyclically connected by 
arrow $a_1a_2,$ $a_2a_3,$$\ldots,$$a_na_1$. 
Let $r_1,\cdots, r_n$ denote the same $n$ arrows numerated arbitrarily. 
For $i=1,\ldots,n$, $\Gamma_i^r$ be
the directed graph obtained from $\Gamma$ by removing the arrow $r_i$ (but keep the corresponding two vertices) and reversing the orientations of 
$r_1,\ldots,r_{i-1}$. If $\Gamma_1^r$ is a tree, then
\[
\sum_{i=1}^n(-1)^i|\mathrm{Ord}(\Gamma_i^r)|=0.
\]
\end{lemma}

The proof of the desired equation, that is,
(\ref{formula:nonprimitivevanishment}) being zero, 
will be accomplished by applying Turaev's lemma to 
the direct graph $\Gamma$ in the lemma below.
We assume that the set
\[C_{(I,\phi),i} \coloneqq \left\{(I_c,\phi_c,c)\in \left.\lab_{|I|/2}^{0}(N^{r_{i},j})\,\right|\, (I_c',\phi_c') = (I,\phi)\right\}\]
is nonempty for some $i$;
otherwise, 
\eqref{formula:nonprimitivevanishment} is automatically zero.

\begin{lemma}\label{lemma:applyinggraph}
Given $(I,\phi)\in\mathcal{S}_{X*Y}$ such that $C_{(I,\phi),i} \neq \varnothing$ for some $i$, there exists a finite directed graph $\Gamma$ satisfying the following conditions.
\begin{enumerate}[label = \rm{(}\arabic*),noitemsep,leftmargin=*,nosep]
\item $\Gamma$ has $|I|/2$ vertices.

\item Some vertices $a_{1},\ldots,a_{n}$ of $\Gamma$ determine cycle $a_{1}a_2,\ldots,a_na_1$ in $\Gamma$, which is moreover the only directed cycle in $\Gamma$.

\item The rest of the directed arrows in $\Gamma$ are determined by for each $i=1,\ldots,n,$ one disjoint tree associated to each vertex $a_i.$

\item For each $i=1,\ldots,n,$ let $\rho_i$ denote the arrow corresponding to the pair $\{e_{r_i},\phi(e_{r_i})\}$ and $\Gamma_i$ the directed tree obtained from $\Gamma$ by deleting the arrow $\rho_i$ while retaining all vertices, and reversing the directions of $\rho_1,\ldots,\rho_{i-1}.$
Then there is a bijection between the set $C_{(I,\phi),i}$ and the set of orders $\omega$ in $\Gamma_i$.
\end{enumerate}
\end{lemma}

Applying Lemma~\ref{graph} with the cycle arrows enumerated as $\rho_1,\ldots,\rho_n,$ Lemma~\ref{lemma:applyinggraph} gives $\sum_{i=1}^{n}(-1)^{i-1}|C_{(I,\phi),i}|=\sum_{i=1}^{n}(-1)^{i-1}|\operatorname{Ord}(\Gamma_i)|=0.$
Thus \eqref{formula:nonprimitivevanishment} vanishes, and consequently $\mathrm{symm}(R)=0.$ Since $R\in E(L_\hbar)$, injectivity yields the desired $R=0.$

\begin{proof}
Since $C_{(I,\phi),i}\ne\varnothing$ for some $i$, resolving the $|I|/2-1$ pairs in $I_c$ on the single knot $N^{r_i,j}$ produces exactly $|I|/2$ components, including vertex components. If even one resolution merged two components, the final number of components would be at most $1+(|I|/2-2)-1=|I|/2-2,$ a contradiction. 
Therefore every resolution splits a component. This holds in any order, since the final cutting orbits depend only on the pairing data.

Let $z_1,\ldots,z_n$ denote a renumbering of $e_{r_{1}},\ldots,e_{r_{n}}$ so that we meet $z_n,\ldots,z_1$ when traversing along $X$ starting at some edge.
By the preceding count, resolving all cross-pairs produces $n$ components. 
Removing the selected edges divides each of $X$ and $Y$ into $n$ paths (possibly empty). 
Since every resulting component alternates between paths from $X$ and $Y$, each must consist of exactly one path from each.
The edges related to $z_{k-1}$ and $z_{k}$ in $X,$ $Y$ are explicit as below, where red edges are from $X$ and blue edges are from $Y$.
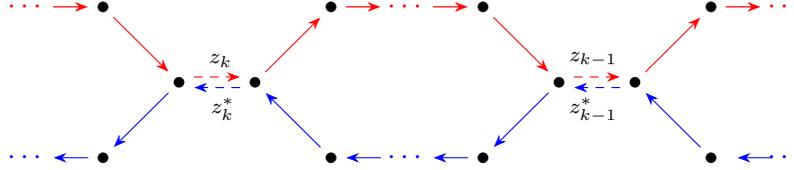
\vspace{-.3cm}
\[\begin{tikzpicture}[>=Stealth, line cap=round, line join=round]

\coordinate (LLLU)  at (-5,  1.0);
\coordinate (LLLD)  at (-5,  -1.0);
\coordinate (LLU)  at (-4,  1);
\coordinate (LLD)  at (-4,  -1);
\coordinate (L)  at (-3,  0);
\coordinate (ML)  at (-2,  0);
\coordinate (MLU)  at (-1, 1.0);
\coordinate (MRU)  at (1,  1.0);
\coordinate (MU)  at (0,  1.0);
\coordinate (MD)  at (0,  -1.0);
\coordinate (MLD)  at (-1, -1.0);
\coordinate (MRD)  at (1,  -1.0);
\coordinate (MR)  at (2, 0);
\coordinate (R)  at (3, 0);
\coordinate (RRU)  at (4,  1);
\coordinate (RRD)  at (4,  -1);
\coordinate (RRRU)  at (5,  1.0);
\coordinate (RRRD)  at (5,  -1.0);

\fill (LLU) circle (2pt);
\fill (LLD) circle (2pt);
\fill (L) circle (2pt);
\fill (ML) circle (2pt);
\fill (MLU) circle (2pt);
\fill (MRU) circle (2pt);
\fill (MLD) circle (2pt);
\fill (MRD) circle (2pt);
\fill (MR) circle (2pt);
\fill (R) circle (2pt);
\fill (RRU) circle (2pt);
\fill (RRD) circle (2pt);

\node at (MU) {\textcolor{red}{$\cdots$}};
\node at (MD) {\textcolor{blue}{$\cdots$}};
\node at (LLLU) {\textcolor{red}{$\cdots$}};
\node at (LLLD) {\textcolor{blue}{$\cdots$}};
\node at (RRRU) {\textcolor{red}{$\cdots$}};
\node at (RRRD) {\textcolor{blue}{$\cdots$}};

\draw[red,shorten >=0.35cm,shorten <=.2cm,->]  (MLU) -- (MU);
\draw[red,shorten >=0.2cm,shorten <=.35cm,->]  (MU) -- (MRU);

\draw[blue,shorten >=0.35cm,shorten <=.2cm,->]  (MRD) -- (MD);
\draw[blue,shorten >=0.2cm,shorten <=.35cm,->]  (MD) -- (MLD);

\draw[dashed,red,shorten >=0.2cm,shorten <=.2cm,->] ([yshift=2pt]L) -- node[pos=0.4, above] {\scriptsize \quad \text{\textcolor{black}{$z_k$}}} ([yshift=2pt]ML);
\draw[dashed,blue,shorten >=0.2cm,shorten <=.2cm,->] ([yshift=-2pt]ML)  -- node[pos=0.6, below] {\scriptsize \quad\,\text{\textcolor{black}{$z_k^*$}}} ([yshift=-2pt]L);

\draw[dashed,red,shorten >=0.2cm,shorten <=.2cm,->] ([yshift=2pt]MR)  -- node[pos=0.6, above] {\scriptsize \text{\textcolor{black}{$z_{k-1}$\quad}}} ([yshift=2pt]R);
\draw[dashed,blue,shorten >=0.2cm,shorten <=.2cm,->]  ([yshift=-2pt]R)-- node[pos=0.4, below] {\scriptsize \text{\textcolor{black}{$z_{k-1}^*$\quad}}}  ([yshift=-2pt]MR);

\draw[red,shorten >=0.2cm,shorten <=.2cm,->]  (ML)  -- node[pos=0.5, below] {\scriptsize } (MLU);
\draw[blue,shorten >=0.2cm,shorten <=.2cm,->]  (MLD)  -- node[pos=0.5, below] {\scriptsize } (ML);

\draw[red,shorten >=0.2cm,shorten <=.2cm,->]  (MRU) -- node[pos=0.5, above] {\scriptsize } (MR);
\draw[blue,shorten >=0.2cm,shorten <=.2cm,->]  (MR) -- node[pos=0.5, above] {\scriptsize } (MRD);


\draw[red,shorten >=.2cm,shorten <=0.35cm,->]  (LLLU) -- (LLU);
\draw[blue,shorten >=.35cm,shorten <=0.2cm,->]  (LLD) -- (LLLD);

\draw[red,shorten >=0.2cm,shorten <=.2cm,->]  (LLU)  -- node[pos=0.5, below] {\scriptsize } (L);
\draw[blue,shorten >=0.2cm,shorten <=.2cm,->]  (L)  -- node[pos=0.5, below] {\scriptsize } (LLD);


\draw[red,shorten >=0.2cm,shorten <=.2cm,->]  (R)  -- node[pos=0.5, below] {\scriptsize } (RRU);
\draw[blue,shorten >=0.2cm,shorten <=.2cm,->]  (RRD)  -- node[pos=0.5, below] {\scriptsize } (R);

\draw[red,shorten >=0.35cm,shorten <=.2cm,->]  (RRU) -- (RRRU);
\draw[blue,shorten >=0.35cm,shorten <=.2cm,->]  (RRRD) -- (RRD);

\end{tikzpicture}\vspace{-.3cm}\]
Let $z_{k-1}z_{k}$ denote the knot given by the middle cycle consisting of solid arrows.
For each $z_{k-1}z_{k}$ we associate a vertex $a_{k-1}$ in the graph.
We have the following directed cycle in (2).
The arrows are $a_{k-1}\to a_{k}$ because the height of $z_{k}^*$ is larger than than of $z_{k}$, where arrows are denoted $z_1z_1^*,\ldots,z_nz_n^*$.
\vspace{-.3cm}
\begin{align}
\begin{split}\begin{tikzpicture}[thick, scale=1]
    \node at (-6,0) (1) {$a_1$};
    \node at (-3,0) (2) {$a_2$};
    \node at (-0,0) (3) {$a_{n-1}$};
    \node at (3,0) (4) {$a_{n}$};
    \draw[->] (1) node[above,xshift=1.5cm] {$z_2z_2^*$} -- (2);
    \draw[dashed] (2) -- (3);
    \draw[->] (3) node[above,xshift=1.7cm] {$z_nz_n^*$} -- (4);
    \draw[->] (3.4,0) .. controls (3.8,0) and (3.8,-.5) .. (3.4,-.5);
    \draw[-] (3.4,-.5) -- (-.9,-0.5);
    \draw (-.9,-.5) -- (-2,-0.5);
    \draw[-] (-2,-.5) -- (-6.4,-0.5);
    \draw[->] (-6.4,-.5) .. controls (-6.8,-.5) and (-6.8,0) .. (-6.4,0);
\end{tikzpicture}\label{formula:graph}
\end{split}
\end{align}\vspace{-.3cm}

To see (3), for the given $(I,\phi)\in \mathcal{S}_{X*Y},$ there is no edge $e$ in $I \cap z_{k-1}z_{k}$ such that $\phi(e)\in z_{l-1}z_{l}$ for any $k\neq l$.
Indeed, if conversely such an edge $e$ exists, the pair of edges $e,$ then following the skein relation (\ref{quiver skein relations 1}), the pair $e$, $\phi(e)$ will result in a single knot in $\na$ given by $z_{k-1}z_{k}$ and $z_{l-1}z_{l}$.
This indicates that $(I,\phi,c)$ can not be in $\lab_p^{0}(N^{i,j})$ for any $i$ and any $c$, a contradiction.

By the preceding count, every subsequent resolution splits one component into two. 
In the associated graph, this replaces one vertex by two vertices joined by one arrow, preserving connectedness and increasing both the number of vertices and the number of arrows by one. 
The direction of this arrow is determined by the heights of $e$, $\phi(e)$.
Recursively, the graph arising within each $z_{k-1}z_k$ is a tree with $a_{k-1}$ being one of its vertex, and these trees are disjoint.

It remains to show (4).
The graph associated with $N^{r_i,j}$ is obtained from $\Gamma$ by deleting $\rho_i$ and reversing the orientations of $\rho_1,\ldots,\rho_{i-1},$ while retaining all vertices and leaving the other arrows unchanged. 
Denote this graph by $\Gamma_i.$
Let $P_{(I,\phi)}$ be the set of components obtained by resolving $(I,\phi)$ in $X*Y$, including vertex components and with components distinguished by their full cutting orbits. 
By the count at the beginning of this proof, $|P_{(I,\phi)}|=|I|/2$.
Hence the construction of $\Gamma_i$ and $\Gamma$ from $X*Y$ gives a bijection $f_i:P_{(I,\phi)}\to v(\Gamma)\to v(\Gamma_i)$.
Let $P_{(I,\phi),i}$ denote set $P_{(I,\phi)}$ with a strict partial order defined as below: 
we have $x>y$ precisely when there is a directed path of positive length from $f_i(x)$ to $f_i(y)$ in $\Gamma_i.$ 
Transitivity follows by concatenating paths, and the absence of directed cycles gives irreflexivity. 
Since $\Gamma_i$ is a tree, this directed path is unique.
Each coloring in $C_{(I,\phi),i}$ assigns its component colors to the full cutting orbits, giving an order-preserving bijection
$P_{(I,\phi),i}\longrightarrow\{1,\ldots,|I|/2\}.$
Conversely, such a bijection determines a coloring by assigning the prescribed color to every member of each corresponding cutting orbit, including those producing vertex components. Compatibility with the arrows of $\Gamma_i$ is precisely the condition imposed by the heights on the paired edges. These constructions are inverse to each other, so the left-hand map below is bijective.
Consequently, we have the following bijections
\[C_{(I,\phi),i} \longrightarrow \mathrm{Ord}(P_{(I,\phi),i}) \longrightarrow \mathrm{Ord}(\Gamma_i).\]
Here $\operatorname{Ord}(P_{(I,\phi),i})$ denotes the set of order-preserving bijections $P_{(I,\phi),i}\to\{1,\ldots,|I|/2\},$ and the right-hand map sends $\omega$ to $\omega\circ f_i^{-1}.$
\end{proof}

\section{The coquantization $
\mathbf{N}(Q)_{h,\hbar} 
\to
V_h(L_h)$}\label{coquancon}

In this section, we introduce the map 
$p_h:\na\to V_h(L_h)$
and prove it is a reduced 
coquantization. 
The main result of this section is 
Theorem 
\ref{coquantization}.

Let $X$ be an 
element in $ \na$ 
and $[X]\in\na/\hbar\na$ its equivalence class. 
By Proposition 
\ref{remark:general form2}~(1), $[X]$ 
has a 
representative given by a 
$k[h]$-linear combination of product 
links in which heights all increase 
along the lexicographical order of the 
indices.
We want a map 
$\mathcal{P}:\na/\hbar\na \to 
V_h(L_h)$ to be linear over $k[h]$ and 
given by 
\begin{align*}
\begin{aligned}
& (a_{1,1},\lambda_{1,1})\cdots (a_{1,l_{1}},\lambda_{1,l_{1}}) 
\&
(a_{2,1},\lambda_{2,1})\cdots (a_{2,l_{2}},\lambda_{2,l_{2}})
\\ &\& \cdots  \& 
(a_{k,1},\lambda_{k,1})\cdots (a_{k,l_{k}},\lambda_{k,l_{k}}) 
\&v_{1} \& v_{2} \&\cdots \& v_{m}
\end{aligned}\mapsto
\begin{aligned}
&(a_{1,1}\cdots a_{1,l_1})\otimes(a_{2,1}\cdots 
a_{2,l_2})\otimes\cdots\\&\otimes(a_{k,1}\cdots 
a_{k,l_k})
\otimes v_1\otimes\cdots\otimes v_m,
\end{aligned}
\end{align*}
where $\lambda_{i,j} < \lambda_{i' ,j'}$ 
whenever ${(i,j)} < {(i' ,j')}$ in 
lexicographical order.

Since $\na$ is obtained as the quotient with
respect to the relations 
(\ref{quiver skein relations 1}) and 
(\ref{quiver skein relations 2}), 
we need to check the well-definedness of 
$\mathcal P$.

\begin{lemma}\label{5.1}
The above map $\mathcal{P}$ is a well-defined algebra morphism.
\end{lemma}
\begin{proof}
It suffices to consider (\ref{quiver skein relations 1}).
By the above expression of $\mathcal{P}$, 
we need to consider $\mathcal{P}([X*Y-Y*X])$ for 
arbitrary knots $X$ and $Y$ both satisfying 
$\lambda_{1,j}<\lambda_{1,j'}$ if $j<j'$.
As in Section~\ref{sec:quantizationpart2}~Step~1, 
by consecutively exchanging the heights in $X$ 
and $Y$ following 
(\ref{quiver skein relations 1}), we may 
get $Y*X$ from $X*Y$.
The equation (\ref{formula:X*Y-Y*X}) can be 
written
as
\[ 
X*Y - Y*X - h\sum_{\substack{1\leq 
i\leq \alpha\\ 1\leq j\leq k_i}} 
\varepsilon_i N^{i,j} = 0.\]
We show that the left side is mapped 
to zero under $\mathcal{P}.$
Note 
\[\mathcal{P}([X*Y-Y*X]) 
= \mathcal{P}([X])\mathcal{P}([Y]) 
- \mathcal{P}([Y])\mathcal{P}([X]) 
= h[\mathcal{P}([X]),\mathcal{P}([Y])],\]
where the bracket on the right side is in fact 
the Lie algebra bracket in $L_h$.
Let $\langle N^{i,j}\rangle$ denote the cyclic 
path underlying $N^{i,j}$ in $L$. 
By interchanging the heights of $N^{i,j}$ so that 
the heights increase from left to right, 
$\langle N^{i,j}\rangle$ is also the image under 
$\mathcal{P}.$
We have
\begin{align}
[\mathcal{P}([X]),\mathcal{P}([Y])]=\sum_{e^*_{i,j}\in 
T_{X,Y}}\varepsilon_{i}\langle N^{i,j}\rangle \label{formula:[alphabeta]}\end{align} because $\langle N^{i,j}\rangle$ is exactly the summand in 
$[\mathcal{P}([X]),\mathcal{P}([Y])]$ given by pairing 
$(e_i,e^*_{i,j})$.
We have shown that
\[\mathcal{P}([X*Y-Y*X]) = h[\mathcal{P}([X]),\mathcal{P}([Y])],\]
as desired.
\end{proof}

Let $p_h:\mathbf{N}(Q)_{h,\hbar}\to V_h(L_h)$ be the composition of maps 
\begin{equation}\label{p_h}
p_h:\mathbf{N(Q)_{h,\hbar}}
\xrightarrow{\textup{projection}}
\na/\hbar\na\overset{\mathcal{P}}{\longrightarrow } 
V_h(L_h).
\end{equation} 
Then it is an algebra morphism 
from $\na $ to $V_h(L_h)$.

\begin{proposition}
\label{coquan}
The algebra morphism $p_h$ 
satisfies $\ker p_h=\hbar \mathbf{N}
(Q)_{h,\hbar}$ and induces \[\na/\hbar 
\na\cong V_h(L_h).\]
\end{proposition}

\begin{proof}
As $p_h$ is an algebra morphism, it suffices to find the preimage of all cyclic paths $\alpha$ as generators of the algebra $V_h(L_h)$.
Consider the following lift $X_\alpha$ of a cyclic path  $\alpha$:
\begin{align}
\textup{if } \alpha=a_1\cdots a_k\in L, 
\textup{ then } X_\alpha=(a_1,1)(a_2,2)\cdots(a_k,k)\in 
\na.\label{formula:Xalphaeta}\end{align}
We have $\mathcal{P}(X_\alpha)=\alpha$ and hence $p_h$ is surjective.

It remains to verify 
$\ker p_h=\hbar \mathbf{N}(Q)_{h,\hbar}$. 
We have   
$\hbar\mathbf{N}(Q)_{h,\hbar}
\subset\ker p_h$ because 
$p_h$ is linear over $k[h,\hbar]\to k[h];\,\,\hbar\mapsto 0$. 
We show $\eta$ is the inverse of $\mathcal{P}$, which shows the desired bijectivity.
For any arbitrary cyclic path $\alpha\in L$, we extend the map $\alpha\to X_\alpha$ in \eqref{formula:Xalphaeta} to be algebra morphism so that we have 
canonical lift $X_\alpha\in\na$. 
This defines a $k[h]$-linear map $\eta:L_h\to \na/\hbar\na.$

We show that $\eta$ may multiplicatively extend to a map $V_h(L_h) \to \na/\hbar\na$.
It suffices to show that for the extended $\eta$, we have 
\begin{align}\eta(\alpha\beta - \beta\alpha - h[\alpha,\beta])=0\text{ for any cyclic paths }\alpha,\beta\in L.\label{formula:etamapstozero}\end{align}
Following Section~\ref{sec:quantizationpart2}~Step~1 with $X = \eta(\alpha)$ and $Y = \eta(\beta)$, we have (the index $\alpha$ is replaced with $n$)
\[\eta(\alpha)*\eta(\beta) - \eta(\beta)*\eta(\alpha) = h\sum_{\substack{1\leq i\leq n\\ 1\leq j\leq k_i}} \varepsilon_i N^{i,j}.\]
By interchanging the heights of $N^{i,j}$ following (\ref{quiver skein relations 2}), we may assume that the heights in $N^{i,j}$ increase from left to right.
Because moreover $\langle N^{i,j}\rangle$ (following the notation in 
(\ref{formula:[alphabeta]})) is exactly the 
summand in $[\mathcal{P}(\eta(\alpha)),
\mathcal{P}(\eta(\beta))] = [\alpha,\beta]$ 
given by pairing $(e_i,e^*_{i,j})$, we have 
\[\eta([\alpha,\beta]) 
= \sum_{\substack{1\leq i\leq n\\ 
1\leq j\leq k_i}} \varepsilon_i N^{i,j}.\]
The equation (\ref{formula:etamapstozero}) holds 
and we have an algebra morphism 
$\eta: V_h(L_h)\to \mathbf{N}
(Q)_{h,\hbar}/\hbar \mathbf{N}(Q)_{h,\hbar}$.
 
For such an
$\eta$, the identities 
$\eta\circ \mathcal{P}=\id$ and 
$\mathcal{P}\circ \eta=\id$ follow 
from the definitions of $\eta$. 
Thus, $\ker p_h=\hbar\mathbf{N}
( Q)_{h,\hbar}$.
\end{proof}

\begin{theorem}
\label{coquantization}
The map $p_h$ is a reduced coquantization of 
the co-Poisson bialgebra $V_h(L_h)$.
\end{theorem}

\begin{proof}

By Proposition~\ref{coquan}, it suffices to check 
that $p_h$ is a coalgebra morphism and satisfies the 
coquantization condition. 
Each element in $\na/\hbar\na$ of the form 
(\ref{form}) has a representative which is a  linear combination of products 
of knots and vertex components.
For every vertex component $v$, we have the vertex case of coproduct preservation:
\[
(p_h\otimes p_h)(\Delta(v))
=v\otimes1+1\otimes v
=\Delta_{V_h(L_h)}(p_h(v)).
\]
Since $p_h$ and the coproducts are algebra morphisms, it therefore remains to verify both conditions on knots.
For a knot $X$, and any 2-coloring $(I,\phi,c)\in \lab_2(X)$, we have either $||c||< 0< |c|/2$, 
$\langle X,c\rangle\Delta(X,c)\in \hbar \mathbf{N}
(Q)_{h,\hbar}^{\otimes 2}$ or $||c||=|c|/2=0$. In the 
latter case, we know that $c$ takes a constant value 
on all arrows. 
By Definition~\ref{coproduct}, we have 
\[\Delta(X)=
X\otimes 1+1\otimes X\mod \hbar\mathbf{N}
(Q)_{h,\hbar}^{\otimes2}.
\]
On the other hand, we have  
\[
(p_h\otimes p_h)
(\Delta(X))=p_h(X)\otimes 1+1\otimes 
p_h(X)=\Delta_{V_h(L_h)}(p_h(X)) 
\] 
and hence $p_h$ is a coalgebra morphism. 

It suffices to prove the following coquantization identity for a knot $X$:
\begin{align}
\Delta(X)-\perm(\Delta(X))=\hbar (p_h\otimes 
p_h)^{-1}(\nu(p_h(X)))\mod \hbar^2\mathbf{N}
(Q)_{h,\hbar}^{\otimes 2}.
\label{formula:coquantizationofVL}\end{align} 
Indeed, by multiplicativity, we compute 
\[\Delta(XY)-\mathrm{Perm}(\Delta(XY))
=\big(\Delta(X)-\mathrm{Perm}(\Delta(X))\big)\Delta(Y)+\mathrm{Perm}(\Delta(X))
\big(\Delta(Y)-\mathrm{Perm}(\Delta(Y))\big).\]
Reducing modulo $\hbar^2$ and using \eqref{copobialge}, we see that \eqref{formula:coquantizationofVL} for $X$ and $Y$ implies \eqref{formula:coquantizationofVL} for $XY.$ 
Equation \eqref{formula:coquantizationofVL} also holds for vertex components, since $\Delta(v)-\mathrm{Perm}(\Delta(v))=0$ and $\nu(p_h(v))=0.$ 
Moreover, because $p_h(\hbar Y)=0,$ \eqref{formula:coquantizationofVL} remains unchanged if $X$ is replaced with $X+\hbar Y$. 
Thus, by Proposition~\ref{formmodh} and linearity, it suffices to verify \eqref{formula:coquantizationofVL} for knots.

For $X$ a knot in $\mathbf{N}(Q)_{h,\hbar}$, we calculate $\Delta(X)\mod\hbar^2\mathbf{N}(Q)_{h,\hbar}^{\otimes 2}$. 
Since $\Delta(X)=X\otimes 1+1\otimes X\mod \hbar\mathbf{N}(Q)_{h,\hbar}$, we need to compute the cases where the power of $\hbar$ equals $1$. 
The power of $\hbar$ being $1$ means  that the corresponding coloring $(I,\phi,c)$ satisfies $|c|/2=||c||+2$ (by Definition~\ref{coproduct}). 
If $|c|=0$, then we have $I=\varnothing$ and  we may assume $|c|\geq 2$. 
Notice $||c||<0.$
Given $|c|/2=||c||+2$, $c$ is not a constant coloring, and so $1\leq |c|/2=||c||+2\leq 1$. 
Hence $|c|/2=1$ and $|I|=2$, i.e.,  there is only a pair of cutting points.
Similar to Section~\ref{sec:quantizationpart2}~Step~1, we write
\[T_{X} = \{e_1,\ldots,e_n\mid e_i\in Q\}\cup \{e_{1,1}^{*},\ldots,e_{n,k_n}^*\mid e_{i,j}^*\in Q^*\}.\]
For each pair $(e_{i},e_{i,j}^*)$, there is a unique coloring $c_{i,j}$, mapping $e_i$ to $2$ if and only if their heights satisfy $\lambda(e_i)>\lambda(e_{i,j}^*).$ 
By Definition~\ref{coproduct}, we have
\[\Delta(X)\equiv X\otimes1+1\otimes 
X+\sum_{e_{i,j}^*\in T_X} \varepsilon_{i,j} \hbar X_{c_{i,j}}^1\otimes 
X_{c_{i,j}}^2 \mod \hbar^2\mathbf{N}
(Q)_{h,\hbar}^{\otimes 2},\]
where $\varepsilon_{i,j}$ equals $1$ if $\lambda(e_i)<\lambda(e_{i,j}^*)$ and equals
$-1$ otherwise. 
Hence 
\begin{align}\Delta(X)-\perm(\Delta(X))\equiv \sum_{e_{i,j}^*\in T_X} \varepsilon_{i,j} \hbar\left(X_{c_{i,j}}^1\otimes X_{c_{i,j}}^2-X_{c_{i,j}}^2\otimes X_{c_{i,j}}^1\right)\mod \hbar^2\mathbf{N}(Q)_{h,\hbar}^{\otimes 2}.\label{formula:deltapermutationdelta}\end{align} 
We need to check that, for the above equation, 
after multiplying $\hbar^{-1}$ and then applying 
$(p_h\otimes p_h)$, the right side equals 
$\nu(p_h(X)) \mod \hbar\na^{\otimes 2}$.
We may assume $\lambda(e_i)>\lambda(e_{i,j}^*)$. 
Then in 
(\ref{formula:deltapermutationdelta}), we have 
that $\varepsilon_{i,j} = -1$ and the knot 
$X_{c_{i,j}}^2$ contains the edge following 
$e_{i,j}^*.$
As for $\nu(p_h(X))$, to apply $p_h$ to $X$, we let $X$ have heights that increase along the lexicographical order of the indices.
By Theorem~\ref{L}~(2), in $\nu(p_h(X))$, the pairing $\langle a_{i},a_{j}\rangle$ writes $\langle e_{i,j}^*,e_{i}\rangle = -1$ and $a_{i+1}\cdots a_{j-1}$ is the cyclic path underlying $X_{c_{i,j}}^2.$
This implies the desired equation and hence
(\ref{formula:coquantizationofVL}) holds.
\end{proof}

\begin{remark}[Comparison with
Schedler's work 
\cite{Schedler}]\label{rem:Schedler}
The above result is due to Schedler 
\cite{Schedler}.
In fact, if we set $h=1$ in our
$\na$ (in particular compared 
\eqref{quiver skein relations 1}
with \cite[(3.3)]{Schedler}), 
then it is exactly his Hopf algebra,
which he denotes by $A$.
\end{remark}

\section{Proof of the main theorem}\label{proofofmain}

Now we prove the main
theorem:

\begin{theorem}[Theorem \ref{main}]
Let $L:=(k \overline{ Q})_{\natural}$
be the necklace Lie bialgebra
of a finite double quiver $\overline Q$.
Then the following diagram 
\begin{equation}\label{diag:thm61}
\begin{tikzcd}
	\mathbf{N}(Q)_{h,\hbar} & {V_h(L_h)} \\
	{\varepsilon_\hbar(L_\hbar)} & S(L)
	\arrow["{p_h}", from=1-1, to=1-2]
	\arrow["{p_\hbar}"', from=1-1, to=2-1]
	\arrow["p", from=1-1, to=2-2]
	\arrow["{q_\hbar}", from=1-2, to=2-2]
	\arrow["{q_h}"', from=2-1, to=2-2]
\end{tikzcd}
\end{equation}
commutes and
gives a biquantization of $L$.
\end{theorem}

In the above diagram, the map $p_\hbar$, and a proof of 
the reduced quantization condition are given in Theorem 
\ref{quantization varepsilon}. The map $p_h$, and a 
proof of the reduced coquantization condition are given 
in Theorem~\ref{coquantization}. 
The maps $q_h$ and $q_\hbar$ are given in 
Corollary~\ref{qh} and Theorem \ref{qhbar}, 
respectively, where their quantization and 
coquantization conditions have also been proved. 

\begin{proof}[Proof of the above theorem]
We are left to prove the commutativity of the diagram. 
For an arbitrary knot $X$ in $\na$, we know that $X$ is 
a lift of a cyclic path, denoted by $\alpha$.
We write $X=X_\alpha$.
By Lemma \ref{map}, we have $p_\hbar(X_\alpha)=\alpha+u$ 
for some $u\in \hbar \varepsilon_\hbar(L)$. 
As $q_h$ maps $\hbar$ to zero, we have 
$q_h(p_\hbar(X_\alpha))=\alpha$.
By the definition of $p_h$, we have 
$q_\hbar(p_h(X_\alpha))=\alpha$.
Thus, the above diagram is commutative for any knots in 
$\na $. 
For an arbitrary link $X\in \na$, by Proposition 
\ref{remark:general form2}, $X=X'+\hbar X''$ where $X'$ 
is a linear combination over $k[h]$ of products of 
knots. Then\[\begin{aligned}
p_\hbar(X)&=p_\hbar(X'+\hbar X'')=p_\hbar(X')+\hbar 
p_\hbar( X''),\\q_h(p_\hbar(X))&=q_h(p_\hbar(X')+\hbar 
p_\hbar( X''))=q_h(p_\hbar(X')),\\
p_h(X)&=p_h(X'+\hbar X'')=p_h(X'),\\
q_\hbar(p_h(X))&=q_\hbar(p_h(X')).
\end{aligned}
\] 
Since $X'$ is a linear combination over $k[h]$ of 
products of knots
and $q_h\circ p_\hbar$
and $q_\hbar\circ p_h$ both
commute with products,
the commutativity of
\eqref{diag:thm61}
holds for arbitrary 
links. By $p=q_h\circ p_\hbar$, we get $\ker p=\hbar\na+h\na$, and this finishes the proof.
\end{proof}

We give a simple example, which might
be helpful for
the interested readers to understand
the above various maps.
This example does not fully illustrate
the complicated argument given in 
Section \ref{quancon}; that said,
we end with a remark explaining 
how such a situation arises.

\begin{example}\label{ex:Jordan}
Let $Q$ be the so-called framed
Jordan quiver whose double $\overline Q$ is 
as follows:
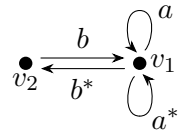
\vspace{-.3cm}
\begin{center}
\begin{tikzpicture}[>=Stealth, line cap=round, line join=round]
    \coordinate (L) at (0,0);
    \coordinate (R) at (1.5,0);
    \coordinate (RU) at (1.5,0.7);
    \coordinate (RD) at (1.5,-.7);

    \fill (L) circle (2.5pt)node[below] {$v_2$};
    \fill (R) circle (2.5pt)node[right] {$v_1$};

        \draw[shorten <=.1cm,-] ([yshift= 3pt]R) to[out=45,in=0] node[above] {\quad$a$} (RU);
    \draw[shorten >=0.05cm,->] (RU) to[out=180,in=135] ([yshift= 3pt]R);
    \draw[shorten <=.1cm,-] ([yshift= -3pt]R) to[out=-135,in=180] (RD);
    \draw[shorten >=0.051cm,->] (RD) to[out=0,in=-45] node[below] {\quad$a^*$} ([yshift= -3pt]R);
    \draw[shorten >=0.2cm,shorten <=.2cm,->]  ([yshift= 2pt]L) -- node[above] {$b$} ([yshift= 2pt]R);
    \draw[shorten >=0.2cm,shorten <=.2cm,->]  ([yshift= -2pt]R) -- node[below] {$b^*$} ([yshift= -2pt]L);
\end{tikzpicture}
\end{center} 
\vspace{-.3cm}
We take two necklaces in $(k \overline{Q})_\natural$: 
$\alpha=aa^*$ and $\beta=b^*ba^*a^*$. Then 
\begin{align}
\{\alpha,\beta\}&=\{aa^*,b^*ba^*a^*\}
=2\langle a,a^*\rangle a^*b^*ba^*=2a^*b^*ba^*,
\label{ex:bracketofaplhabeta}\\
\nu(\alpha)&=\langle a,a^*\rangle v_1 
\wedge v_1=0,\label{ex:cobracketalpha}\\
\nu(\beta)&=\langle b^*,b\rangle 
a^*a^*\wedge v_1=v_1\otimes a^*a^*-a^*a^*\otimes v_1.
\label{ex:cobracktbeta}
\end{align}
Choose two liftings of $\alpha$ and $\beta$ respectively: 
$X=(a,1)(a^*,2),\ Y=(b^*,1)(b,2)(a^*,3)(a^*,4)\in \na$. 
Then by recursively applying
\eqref{quiver skein relations 1},
we get
\begin{align*}
X*Y
&=(a,1)(a^*,2)*(b^*,1)(b,2)(a^*,3)(a^*,4)\\
&=(a,1)(a^*,2)\&(b^*,3)(b,4)(a^*,5)(a^*,6)\\
&=(a,1)(a^*,6)\&(b^*,2)(b,3)(a^*,4)(a^*,5)\\
&=(a,3)(a^*,6)\&(b^*,1)(b,2)(a^*,4)(a^*,5)\\
&=(a,4)(a^*,6)\&(b^*,1)(b,2)(a^*,3)(a^*,5)+h\langle a,a^*\rangle 
(a^*,6)(b^*,1)(b,2)(a^*,5)\\
&=(a,5)(a^*,6)\&(b^*,1)(b,2)(a^*,3)(a^*,4)+2h (a^*,6)(b^*,1)(b,2)(a^*,3)\\
&=Y*X+2h(a^*,6)(b^*,1)(b,2)(a^*,3),
\end{align*}
which is a lifting of
\eqref{ex:bracketofaplhabeta}.

We now compute $q_\hbar(p_h(X))$, $q_h(p_\hbar(X)),$ $q_\hbar(p_h(Y))$, and $q_h(p_\hbar(Y))$.
We first find the non-trivial 
colorings of $X$ 
and $Y$ respectively. 
For a coloring
$(I,\phi,c)\in \lab_2(X)$ with nonempty $I$, we have 
$I=\{a,a^*\}$, $\phi(a)=a^*,\phi(a^*)=a$ 
and $c(a)=1$, $c(a^*)=2$.
For any coloring
$(I',\phi',c')\in \lab_2(Y)$ with nonempty $I'$, 
we have $I'=\{b,b^*\}$ and $\phi'$ to be $\phi'(b)=b^*,\phi'(b^*)=b$. 
By the heights of $b$ and $b^*$, one has to color $b$ with $2$ and $b^*$ with $1$. 
The heights of $X$ are $Y$ are assigned as follows:
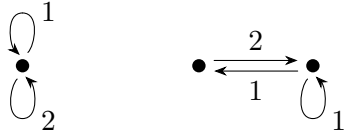
\vspace{-.3cm}
\begin{center}
\begin{tikzpicture}[>=Stealth, line cap=round, line join=round]
    \coordinate (L) at (0,0);
    \coordinate (R) at (1.5,0);
    \coordinate (RU) at (1.5,0.7);
    \coordinate (RD) at (1.5,-.7);

    \fill (R) circle (2.5pt);

    \draw[shorten <=.1cm,-] ([yshift= 3pt]R) to[out=45,in=0] node[above] {\quad 1} (RU);
    \draw[shorten >=0.05cm,->] (RU) to[out=180,in=135] ([yshift= 3pt]R);
    \draw[shorten <=.1cm,-] ([yshift= -3pt]R) to[out=-135,in=180] (RD);
    \draw[shorten >=0.051cm,->] (RD) to[out=0,in=-45] node[below] {\quad$2$} ([yshift= -3pt]R);
\end{tikzpicture}
\quad\quad\quad\quad
\begin{tikzpicture}[>=Stealth, line cap=round, line join=round]
    \coordinate (L) at (0,0);
    \coordinate (R) at (1.5,0);
    \coordinate (RU) at (1.5,0.7);
    \coordinate (RD) at (1.5,-.7);

    \fill (L) circle (2.5pt);
    \fill (R) circle (2.5pt);

    \draw[shorten <=.1cm,-] ([yshift= -3pt]R) to[out=-135,in=180] (RD);
    \draw[shorten >=0.051cm,->] (RD) to[out=0,in=-45] node[below] {\quad$1$} ([yshift= -3pt]R);
    \draw[shorten >=0.2cm,shorten <=.2cm,->]  ([yshift= 2pt]L) -- node[above] {$2$} ([yshift= 2pt]R);
    \draw[shorten >=0.2cm,shorten <=.2cm,->]  ([yshift= -2pt]R) -- node[below] {$1$} ([yshift= -2pt]L);
\end{tikzpicture}
\end{center} 
\vspace{-.3cm}
For $Y$, we label two $a^*$'s by $a_1^*$ 
and $a_2^*$ respectively; that is, we
write
$Y=(b^*,1)(b,2)(a^*_1,3)(a^*_2,4)$,
then the $f$ in Definition \ref{coproduct} writes
\begin{align*}
f:&P_Y\to P_Y,\\ 
&b^*\mapsto a_1^*,\ b \mapsto b,\ a_1^*\mapsto a_2^*,\ \ a_2^*\mapsto b^*. 
\end{align*}
Thus, $P_Y=\{a_1^*,a_2^*, b^*\}\sqcup\{b\} $, listed in 
the descending
order of colorings. 
Then, by the notations in 
Definition \ref{coproduct}, we get 
$Y_{c'}^1=Y_1=(a^*,3)(a^*,4),\ Y_{c'}^2=Y_2=v_2$. We also get 
$|c_-'|=1,|c'|=2$ and $||c'||=-1$.  
A similar procedure applying to $X$ 
gives $X_c^1=X_c^2=v_1$, $|c_-|=0$, $|c|=2$ and 
$||d||=-1$.
Thus, we get the following
\begin{align*}
&\Delta(X)=X\otimes 1+ 1\otimes X+
\hbar v_1\otimes v_1
\\
&\Delta(Y)=1\otimes Y+Y\otimes 1
- \hbar (a^*,3)(a^*,4)\otimes v_2,
\end{align*}
and hence
\begin{align*}
&\hbar^{-1}(\Delta(X)-\perm(\Delta(X)))=0,\\
&\hbar^{-1}(\Delta(Y)-\perm(\Delta(Y))) 
=v_2 
\otimes (a^*,3)(a^*,4)-(a^*,3)(a^*,4)\otimes v_2.
\end{align*}
They are compatible with the above 
cobrackets \eqref{ex:cobracketalpha}
and \eqref{ex:cobracktbeta}.

For reader's convenience, 
we next find the images of $X$ and $Y$ 
under $p_h,\ p_\hbar,\ q_h$ and $q_{\hbar}$. 
By definition,
$$\mathcal{J}
(Y)=\sum_{n\geq0}q^{\otimes n}\circ \Delta^{n-1}
(Y)=b^*ba^*a^*-\hbar a^*a^*\otimes 
v_2+\sum_{n\geq3}q^{\otimes n}\circ \Delta^{n-1}
(Y).$$ 
Since for any vertex $v$, we have 
$\Delta(v)=1\otimes v+v\otimes1$, 
by induction we have 
\begin{align*}
\Delta^{n-1}(Y)=&\sum _{i=0}^{n-1}1^{\otimes 
i}\otimes 
Y\otimes  
1^{\otimes (n-1-i)}\\
&-\hbar\sum_{i+j\leq n-2}1^{\otimes i}\otimes (a^*,3)
(a^*,4) \otimes 1^{\otimes j}\otimes v_2\otimes  
1^{\otimes (n-2-i-j)}
\end{align*}
for any $n\geq 3$. 
Thus by $q(1)=0$, we get 
$\mathcal{J}(Y)=b^*ba^*a^*
-\hbar a^*a^*\otimes v_2$ and 
$p_\hbar(Y)=b^*ba^*a^*-(\hbar/2) a^*a^*v_2.$
We then have $p_h(Y)=b^*ba^*a^*$ 
and hence 
$$q_\hbar(p_h(Y))=q_h(p_\hbar(Y))=b^*ba^*a^*.$$  
Similarly, 
$$\Delta^n(X)=\sum _{i=0}^{n}1^{\otimes i}
\otimes X \otimes  
1^{\otimes (n-i)}+\hbar v_1\otimes 
\sum_{j=0}^{n-1}1^{\otimes j}\otimes v_1\otimes  
1^{\otimes (n-j-1)},$$ for $n\geq 2$, 
and hence 
$\mathcal{J}(X)=aa^*+\hbar v_1\otimes v_1$, 
$p_\hbar(X)=aa^*+(\hbar/2) v_1v_1$. Thus, 
$$q_\hbar(p_h(Y))=q_h(p_\hbar(Y))=aa^*.$$
\end{example}

\begin{remark}[Why is the quantization 
so complicated?]\label{rem:why} 
Our proof of the
quantization in Theorem \ref{quantization varepsilon}
is inspired by
Turaev's technical work (see \cite[\S12-15]{Turaev}).
However, our situation is even more complicated 
in the sense that, 
for loops on surfaces, we may assume
they are in generic position,
but this is not the case for necklaces:
there might be two necklaces that pass
through an edge multiple times, and we must
deal with such a situation carefully (see
\S\ref{sec:quantizationpart2}), which increases
the length and technicality of the paper.
\end{remark}

\end{document}